\documentclass[a4paper,11pt,twoside,reqno]{amsart}

\usepackage[T1]{fontenc}
\usepackage[utf8]{inputenc}
\usepackage{comment}

\usepackage{geometry}
\usepackage{amsmath}
\usepackage{amsthm}
\usepackage{bm}

\usepackage[english]{babel}

\usepackage{enumitem}
\setlist[enumerate,1]{label=(\alph*)}
\setlist[enumerate,2]{label=(\roman*)}
\setlist[enumerate]{font=\normalfont,leftmargin=2.8em}

\usepackage[lowtilde]{url}
\usepackage[unicode,bookmarks,colorlinks,breaklinks]{hyperref}
\hypersetup{linkcolor=blue,citecolor=blue,filecolor=black,urlcolor=blue}
\usepackage[capitalise]{cleveref}

\usepackage{mathtools}
\usepackage{calligra}
\usepackage{stmaryrd}

\usepackage{amssymb}
\usepackage{mathrsfs} 
\usepackage{wrapfig}
\makeatletter
\@namedef{subjclassname@2020}{\textup{2020} Mathematics Subject Classification}
\makeatother

\usepackage{tikz-cd}
\tikzcdset{scale cd/.style={every label/.append style={scale=#1},
    cells={nodes={scale=#1}}}}
\usepackage{rotating}

\usepackage{fixmath} 

\theoremstyle{plain}
\newtheorem{theorem}{Theorem}[section]
\newtheorem{proposition}[theorem]{Proposition}
\newtheorem{lemma}[theorem]{Lemma}
\newtheorem{corollary}[theorem]{Corollary}

\theoremstyle{definition}
\newtheorem{definition}[theorem]{Definition}

\newtheorem*{example*}{Example}
\theoremstyle{remark}
\newtheorem{remark}[theorem]{Remark}
\newtheorem{question}[theorem]{Question}
\newtheorem*{question*}{Question}
\newtheorem{claim}[theorem]{Claim}

\newcounter{char}
\loop\ifnum\value{char}<27
  \expandafter\edef\csname\Alph{char}\Alph{char}\endcsname{\noexpand\mathbb{\Alph{char}}}
  \expandafter\edef\csname\Alph{char}bb\endcsname{\noexpand\mathbb{\Alph{char}}}
  \expandafter\edef\csname\Alph{char}ca\endcsname{\noexpand\mathcal{\Alph{char}}}
  \expandafter\edef\csname\Alph{char}sc\endcsname{\noexpand\mathscr{\Alph{char}}}
  \expandafter\edef\csname\Alph{char}fr\endcsname{\noexpand\mathfrak{\Alph{char}}}
  \expandafter\edef\csname\alph{char}fr\endcsname{\noexpand\mathfrak{\alph{char}}}
  \expandafter\edef\csname\Alph{char}bf\endcsname{\noexpand\mathbf{\Alph{char}}}
  \expandafter\edef\csname\alph{char}bf\endcsname{\noexpand\mathbf{\alph{char}}}
  \expandafter\edef\csname\Alph{char}rm\endcsname{\noexpand\mathrm{\Alph{char}}}
  \expandafter\edef\csname\alph{char}rm\endcsname{\noexpand\mathrm{\alph{char}}}
\addtocounter{char}{1}
\repeat

\loop\ifnum\value{char}<27
  \expandafter\edef\csname b\Alph{char}\endcsname{\noexpand\mathbb{\Alph{char}}}
  \expandafter\edef\csname c\Alph{char}\endcsname{\noexpand\mathcal{\Alph{char}}}
\addtocounter{char}{1}
\repeat

\DeclareMathOperator{\Bl}{Bl}
\DeclareMathOperator{\Ho}{H}

\DeclareMathOperator{\im}{im}

\DeclareMathOperator{\Kahler}{\Omega}
\DeclareMathOperator{\NS}{NS}
\DeclareMathOperator{\Mori}{NE}

\newcommand*{\Category}[1]{\operatorname{\mathbf{#1}}}
\newcommand*{\Db}{\Category{D}^{\mathrm{b}}}

\newcommand*{\tf}{\mathrm{tf}}

\newcommand*{\Lder}{\mathbf{L}}
\newcommand*{\Rder}{\mathbf{R}}

\newcommand*{\coker}{\mathrm{coker}}

\newcommand*{\sHom}{\mathcal{H}om\,}
\newcommand*{\sIsom}{\mathcal{I}som\,}
\newcommand*{\sExt}{\mathcal{E}xt\,}

\newcommand*{\isoarrow}[1]{\arrow[#1,"\rotatebox{90}{\(\sim\)}"
]}

\renewcommand*{\PP}{\mathbf{P}}
\renewcommand*{\AA}{\mathbf{A}}
\newcommand*{\Grass}{\mathbf{Gr}}

\usepackage{xspace}
\usepackage{xpunctuate}

\newcommand*{\cf}{cf.\@\xspace}
\newcommand*{\loccit}{loc.~cit\xperiod}

\title[Higher-dimensional O'Grady's resolutions]{Higher-dimensional O'Grady's resolutions \\ and their exceptional locus}
\author{Luigi Martinelli}

\begin{document}

\hypersetup{
  pdfauthor={L. Martinelli},
  pdftitle={Higher-dimensional O'Grady's resolutions and their exceptional locus}
}

\address{Fakultät für Mathematik, Universität Bielefeld, D-33501, Bielefeld, Germany}
\email{lmartinelli@math.uni-bielefeld.de}




\begin{abstract}
For any $n \ge 3$, we consider the moduli space $M_n$ of semistable sheaves with Mukai vector $2(1,0,1-n)$ on a K3 surface. The moduli space $M_n$ is singular and lacks a crepant resolution, but might still admit a \emph{categorical} crepant one. As a preliminary to exploring this possibility, we study the explicit resolution of singularities of $M_n$ constructed by O'Grady, and provide a global description of its exceptional locus.
\end{abstract}

\maketitle

\section{Introduction}
\subsection*{Motivation}
Hyper-Kähler (HK) manifolds are the higher-dimensional analogues of K3 surfaces. While their geometry is very rich, only few deformation classes of HK manifolds are known, and it seems tremendously hard to provide new ones. The theory of moduli spaces of sheaves on K3 surfaces has proved to be a source of important examples.

\smallskip

Let us fix once and for all a smooth projective complex K3 surface $X$. Pick an element $\mathbf{v}$ in $\Ho_{\mathrm{alg}}^*(X,\mathbb{Z}) \coloneqq \Ho^0(X, \mathbb{Z})\oplus \NS(X) \oplus \Ho^4(X, \mathbb{Z})$, $H$ a $\mathbf{v}$-generic polarisation, and consider the moduli space $M(\mathbf{v})$ of $H$-semistable sheaves on $X$ with Mukai vector $\mathbf{v}$. 

Write $\mathbf{v}=m \mathbf{v}_0$, with $m \in \mathbb{N}$ a multiplicity and $\mathbf{v}_0=(r,c,a)$ primitive in $\Ho_{\mathrm{alg}}^*(X,\mathbb{Z})$. It is convenient to study the properties of  $M(\mathbf{v})$ according to $m$ and $\mathbf{v}_0^2 \coloneqq c^2-2ra$. To avoid technicalities and simplify the exposition, we assume that $r > 0$ and $\mathbf{v}_0^2 \ge 2$. Different situations occur.

\begin{enumerate}
\item[(1)] If $m=1$, then $M(\mathbf{v})$ is smooth; it is a HK manifold, deformation equivalent to the Hilbert scheme of $1+\frac{1}{2}\mathbf{v}_0^2$ points on $X$.

\item[(2)] As soon as $m \ge 2$, $M(\mathbf{v})$ is singular of dimension $2+m^2\mathbf{v}_0^2$. As proved in \cite[Proposition 2.0.3]{O'GdmsK3}, \cite[Th\'eorème 1.1]{LSsO'G} and \cite[Theorem 6.2]{KLSssms}, the singularities of $M(\mathbf{v})$ are symplectic in the sense of Beauville \cite{Bss}; in particular, they are rational Gorenstein \cite[Proposition 1.3]{Bss}. One may wonder whether a crepant resolution $\widetilde{M}(\mathbf{v}) \to M(\mathbf{v})$ exists.

\begin{enumerate}
\item[(2.1)] If $m=2$ and $\mathbf{v}_0^2=2$, $M(\mathbf{v})$ does admit a crepant resolution, and $\widetilde{M}(\mathbf{v})$ yields the OG10-deformation class of HK manifolds (see \cite{O'GdmsK3}, \cite{LSsO'G}).

\item[(2.2)] In any other case, such a resolution does not exist \cite{KLSssms}: no further deformation class of HK manifolds can be obtained by resolving the singularities of $M(\mathbf{v})$.
\end{enumerate}
\end{enumerate}
In case (2.2), however, one may still ask the following 
\begin{question}[{\cf \cite[Question 3.13]{AcHKc}\label{existence of a categorical crepant resolution}}]
Does $M(\mathbf{v})$ admit a \emph{categorical} crepant resolution?
\end{question}
To provide some insight into this question, we need to spend a few words on non-commutative algebraic geometry. Inspired by the successful study of a smooth variety $Y$ through its derived category $\Db(Y)$ of coherent sheaves, non-commutative algebraic geometry suggests that any triangulated category with similar properties should be regarded as (the derived category of) a \emph{non-commutative variety}.

This viewpoint has already been fruitful. For instance, the study of K3 categories - non-commutative counterparts of K3 surfaces - provided new proofs of some classical results on cubic fourfolds. Other applications include the theory of singularities: with respect to classical resolutions, their categorical analogues, defined by Kuznetsov in \cite{Kldcr}, have better minimality properties; also crepant resolutions, suitably reinterpreted, become less rare. Kuznetsov's idea is that, for a resolution $\pi \colon \widetilde{Y} \to Y$ of rational singularities, $\Db(\widetilde{Y})$ is in a sense too big, and should be replaced with a non-commutative variety $\widetilde{\mathscr{D}}\subset \Db(\widetilde{Y})$ that still resolves $Y$. The main theorem of \loccit allows to construct such a $\widetilde{\mathscr{D}}$ when the exceptional locus of $\pi$ is an irreducible divisor $E$, and $\Db(E)$ admits a semiorthogonal decomposition of a special kind. Note that an explicit resolution of $Y$ and a thorough understanding of the geometry of $E$ are both required.

\smallskip

Should a categorical crepant resolution $\widetilde{\mathscr{D}}$ of $M(\mathbf{v})$ in case (2.2) exist, the construction of OG10 would no longer appear sporadic, but part of a more general, although non-commutative, phenomenon. 
Should $\widetilde{\mathscr{D}}$ have, in addition, trivial Serre functor, it would definitely be an example, and motivate a new definition, of \emph{hyper-Kähler category}, the higher-dimensional analogue of a K3 category (the previous tentative definition proposed by Abuaf in \cite{AcHKc} is indeed difficult to check in practice).

\subsection*{O'Grady's examples}
The easiest testing ground for \cref{existence of a categorical crepant resolution} is probably the series of moduli spaces that were explicitly desingularised by O'Grady in \cite{O'GdmsK3I}. For any of them, he even exhibited a resolution whose exceptional locus is an irreducible divisor. The purpose of this paper is to get a good control of the global geometry of this divisor: in the quest of a categorical crepant resolution, this is the necessary preliminary step.

\smallskip

Here is the setting. Fix $n \ge 3$, and consider the primitive vector $\mathbf{v}_0=(1,0,1-n) \in \Ho_{\mathrm{alg}}^*(X,\mathbb{Z})$. Note in passing that $\mathbf{v}_0$ is the Mukai vector of the ideal sheaf $I_Z$ of a length-$n$ closed subscheme $Z \subset X$, so that $M(\mathbf{v}_0) \simeq X^{[n]}$. We focus our attention on the singular moduli space $M_n \coloneqq M(2 \mathbf{v}_0)$. Its singular locus is described as the symmetric square $\mathrm{S}^2 X^{[n]}$, and parametrises polystable sheaves of the form $I_{Z_1} \oplus I_{Z_2}$, for $[Z_1], [Z_2] \in X^{[n]}$.

\begin{wrapfigure}[6]{r}{0cm}
\begin{tikzcd}[scale cd=1.0, ampersand replacement=\&]
\widehat{M}_n \ar[dd,"\widehat{\pi}"'] \ar[rd] \& \\
\& \widetilde{M}_n \ar[ld,"\widetilde{\pi}"] \\
M_n
\end{tikzcd}
\end{wrapfigure}
O'Grady's procedure for desingularising $M_n$ is rather complicated. 

First, performing three blow-ups, he obtains a resolution $\widehat{\pi} \colon \widehat{M}_n \to M_n$, whose exceptional locus is a simple normal crossing divisor with irreducible components $\widehat{\mathrm{\Omega}}$, $\widehat{\mathrm{\Sigma}}$ and $\widehat{\mathrm{\Delta}}$. 

Second, contracting $\widehat{\mathrm{\Delta}}$ and then the image of $\widehat{\mathrm{\Omega}}$, he gets a smooth, projective variety $\widetilde{M}_n$, together with a resolution morphism $\widetilde{\pi} \colon \widetilde{M}_n \to M_n$. The exceptional locus of $\widetilde{\pi}$ is the image $\widetilde{\mathrm{\Sigma}}$ of $\widehat{\mathrm{\Sigma}}$ under the two contractions.

\subsection*{Results}
Our main contribution is an explicit global description of $\widehat{\mathrm{\Sigma}}$, which leads to a better understanding of the geometry of $\widetilde{\mathrm{\Sigma}}$. Before stating the results, we need to introduce some notation.

\smallskip

The singular locus $\mathrm{S}^2 X^{[n]}$ of $M_n$ is realised as the quotient of $X^{[n]} \times X^{[n]}$ by the involution $\iota$ interchanging the two factors. Let $\Xca^{[n]}$ be the blow-up of $X^{[n]} \times X^{[n]}$ along the diagonal $X^{[n]}$. The involution $\iota$ lifts to an involution $\widetilde{\iota}$ of $\Xca^{[n]}$. 
On $\Xca^{[n]}$,  Markman \cite{MBBccc} constructed a vector bundle $V$ of rank $2n-2$, whose pullback by $\widetilde{\iota}$ is isomorphic to $V^\vee$. The incidence correspondence $\mathbb{I} \subset \PP(V) \times _{\Xca^{[n]}} \PP(V^\vee)$ is preserved by the induced action of $\bm{\mu}_2$, the cyclic group of order two. Let $\widehat{\mathbb{I}}$ be the blow-up of $\mathbb{I}$ along the fixed locus $\mathbb{I}^{\bm{\mu}_2}$; the $\bm{\mu}_2$-action on $\mathbb{I}$ lifts to $\widehat{\mathbb{I}}$.

\begin{theorem}\label{thm:global description Sigma hat}
The irreducible divisor $\widehat{\mathrm{\Sigma}} \subset \widehat{M}_n$ is isomorphic to $\widehat{\mathbb{I}}/\bm{\mu}_2$. 
\end{theorem}

As a by-product, we obtain also a global description of $\widehat{\mathrm{\Delta}}$. This is achieved through the study of two $\bm{\mu}_2$-invariant divisors $\widehat{\mathbb{J}}$ and $\widehat{\mathbb{E}}$ on $\widehat{\mathbb{I}}$, whose quotient by $\bm{\mu}_2$ gives $\widehat{\mathrm{\Sigma}} \cap \widehat{\mathrm{\Omega}}$ and $\widehat{\mathrm{\Sigma}} \cap \widehat{\mathrm{\Delta}}$ respectively. For the precise statement, we need further notation. The tangent bundle $\Tca$ of $X^{[n]}$ comes equipped with a symplectic form $\omega$. Let $\alpha \colon \Grass^\omega(2, \Tca) \to X^{[n]}$ be the relative symplectic Grassmannian, $\Aca \subset \alpha^*\Tca$ its tautological rank-2 bundle, and $\Aca^{\perp_\omega}$ the orthogonal complement of $\Aca$ with respect to $\alpha^*\omega$. Set also $\Oca_\alpha(1)\coloneqq \wedge^2 \Aca^\vee$.
\begin{corollary}\label{cor:global description Delta hat}
The irreducible divisor $\widehat{\mathrm{\Delta}} \subset \widehat{M}_n$ is isomorphic to $\PP(\mathrm{S}^2 \Aca) \times_{\Grass^\omega(2, \Tca)} \PP(\Gca)$, where $\Gca$ is an extension of $\Oca$ by $(\Aca^{\perp_\omega}/\Aca) \otimes \Oca_\alpha(1)$ on $\Grass^\omega(2, \Tca)$. 
\end{corollary}

Together with O'Grady's global description of $\widehat{\mathrm{\Omega}}$, these two results complete the analysis of the exceptional locus of $\widehat{\pi}$.

Lastly, we explore the relation between $\widehat{\mathbb{I}}$ and $\widetilde{\mathrm{\Sigma}}$. To this end, we make use of Mori theory, mimicking the computations that led O'Grady from $\widehat{M}_n$ to $\widetilde{M}_n$. Contracting first $\widehat{\mathbb{E}}$, then the image of $\widehat{\mathbb{J}}$ we get a singular projective variety $\widetilde{\mathbb{I}}$. We make explicit the two blow-ups that allow to recover $\widehat{\mathbb{I}}$ from $\widetilde{\mathbb{I}}$, and show that the $\bm{\mu}_2$-action on the former descends to the latter.

\begin{theorem}\label{thm:global description of Sigma tilde}
The exceptional locus $\widetilde{\mathrm{\Sigma}}$ of the resolution $\widetilde{\pi} \colon \widetilde{M}_n \to M_n$ is isomorphic to $\widetilde{\mathbb{I}}/\bm{\mu}_2$.
\end{theorem}

\subsection*{Plan of the paper} 
In \cref{sec:moduli spaces}, we recall the classical construction of the moduli space $M_n$ as the GIT quotient of a certain Quot scheme.

\cref{sec:The desingularisation procedure} details O'Grady's construction of the resolution of singularities $\widehat{\pi} \colon \widehat{M}_n \to M_n$; it is deduced by GIT from three blow-ups performed at the level of the Quot scheme.

For the exceptional divisor of the second of these blow-ups, we provide a global description in \cref{sec:The exceptional divisor of the second blow-up}. This description involves Markman's vector bundle $V$ (whose construction we present in \cref{sec:Markman}), and is crucial for the proof of \cref{thm:global description Sigma hat} and \cref{cor:global description Delta hat}, which is given in \cref{sec:description of Sigma hat and Delta hat}.

In \cref{sec:divisorial contractions of M hat}, we briefly mention how $\widetilde{M}_n$ is obtained from $\widehat{M}_n$, and in \cref{sec: contractions of I hat} we prove \cref{thm:global description of Sigma tilde}.

\subsection*{Notation and conventions} Our work heavily relies on O'Grady's papers \cite{O'GdmsK3} and \cite{O'GdmsK3I}, in particular on the first section of the former, and on the third section of the latter. For the reader's convenience, we adhere as much as possible to the notation adopted therein. 

\smallskip

We work over the field $\mathbb{C}$ of complex numbers. 

\subsubsection{} Let $Y$ be a scheme of finite type over $\mathbb{C}$. 
Given a coherent $\Oca_Y$-module $\Fca$, we write $\Fca^\vee = \sHom(\Fca,\Oca)$ for its dual, and $\mathbb{D}(\Fca) = \Rder \sHom(\Fca,\Oca)$ for its derived dual. If $Y$ is integral, we denote by $\Fca_\tf$ the quotient of $\Fca$ by its torsion subsheaf. If $\Fca$ is finite locally free, the associated projective bundle is 
\[
\PP(\Fca) \coloneqq \mathbf{Proj}(\mathrm{S}^\bullet \Fca^\vee) \xrightarrow{\tau}  Y,
\]
and the corresponding tautological quotient line bundle is denoted $\Oca_\tau(1)$.

\subsubsection{} Let $\sigma\colon W \to Y$ be a morphism of schemes of finite type over $\mathbb{C}$.

\smallskip

Given a scheme $Z$ over $Y$, we sometimes denote the base change $Z \times_Y W$ as $\sigma^*Z$. If $\Sca$ is a quasi-coherent graded $\Oca_Y$-algebra, and $Z =  \mathbf{Spec}(\Sca)$, we set $\PP(Z) \coloneqq \mathbf{Proj}(\Sca)$. In particular, $\sigma^*Z \simeq \mathbf{Spec}(\sigma^* \Sca)$ and $\sigma^* \PP(Z) \simeq \mathbf{Proj}(\sigma^* \Sca)$.

\smallskip

The sheaf of relative differentials of $W$ over $Y$ will be denoted as $\Kahler^1_\sigma$ or $\Kahler^1_{W/Y}$. If $Y = \mathrm{Spec} (\mathbb{C})$, we simply write $\Kahler^1_W$. Suppose, in addition, that $\sigma$ is smooth, so that $\Kahler^1_\sigma$ is locally free. Its dual $\Tca_\sigma \coloneqq (\Kahler^1_\sigma)^\vee$ is the \emph{relative tangent bundle} of $\sigma$. If $Y = \mathrm{Spec} (\mathbb{C})$, we simply write $\Tca_W$, the \emph{tangent bundle} of $W$. The tangent space of $W$ at any closed point $w$ will be denoted $\mathrm{T}_{W,w}$.

\subsubsection{}\label{normal cones} Let $Y$ be a scheme of finite type over $\mathbb{C}$, and $W$ a closed subscheme of $Y$ defined by a sheaf of ideals $\Ica_W \subset \Oca_Y$. The \emph{normal cone} to $W$ in $Y$ is the $W$-scheme
\[
C_{W/Y} \coloneqq \mathbf{Spec} \ ( \oplus_{m \ge 0} \ \Ica_W^{m}/{\Ica_W^{m+1}}).
\]
By \cite[B.6.3]{Fit}, $\PP(C_{W/Y})$ equals the exceptional divisor of the blow-up  $\mathrm{Bl}_W Y$ of $Y$ along $W$.

Assume that the conormal sheaf $\Ica_W/\Ica_W^2$ to $W$ in $Y$ is finite locally free. Its dual is the \emph{normal bundle} to $W$ in $Y$, and will be denoted as $\Nca_{W/Y}$. There is a closed immersion of $C_{W/Y}$ in the total space $N_{W/Y}$ of $\Nca_{W/Y}$; it is defined by the graded sheaf of ideals
\[
\oplus_{m \ge 0} \Jca_m \coloneqq \left[\oplus_{m \ge 0} \mathrm{S}^m (\Ica_W/\Ica_W^2) \twoheadrightarrow \oplus_{m \ge 0} \Ica_W^m/\Ica_W^{m+1} \right].
\]
Note that $\Jca_0$ and $\Jca_1$ vanish. The graded ideal sheaf generated by $\Jca_2$ in $\oplus_{m \ge 0} \mathrm{S}^m (\Ica_W/\Ica_W^2)$ defines a closed subscheme $H_{W/Y}$ of $N_{W/Y}$, called the \emph{Hessian cone} to $W$ in $Y$. We have inclusions
\[
C_{W/Y} \subset H_{W/Y} \subset N_{W/Y}.
\]
\subsubsection{}\label{Notation GIT}
For a quick introduction to group actions and quotients in algebraic geometry, we recommend \cite[Section 4.2]{HLgmss}. We follow essentially the terminology and notation adopted therein. So, if $\mathrm{G}$ is an affine algebraic group acting on a scheme $Y$ of finite type over $\mathbb{C}$, we reserve the notation $Y \sslash \mathrm{G}$ for the good quotient of $Y$ by $\mathrm{G}$  (see \cite[Definition 4.2.2]{HLgmss}), when it exists. In particular, if $Y$ is projective, $\mathrm{G}$ is reductive, the $\mathrm{G}$-action on $Y$ is linearised in an ample line bundle $\Lca$ and $Y^{ss} \coloneqq Y^{ss}(\Lca) \subset Y$ is the corresponding semistable locus, then the notation for the GIT quotient of $Y$ by $\mathrm{G}$ will be $Y^{ss} \sslash \mathrm{G}$ (such a good quotient exists by \cite[Theorem 4.2.10]{HLgmss}).  

However, when dealing with a $\mathrm{G}$-linearisation of a quasi-coherent $\Oca_Y$-module $\Fca$ (see \cite[Definition 4.2.5]{HLgmss}), we prefer to talk of a $\mathrm{G}$-\emph{equivariant structure} on $\Fca$. A quasi-coherent $\Oca_Y$-module equipped with a $\mathrm{G}$-equivariant structure will be called a $\mathrm{G}$-\emph{equivariant sheaf}.

\subsubsection{} \label{sec: notation projections and base change} Throughout the paper, $X$ will be a fixed smooth projective complex K3 surface with very ample line bundle $\Oca(1)$. We will need to consider the fibre product of $X$ with several schemes $Y$ over $\mathbb{C}$. The pullback of $\Oca(1)$ to $Y \times X$ will be denoted $\Oca(1)$ as well; moreover, we shall denote as $p_Y \colon Y \times X \to Y$ the corresponding projection onto the first component. 
Given $\Eca$ and $\Fca$ two coherent $\Oca_{Y \times X}$-modules, we denote as $\sExt_Y^m(\Eca,\Fca)$ the degree-$m$ cohomology sheaf of the complex $\Rder p_{Y,*} \circ \Rder \sHom(\Eca, \Fca)$; if $m=0$, we write also $\sHom_Y(\Eca, \Fca)$. We denote the base change of $\sigma \colon W \to Y$ to $Y \times X$ by $\sigma_X \colon W \times X \to Y \times X$. It is worth noting that, since $p_Y$ is flat, by \cite[Corollary 2.23]{Ket} we have an isomorphism of functors
\[
\Lder \sigma^* \circ \Rder p_{Y,*} \simeq \Rder p_{W,*} \circ \Lder \sigma_X^*.
\]

\subsection*{Acknowledgements} This paper developed from the second chapter of my PhD thesis. I would like to thank my advisor Emanuele Macrì for suggesting the project, and for his constant encouragement and support. I would also like to thank: Riccardo Carini, Alexander Samokhin and Haitao Zou for helpful conversations; Eyal Markman and Kieran O'Grady for answering my questions; Laurent Manivel and Antonio Rapagnetta, who read a preliminary version of this paper; Alexander Kuznetsov for valuable discussions and detailed comments on this manuscript.

\medskip

This research was funded by the European Research Council (ERC) under the European Union’s Horizon 2020 research and innovation programme – ERC-2020-SyG-854361-HyperK, by the Deutsche Forschungsgemeinschaft (DFG, German Research Foundation) under Germany’s Excellence Strategy – EXC-2047/1 – 390685813, and by the Deutsche Forschungsgemeinschaft (DFG, German Research Foundation) – Project-ID 491392403 – TRR 358.

\section{Moduli spaces on a K3 surface}\label{sec:moduli spaces}

In this section, after some generalities on the construction and the local study of moduli spaces of sheaves on a K3 surface, we focus on the series of examples studied by O'Grady in \cite{O'GdmsK3I} and \cite{O'GdmsK3}.

\smallskip

Let $X$ be a smooth projective complex K3 surface. Fix on $X$ a vector $\mathbf{v} \in \Ho^*_{\mathrm{alg}}(X, \mathbb{Z})$ and a $\mathbf{v}$-generic polarisation $H= \Oca(1)$. Then, one can consider the moduli space $M(\mathbf{v}) \coloneqq M_H(X,\mathbf{v})$ of $H$-semistable sheaves on $X$ with Mukai vector $\mathbf{v}$. 

\subsection{The construction}\label{sec:construction of moduli space}
Let us recall some aspects of the construction of $M(\mathbf{v})$; we refer to \cite[Section 4.3]{HLgmss} for more details. 

All coherent sheaves on $X$ with Mukai vector $\mathbf{v}$ have the same Hilbert polynomial $P_\mathbf{v}$ with respect to the polarisation $H$. Moreover, if we choose a sufficiently large integer $k = k(\mathbf{v})$, the following holds: for any $H$-semistable sheaf $F$ on $X$ with Mukai vector $\mathbf{v}$, $F(k)$ is globally generated and $\Ho^m(X,F(k))$ vanishes for any $m \ge 1$. Therefore, $F$ can be realised as a quotient
\begin{equation*}
\Hca \coloneqq \Oca(-k)^{\oplus N} \xrightarrow{q} F,
\end{equation*}
where $N \coloneqq P_\mathbf{v}(k) = \chi(X,F(k))= \mathrm{dim}  \Ho^0(X,F(k))$. This defines a closed point $[q \colon \Hca \to F]$ in the Quot scheme $\mathrm{Quot}_X^{\Hca,P_\mathbf{v}}$.
More precisely, $[q]$ lies in the open subscheme
\[
Q^{ss} \coloneqq \{ [q \colon \Hca \to F] \mid F \text{ is $H$-semistable and } \Ho^0(X,\Hca(k)) \to \Ho^0(X,F(k)) \text{ is an isomorphism} \}.
\]
Denote by $Q^s$ the open subset of $Q^{ss}$ parametrising $H$-stable sheaves, and by $Q$ the closure in $\mathrm{Quot}_X^{\Hca,P_\mathbf{v}}$ of $Q^{ss}$.

On $\mathrm{Quot}_X^{\Hca,P_\mathbf{v}}$ the algebraic group $\mathrm{GL}(N)$ acts on the right by composition, preserving $Q^{ss}$ and $Q$; the centre of $\mathrm{GL}(N)$ acts trivially. The induced $\mathrm{PGL}(N)$-action on $Q$ can be linearised so that $Q^{ss}$ (resp.\ $Q^s$) coincides with the set of $\mathrm{PGL}(N)$-semistable (resp.\ stable) points. Note that the stabiliser of $[q \colon \Hca \to F] \in Q^{ss}$ in $\mathrm{PGL}(N)$ is identified with $\mathrm{PAut}(F)\coloneqq \mathrm{Aut}(F)/\mathbb{C}^\times$. 

The moduli space $M(\mathbf{v})$ is defined as the GIT quotient $Q^{ss} \sslash \mathrm{PGL}(N)$. Its closed points are in one-to-one correspondence with $H$-polystable sheaves on $X$ with Mukai vector $\mathbf{v}$. The quotient of $Q^s$ by $\mathrm{PGL}(N)$ is the open subscheme of $M(\mathbf{v})$ that parametrises $H$-stable sheaves. We denote it as $M(\mathbf{v})^s$.

\subsubsection{Cotangent sheaf of $Q^{ss}$}\label{sec:cotangent sheaf} As it will play a crucial role in this paper, let us recall the global description of the cotangent sheaf of $Q^{ss}$ given in \cite[Theorem 3.1]{LcsQs}. Let
\begin{equation}\label{tautological sequence Quot scheme}
\begin{tikzcd}
0 \ar[r] & \Kca \ar[r] & \Oca(-k)^{\oplus N} \ar[r] & \Fca \ar[r] & 0   
\end{tikzcd}
\end{equation}
be the restriction to $Q^{ss} \times X$ of a tautological short exact sequence on $\mathrm{Quot}_X^{\Hca,P_\mathbf{v}} \times X$. 
Then there is a natural isomorphism of $\mathrm{GL}(N)$-equivariant $\Oca_{Q^{ss}}$-modules
\begin{equation}\label{cotangent sheaf Quot scheme}
\Kahler^1_{Q^{ss}} \xrightarrow{\sim} \sExt_{Q^{ss}}^2(\Fca,\Kca).
\end{equation}
For any morphism $f \colon Y \to Q^{ss}$, we denote the pullback of \eqref{tautological sequence Quot scheme} along $f_X \colon Y \times X \to Q^{ss} \times X$ by
\[
\begin{tikzcd}
0 \ar[r] & \Kca_Y \ar[r] & \Oca(-k)^{\oplus N} \ar[r] & \Fca_Y \ar[r] & 0.  
\end{tikzcd}
\]
This sequence is still exact because $\Fca$ is $Q^{ss}$-flat; moreover, $f^*\Kahler^1_{Q^{ss}} \simeq \sExt_Y^2(\Fca_Y,\Kca_Y)$.

\subsection{Local structure}\label{sec:local structure via deformation theory and Luna's theorem} The local study of $Q^{ss}$ and $M(\mathbf{v})$ can be reduced to a deformation-theoretic problem via Luna's étale slice theorem.

\subsubsection{Deformation theory}\label{sec:deformation theory}
Let $F$ be a polystable sheaf on $X$. Associated with $F$ is a certain deformation functor $\mathrm{Def}_F$. We shall not need its precise definition; the interested reader can find it in \cite[Appendix 2.A.6]{HLgmss} or \cite[Section 3]{ASsmssK3snqv}.
The tangent space of $\mathrm{Def}_F$, which parametrises first-order deformations of $F$, is canonically isomorphic to $\mathrm{Ext}^1(F,F)$. An obstruction space for $\mathrm{Def}_F$ is given by $\mathrm{Ext}^2(F,F)_0 \coloneqq \ker[\mathrm{Tr} \colon \mathrm{Ext}^2(F,F) \to \Ho^2(\Oca_X)]$. The group $\mathrm{Aut}(F)$ acts on the tangent and obstruction spaces by conjugation; in particular, these actions descend to $\mathrm{PAut}(F)$.

One can show that the functor $\mathrm{Def}_F$ admits semiuniversal deformations. Given two such deformations, their base spaces are non-canonically isomorphic. The isomorphism between them is not unique, but unique is the induced tangent map.

\smallskip

As $F$ is a polystable sheaf on a K3 surface,  the parameter space of a semiuniversal deformation for $\mathrm{Def}_F$ admits a particularly simple description. Consider the Yoneda square map
\begin{equation}\label{Yoneda square}
\mathrm{\Upsilon}_F \colon \mathrm{Ext}^1(F,F) \to \mathrm{Ext}^2(F,F)_0, \ \ \ e \mapsto e \cup e.
\end{equation}
It is a $\mathrm{PAut}(F)$-equivariant morphism of affine varieties.
\begin{proposition} \label{prop:formality implies quadraticity}
The formal completion of $\mathrm{\Upsilon}_F^{-1}(0)$ at the origin $o \in \mathrm{Ext}^1(F,F)$ is the parameter space of a semiuniversal deformation for $\mathrm{Def}_F$.
\end{proposition}
\begin{proof}
As explained in \cite[Section 3]{ASsmssK3snqv}, the statement follows from the formality of the differential graded Lie algebra $\Rder \mathrm{Hom}(F,F)$, which was proved in \cite[Theorem 1.1]{BMMfcmsKd0} (see also \cite[Theorem 1.2]{BZfcK3s}).
\end{proof}

\subsubsection{Luna's theorem} Let $F$ be a polystable sheaf on $X$ with Mukai vector $\mathbf{v}$, and $[F] \in M(\mathbf{v})$ the corresponding point of the moduli space. Choose $[q \colon \Hca \to F]$ a lift of $[F]$ in $Q^{ss}$. Since $F$ is polystable, the $\mathrm{PGL}(N)$-orbit of $[q]$ is closed in $Q^{ss}$; the stabiliser of $[q]$ in $\mathrm{PGL}(N)$ is reductive, and isomorphic to $\mathrm{PAut}(F)$. 

Luna's étale slice theorem (\cf \cite{SSaT}) states the existence of a locally closed, affine $\mathrm{PAut}(F)$-invariant subscheme $\Vca \subset Q^{ss}$ containing $[q]$ with étale morphisms
\[
(\Vca \times \mathrm{PGL}(N)) \sslash \mathrm{PAut}(F) \to Q^{ss} \ \ \ \text{and} \ \ \ \Vca \sslash \mathrm{PAut}(F) \to Q^{ss} \sslash \mathrm{PGL}(N) = M(\mathbf{v}).
\]
Such a $\Vca$ will be called an étale slice in $[q]$. Its relation to the deformation theory of $F$ is clarified by the following
\begin{proposition}\cite[Proposition 1.2.3]{O'GdmsK3}\label{prop:semiuniversal deformation}
Let $\widehat{\Vca}$ be the formal completion of $\Vca$ at $[q]$. Let $\Fca_{\widehat{\Vca}}$ be the restriction to $\widehat{\Vca} \times X$ of the tautological quotient sheaf $\Fca$ on $Q^{ss} \times X$. Then $(\widehat{\Vca},\Fca_{\widehat{\Vca}})$ is a semiuniversal deformation for $\mathrm{Def}_F$.
\end{proposition}

\begin{corollary}\label{cor:slice as analytic deformation space}
Let $[q \colon \Hca \to F] \in Q^{ss}$, with $F$ polystable. Let $\Vca \subset Q^{ss}$ be an \'etale slice in $[q]$, and let $\mathrm{\Upsilon}_F$ denote the Yoneda square map \eqref{Yoneda square}. Then there is a $\mathrm{PAut}(F)$-equivariant isomorphism of analytic germs 
\begin{equation}\label{iso analytic germs}
(\Vca,[q]) \simeq (\mathrm{\Upsilon}_F^{-1}(0),o).
\end{equation}
\end{corollary}
\begin{proof}
By \cref{prop:formality implies quadraticity} and \cref{prop:semiuniversal deformation}, the formal neighbourhoods of $\mathrm{\Upsilon}_F^{-1}(0)$ at the origin $o \in \mathrm{Ext}^1(F,F)$ and of $\Vca$ at $[q]$ are each the parameter space of a semiuniversal deformation for $\mathrm{Def}_F$; in particular, they are isomorphic. As explained in \cite[Section 4]{ASsmssK3snqv}, a (formal) isomorphism between them can be chosen to be $\mathrm{PAut}(F)$-equivariant, and even lifted to a $\mathrm{PAut}(F)$-equivariant isomorphism between the analytic germs $(\Vca,[q])$ and $(\mathrm{\Upsilon}_F^{-1}(0),o)$. 
\end{proof}
In particular, under \eqref{iso analytic germs}, points with the same stabiliser in $\mathrm{PAut}(F)$ correspond to each other.
\begin{remark}\label{rem:transitivity of stabilisers}
For a suitable $\mathrm{PAut}(F)$-invariant \emph{analytic} neighbourhood $\Vca'$ of $[q]$ in $\Vca$, the multiplication map $\Vca' \times \mathrm{PGL}(N) \to Q^{ss}$ is a principal $\mathrm{PAut}(F)$-bundle onto an analytic open subset of $Q^{ss}$ (see the last paragraph of \cite[Section 4]{SSaT}). We easily deduce that, for any $[q'] \in \Vca'$, the stabiliser of $[q'] \in \Vca$ in $\mathrm{PAut}(F)$ equals the stabiliser of $[q'] \in Q^{ss}$ in $\mathrm{PGL}(N)$.
\end{remark}
\subsection{O'Grady's examples} 
Fix an integer $n \ge 3$, and consider the primitive vector 
\[
\mathbf{v}_0 \coloneqq (1,0,1-n) \in \Ho^*_{\mathrm{alg}}(X,\mathbb{Z}).
\]
The associated moduli space $M(\mathbf{v}_0)$ can be identified with the Hilbert scheme $X^{[n]}$ of $n$ points on $X$. We shall denote by $\Ica$ the universal ideal sheaf on $X^{[n]} \times X$, and by $\Tca \coloneqq \Tca_{X^{[n]}}$ the tangent bundle of $X^{[n]}$; by \cite[Theorem 10.2.1]{HLgmss}, we have 
\begin{equation*}
\Tca \simeq \sExt^1_{X^{[n]}}(\Ica,\Ica).
\end{equation*} 
Grothendieck-Verdier duality yields the isomorphisms 
\begin{equation*}
\sExt_{X^{[n]}}^2(\Ica, \Ica) \simeq \sHom_{X^{[n]}}(\Ica, \Ica)^\vee \simeq \Oca\quad \text{and} \quad
\sExt_{X^{[n]}}^1(\Ica, \Ica) \simeq \sExt_{X^{[n]}}^1(\Ica, \Ica)^\vee.
\end{equation*}
The second isomorphism defines a symplectic form $\omega$ on $\Tca$. Its restriction to any fibre $\mathrm{Ext}^1(I_Z,I_Z)$ over $[Z] \in X^{[n]}$ will be denoted by $\omega$ as well.

\smallskip
Now, consider the moduli space
\[
M_n \coloneqq M(2 \mathbf{v}_0).
\]
This $(8n-6)$-dimensional variety was studied by O'Grady in \cite{O'GdmsK3I} and \cite{O'GdmsK3}, and will be the main character of this paper. Sticking to the notation of \cref{sec:construction of moduli space}, it is constructed as the GIT quotient of $Q$ by $\mathrm{PGL}(N)$. The strictly semistable locus $Q^{ss} \setminus Q^s$ was determined by O'Grady. 
\begin{lemma}\cite[Lemma 1.1.5]{O'GdmsK3} A point $[q \colon \Hca \to F] \in Q^{ss}$ is strictly semistable if and only if $F$ sits into a short exact sequence
\[
\begin{tikzcd}
0 \ar[r] & I_{Z_1} \ar[r] & F \ar[r] & I_{Z_2} \ar[r] & 0,
\end{tikzcd}
\]
where, for $i=1,2$, $Z_i$ is a length-$n$ subscheme of $X$ and $I_{Z_i}$ is its ideal sheaf.
\end{lemma}
As a consequence, $Q^{ss}\setminus Q^s$ can be written as the disjoint union of
\begin{align*}
\mathrm{\Omega}_Q^\circ & \coloneqq \{ [q \colon \Hca \to F] \mid F \simeq I_Z^{\oplus 2}, [Z] \in X^{[n]} \}, \\
\mathrm{\Gamma}_Q^\circ & \coloneqq \{ [q \colon \Hca \to F] \mid F \ \text{is a non-trivial extension of} \ I_Z \ \text{by} \ I_Z, [Z] \in X^{[n]} \}, \\
\mathrm{\Sigma}_Q^\circ & \coloneqq \{ [q \colon \Hca \to F] \mid F \simeq I_{Z_1} \oplus I_{Z_2}, [Z_1],[Z_2] \in X^{[n]}, Z_1 \neq Z_2 \}, \\
\mathrm{\Lambda}_Q^\circ & \coloneqq \{ [q \colon \Hca \to F] \mid F \ \text{is a non-trivial extension of} \ I_{Z_1} \ \text{by} \ I_{Z_2}, [Z_1],[Z_2] \in X^{[n]}, Z_1 \neq Z_2 \}.
\end{align*}
The stabiliser of $[q \colon \Hca \to F] \in Q^{ss}$ in $\mathrm{PGL}(N)$, computed in \cite[Corollary 1.1.8]{O'GdmsK3}, is 
\begin{equation}\label{stabilisers in PGL(N)}
\mathrm{St}([q]) \simeq \begin{cases} \mathrm{PGL}(2)  &\text{if} \  [q] \in \mathrm{\Omega}_Q^\circ \\
(\mathbb{C}, +) &\text{if} \ [q] \in \mathrm{\Gamma}_Q^\circ \\
\mathbb{C}^\times &\text{if} \ [q] \in \mathrm{\Sigma}_Q^\circ \\
\{1\} &\text{if} \ [q] \in \mathrm{\Lambda}_Q^\circ \sqcup Q^s.
\end{cases}
\end{equation}
 
Let $\mathrm{\Omega}_Q$ and $\mathrm{\Sigma}_Q$ be the Zariski closures of $\mathrm{\Omega}_Q^\circ$ and of $\mathrm{\Sigma}_Q^\circ$ in $Q^{ss}$, respectively. As the dimension of the stabilisers is an upper semicontinuous function on $Q^{ss}$, we have 
\begin{equation}\label{closures of OmegaQ and SigmaQ}
\mathrm{\Omega}_Q = \mathrm{\Omega}_Q^\circ \quad \text{and} \quad \mathrm{\Sigma}_Q^\circ \subset \mathrm{\Sigma}_Q \subset \mathrm{\Omega}_Q^\circ \sqcup \mathrm{\Sigma}_Q^\circ \sqcup \mathrm{\Gamma}_Q^\circ.
\end{equation}
Let us describe the local structure of $Q^{ss}$ at points of $\mathrm{\Sigma}_Q^\circ$  and $\mathrm{\Omega}_Q^\circ$. We use the tools of \cref{sec:local structure via deformation theory and Luna's theorem}, which were computed in \cite[Sections 1.4 and 1.5]{O'GdmsK3}.

\smallskip

\subsubsection{Points of $\mathrm{\Sigma}_Q^\circ$}\label{Points in SigmaQ} Let $[q \colon \Hca \to F]$ be a point of $\mathrm{\Sigma}_Q^\circ$. The sheaf $F$ is isomorphic to $I_{Z_1} \oplus I_{Z_2}$, for some \emph{distinct} length-$n$ closed subschemes $Z_1$ and $Z_2$ of $X$. The projective automorphism group $\mathrm{PAut}(F)$ is isomorphic to $\mathbb{C}^\times$.

\smallskip
The tangent and obstruction spaces of $\mathrm{Def}_F$ are
\begin{gather*}
\mathrm{Ext}^1(F,F) \simeq \mathrm{Ext}^1(I_{Z_1},I_{Z_2}) \oplus \mathrm{Ext}^1(I_{Z_2},I_{Z_1}) \oplus \mathrm{Ext}^1(I_{Z_1},I_{Z_1}) \oplus \mathrm{Ext}^1(I_{Z_2},I_{Z_2}), \\
\mathrm{Ext}^2(F,F)_0 = \{ (\alpha,\beta) \in  \mathrm{Ext}^2(I_{Z_1},I_{Z_1}) \oplus \mathrm{Ext}^2(I_{Z_2},I_{Z_2}) \mid \mathrm{Tr}(\alpha) + \mathrm{Tr}(\beta) = 0 \}.
\end{gather*}
The group $\mathrm{PAut}(F)$ acts trivially on $\mathrm{Ext}^2(F,F)_0$. On $\mathrm{Ext}^1(F,F)$, it acts as follows: given
\[
e_{ij} \in \mathrm{Ext}^1(I_{Z_i},I_{Z_j}), \quad 1 \le i,j \le 2,
\]
and for an appropriate choice of isomorphism $\mathrm{PAut}(F) \simeq \mathbb{C}^\times$, we have
\[
\lambda(e_{12},e_{21},e_{11},e_{22}) = (\lambda e_{12}, \lambda^{-1} e_{21},e_{11},e_{22}), \quad \lambda \in \mathbb{C}^\times.
\]
In particular, the stabiliser of $(e_{12},e_{21},e_{11},e_{22}) \in \mathrm{Ext}^1(F,F)$ in $\mathrm{PAut}(F)$ is
\begin{equation}\label{stabilisers in C*}
\mathrm{St}(e_{12},e_{21},e_{11},e_{22}) = \begin{cases} \mathbb{C}^\times &\text{if } e_{12} = e_{21} = 0 \\
\{ 1 \} & \text{otherwise.}
\end{cases}
\end{equation}
Moreover, the Yoneda square map $\mathrm{\Upsilon}_F$ reads
\[
\mathrm{\Upsilon}_F(e_{12},e_{21},e_{11},e_{22}) = (e_{21} \cup e_{12}, e_{12} \cup e_{21}).
\]

\smallskip

Now, pick $\Vca \subset Q^{ss}$ an étale slice in $[q]$. By \cref{cor:slice as analytic deformation space}, there is a $\mathrm{PAut}(F)$-equivariant isomorphism of analytic germs
\begin{equation*}
(\Vca,[q]) \simeq (\mathrm{\Upsilon}_F^{-1}(0),o).
\end{equation*}
The points on both sides can be compared by means of their stabilisers in $\mathrm{PAut}(F)$, given in \eqref{stabilisers in PGL(N)} via \cref{rem:transitivity of stabilisers}, and in \eqref{stabilisers in C*}, respectively. We deduce from \eqref{closures of OmegaQ and SigmaQ} an isomorphism of analytic germs\footnote{The analytic germs that we claim to be isomorphic are reduced, because they are associated to reduced schemes: if we endow a closed subset $Y$ of $Q^{ss}$ with the reduced induced scheme structure, then $\Wca \coloneqq \Vca \times_{Q^{ss}}Y$ is reduced, too. Indeed, as the morphism $\Vca \times \mathrm{PGL}(N) \to Q^{ss}$ is smooth, so is $\Wca \times \mathrm{PGL}(N) \to Y$; since $Y$ is reduced, so are $\Wca \times \mathrm{PGL}(N)$ \cite[IV, Proposition 17.5.7]{EGA} and, in particular, $\Wca$.}
\[
(\mathrm{\Sigma}_Q \cap \Vca, [q]) \simeq (\{0 \} \times \mathrm{Ext}^1(I_{Z_1},I_{Z_1}) \times \mathrm{Ext}^1(I_{Z_2},I_{Z_2}),o).
\]
Thus, the fibre over $[q]$ of the normal cone $C_{\mathrm{\Sigma}_Q \cap \Vca/\Vca}$ to $\mathrm{\Sigma_Q} \cap \Vca$ in $\Vca$ (and, by Luna's theorem, also the fibre over $[q]$ of the normal cone $C_{\mathrm{\Sigma}_Q /Q^{ss}}$ to $\mathrm{\Sigma_Q}$ in $Q^{ss}$) can be expressed as
\begin{equation}\label{fibres normal cone SigmaQ}
(C_{\mathrm{\Sigma_Q} /Q^{ss}})_{[q]} \simeq (C_{\mathrm{\Sigma_Q} \cap \Vca/\Vca})_{[q]} \simeq \{ (e_{12}, e_{21}) \in \mathrm{Ext}^1(I_{Z_1},I_{Z_2}) \oplus \mathrm{Ext}^1(I_{Z_2},I_{Z_1}) \mid e_{12} \cup e_{21} = 0 \}.
\end{equation}
The $\mathrm{PAut}(F)$-action on $(C_{\mathrm{\Sigma_Q}/Q^{ss}})_{[q]}$ is given by $\lambda (e_{12},e_{21}) = (\lambda e_{12}, \lambda^{-1} e_{21})$. The corresponding GIT quotient is the affine cone over
\[
\{ ([e_{12}], [e_{21}]) \in \PP \mathrm{Ext}^1(I_{Z_1},I_{Z_2}) \times \PP \mathrm{Ext}^1(I_{Z_2},I_{Z_1}) \mid e_{12} \cup e_{21} = 0 \},
\]
regarded as a closed subscheme of $\PP (\mathrm{Ext}^1(I_{Z_1},I_{Z_2}) \otimes \mathrm{Ext}^1(I_{Z_2},I_{Z_1}))$ via the Segre imbedding.

\subsubsection{Points of $\mathrm{\Omega}_Q^\circ$}\label{sec:Points in in OmegaQ} Let $[q \colon \Hca \to F]$ be a point of $\mathrm{\Omega}_Q^\circ$. The sheaf $F$ is isomorphic to $I_Z^{\oplus 2}$, for some closed subscheme $Z$ of $X$ of length $n$. The projective automorphism group  $\mathrm{PAut}(F)$, which is isomorphic to $\mathrm{PGL}(2)$, admits also the following useful description. Let $W$ be the Lie algebra of $\mathrm{PGL}(2)$. The adjoint action of $\mathrm{PGL}(2)$ on $W \simeq \mathfrak{sl}(2)$ preserves the Killing form. We deduce a group homomorphism $\mathrm{PGL}(2) \to \mathrm{SO}(W)$, which turns out to be an isomorphism. 

\smallskip

Set
\[
\mathrm{E}_Z \coloneqq \mathrm{Ext}^1(I_Z,I_Z).
\]
The tangent and obstruction spaces of $\mathrm{Def}_F$ are
\begin{gather*}
\mathrm{Ext}^1(F,F) \simeq \mathrm{E}_Z \otimes \mathfrak{gl}(2) \simeq (\mathrm{E}_Z \otimes W) \oplus (\mathrm{E}_Z \otimes \mathbb{C}\mathrm{Id}), \\
\mathrm{Ext}^2(F,F)_0 \simeq \mathrm{Ext}^2(I_Z,I_Z) \otimes W \simeq W.
\end{gather*}
The action of $\mathrm{PAut}(F) \simeq \mathrm{SO}(W)$ on these spaces is deduced from the tautological one on $W$, and the trivial one on $\mathbb{C}\mathrm{Id}$. 

Now, the Killing form $\kappa$ (resp.\ the Lie bracket) of $W$ yields identifications
\begin{equation*}
\mathrm{E}_Z \otimes W \simeq \mathrm{Hom}(W,\mathrm{E}_Z) \quad \text{and} \quad \wedge^2 W^\vee \simeq \wedge^2 W  \quad (\text{resp.}\wedge^2 W \xrightarrow{\sim} W).
\end{equation*}
Up to a non-zero multiplicative factor, the Yoneda square map $\mathrm{\Upsilon}_F$ for $F$ reads accordingly
\[
\mathrm{\Upsilon}_F \colon \mathrm{Hom}(W,\mathrm{E}_Z) \times (\mathrm{E}_Z \otimes \mathbb{C}\mathrm{Id}) \to \wedge^2 W^\vee, \quad (\varphi, e) \mapsto \varphi^* \omega
\]
(recall that $\omega \in \wedge^2 \mathrm{E}_Z^\vee$ denotes the symplectic form on $\mathrm{E}_Z$). Thus, setting
\begin{equation*}
\mathrm{Hom}^{\omega}(W,\mathrm{E}_Z) \coloneqq \{ \varphi \in \mathrm{Hom}(W, \mathrm{E}_Z) \mid \im(\varphi) \ \text{is isotropic for} \ \omega \},
\end{equation*}
the zero locus of $\mathrm{\Upsilon}_F$ is
\[
\mathrm{\Upsilon}_F^{-1}(0) = \mathrm{Hom}^\omega(W,\mathrm{E}_Z)  \times (\mathrm{E}_Z \otimes \mathbb{C}\mathrm{Id}).
\]
The stabiliser in $\mathrm{PAut}(F) \simeq \mathrm{SO}(W)$ of any pair $(\varphi, e) \in \mathrm{Hom}^\omega(W,\mathrm{E}_Z)  \times (\mathrm{E}_Z \otimes \mathbb{C}\mathrm{Id})$ equals the stabiliser of its first component $\varphi$, which is easily computed as
\begin{equation}\label{stabilisers in PGL(2)}
\mathrm{St}(\varphi) \simeq \begin{cases} \mathrm{SO}(W) &\text{if }  \varphi = 0 \\
(\mathbb{C}, +) &\text{if } \mathrm{rk}(\varphi) = 1  \text{ and }\ker(\varphi)^{\perp_\kappa} \ \text{is isotropic for } \kappa \\
\mathbb{C}^\times &\text{if } \mathrm{rk}(\varphi) = 1  \text{ and }\ker(\varphi)^{\perp_\kappa} \ \text{is non-isotropic for } \kappa \\
\{1\} &\text{if } \mathrm{rk}(\varphi) \ge 2.
\end{cases}
\end{equation}
Note that the analytic (and even Zariski) closure in $\mathrm{Hom}^\omega(W,\mathrm{E}_Z)$ of the set of homomorphisms with stabiliser isomorphic to $\mathbb{C}^\times$ equals
\begin{equation*}
\mathrm{Hom}_1(W, \mathrm{E}_Z) \coloneqq  \{ \varphi \in \mathrm{Hom}(W, \mathrm{E}_Z) \mid \mathrm{rk}(\varphi) \le 1 \}.
\end{equation*}

Now, let $\Vca \subset Q^{ss}$ be an étale slice in $[q]$. By \cref{cor:slice as analytic deformation space}, there is a $\mathrm{PAut}(F)$-equivariant isomorphism of analytic germs
\begin{equation} \label{Q local description}
(\Vca, [q]) \simeq (\mathrm{Hom}^{\omega}(W,\mathrm{E}_Z) \times (\mathrm{E}_Z \otimes \mathbb{C}\mathrm{Id}),o).
\end{equation}
The points on both sides can be compared by means of their stabilisers in $\mathrm{PAut}(F)$, given in \eqref{stabilisers in PGL(N)} via \cref{rem:transitivity of stabilisers}, and in \eqref{stabilisers in PGL(2)}, respectively; in particular, in an analytic neighbourhood of $[q]$, the closure of $\Vca \cap \mathrm{\Sigma}_Q^\circ$ in $\Vca$ is described as $\Vca \cap (\mathrm{\Omega}_Q^\circ \sqcup \mathrm{\Sigma}_Q^\circ \sqcup \mathrm{\Gamma}_Q^\circ)$. 
From \eqref{closures of OmegaQ and SigmaQ} we deduce that \eqref{Q local description} restricts to isomorphisms of analytic germs
\begin{gather}
(\mathrm{\Sigma}_Q \cap \Vca, [q]) \simeq (\mathrm{Hom}_1(W,\mathrm{E}_Z) \times (\mathrm{E}_Z \otimes \mathbb{C}\mathrm{Id}),o),  \label{SigmaQ local description} \\
(\mathrm{\Omega}_Q \cap \Vca,[q]) \simeq (\{ 0 \} \times (\mathrm{E}_Z \otimes \mathbb{C}\mathrm{Id}),o). \label{OmegaQ local description}
\end{gather}
This allows us to describe the fibre over $[q]$ of the normal cones $C_{\mathrm{\Omega}_Q \cap \Vca/\Vca}$ to $\mathrm{\Omega}_Q \cap \Vca$ in $\Vca$  and $C_{\mathrm{\Omega}_Q \cap \Vca/\mathrm{\Sigma}_Q \cap \Vca}$ to $\mathrm{\Omega}_Q \cap \Vca$ in $\mathrm{\Sigma}_Q \cap \Vca$ (and, by Luna's theorem, also the fibre over $[q]$ of the normal cones $C_{\mathrm{\Omega}_Q /Q^{ss}}$ to $\mathrm{\Omega}_Q$ in $Q^{ss}$ and $C_{\mathrm{\Omega}_Q/\mathrm{\Sigma}_Q}$ to $\mathrm{\Omega}_Q$ in $\mathrm{\Sigma}_Q$): we have $\mathrm{PAut}(F)$-equivariant isomorphisms
\begin{gather}
(C_{\mathrm{\Omega}_Q/Q^{ss}})_{[q]} \simeq (C_{\mathrm{\Omega}_Q \cap \Vca/\Vca})_{[q]} \simeq \mathrm{Hom}^{\omega}(W, \mathrm{E}_Z), \label{normal cone OmegaQ fibres} \\
(C_{\mathrm{\Omega}_Q/\mathrm{\Sigma}_Q})_{[q]} \simeq (C_{\mathrm{\Omega}_Q \cap \Vca/\mathrm{\Sigma}_Q \cap \Vca})_{[q]} \simeq \mathrm{Hom}_1(W,\mathrm{E}_Z). \label{fibres cone to OmegaQ in SigmaQ}
\end{gather}

\subsubsection{} As it is useful for the study of $Q^{ss}$ along $\mathrm{\Omega}_Q$ or $\mathrm{\Sigma}_Q$, and of $\mathrm{\Sigma}_Q$ along $\mathrm{\Omega}_Q$, we recall here the notion of normal flatness; we refer to \cite[Chapter IV, Section 21]{HOIebu} for more details.
\begin{definition}[Hironaka]
Let $Y$ be a scheme of finite type over $\mathbb{C}$, and $V$ a closed subscheme of $Y$ defined by a sheaf of ideals $\Jca \subset \Oca_Y$. Given a point  $y$ of $V$, we say that $Y$ is \emph{normally flat along $V$ at $y$} if the algebra $(\oplus_{m \ge 0} \Jca^m/\Jca^{m+1})_y$ is a flat $\Oca_{V,y}$-module. We say that $Y$ is \emph{normally flat along $V$} if it is so at any point of $V$ (or, equivalently, at any \emph{closed} point of $V$). 
\end{definition}
\begin{lemma}\label{lem:normal flatness} Let $Y$ be a scheme of finite type over $\mathbb{C}$, $V$ a closed subscheme of $Y$ and $y$ a closed point of $V$. 
\begin{enumerate}
\item[(i)] Let $f \colon Y' \to Y$ be a flat morphism. Consider in $Y'$ the closed subscheme $V' \coloneqq Y' \times_Y V$, and in $V'$ a closed point $y'$ whose image under $f$ is $y$. Then $Y$ is normally flat along $V$ at $y$ if and only if $Y'$ is normally flat along $V'$ at $y'$.
\item[(ii)] Assume that $V$ is smooth at $y$. Then $Y$ is normally flat along $V$ at $y$ if and only if the natural morphism 
\[
(C_{V/Y})_y \times C_{y/V} \to C_{y/Y} 
\]
is an isomorphism.
\end{enumerate}
\end{lemma}
\begin{proof}
The first statement is easily verified. The second is \cite[Chapter IV, Corollary 21.11]{HOIebu}, phrased in geometric terms.
\end{proof}
\begin{proposition}\label{normal flatness of Q} The following hold true.
\begin{enumerate}
\item[(i)] $Q^{ss}$ is normally flat along $\mathrm{\Omega}_Q=\mathrm{\Omega}_Q^\circ$. 
\item[(ii)] $\mathrm{\Sigma}_Q$ is normally flat along $\mathrm{\Omega}_Q=\mathrm{\Omega}_Q^\circ$.
\item[(iii)] $Q^{ss}$ is normally flat along $\mathrm{\Sigma}_Q$ at any point of $\mathrm{\Sigma}_Q^\circ$. 
\end{enumerate}
In particular, the natural morphisms $\PP(C_{\mathrm{\Omega_Q}/Q^{ss}}) \to \mathrm{\Omega}_Q$ and $\PP(C_{\mathrm{\Omega_Q}/\mathrm{\Sigma}_Q}) \to \mathrm{\Omega}_Q$ are flat.
\end{proposition}
\begin{proof}
As all these statements are proved similarly, we deal just with the first one. Pick a closed point of $\mathrm{\Omega}_Q^\circ= \mathrm{\Omega}_Q$; in the notation of \cref{sec:Points in in OmegaQ}, which we keep in this proof, it is of the form $[q \colon \Hca \to I_Z^{\oplus 2}]$, with $[Z] \in X^{[n]}$. Let $\Vca \subset Q^{ss}$ be an \'etale slice in $[q]$, and set $\Wca \coloneqq \Vca \times_{Q^{ss}} \mathrm{\Omega}_Q$. Recalling \eqref{OmegaQ local description}, we have a commutative diagram with cartesian squares and flat horizontal arrows:
\[
\begin{tikzcd}[column sep= 0.7em]
\{ 0 \} \ar[d,hook] & \{ 0 \} \times (\mathrm{E}_Z \otimes \mathbb{C}\mathrm{Id}) \ar[l] \ar[d,hook] & \mathrm{Spec}(\widehat{\Oca}_{\Wca, [q]}) \ar[l] \ar[r] \ar[d,hook] & \Wca \ar[d,hook]  & \Wca \times \mathrm{PGL}(N) \ar[l] \ar[r] \ar[d,hook] & \mathrm{\Omega}_Q \ar[d,hook] \\
\mathrm{Hom}^{\omega}(W, \mathrm{E}_Z) & \mathrm{Hom}^{\omega}(W, \mathrm{E}_Z) \times (\mathrm{E}_Z \otimes \mathbb{C}\mathrm{Id}) \ar[l] & \mathrm{Spec}(\widehat{\Oca}_{\Vca, [q]}) \ar[l] \ar[r] & \Vca & \Vca \times \mathrm{PGL}(N) \ar[l] \ar[r] & Q^{ss}.
\end{tikzcd}
\]
By \cref{lem:normal flatness}(i), the result follows from the normal flatness of $\mathrm{Hom}^{\omega}(W, \mathrm{E}_Z)$ along $\{ 0 \}$.
\end{proof}

\subsubsection{The singular locus of $M_n$} 
From the stratification $\mathrm{\Omega}_Q \subset \mathrm{\Sigma}_Q \subset Q^{ss}$ at the level of the Quot scheme, we deduce by GIT quotient a stratification of the moduli space:
\begin{equation}\label{stratification moduli space}
\mathrm{\Omega}_Q \sslash \mathrm{PGL}(N) \quad \subset \quad \mathrm{\Sigma}_Q \sslash \mathrm{PGL}(N) \quad \subset \quad Q^{ss} \sslash \mathrm{PGL}(N) = M_n.
\end{equation}
The first term parametrises polystable sheaves of the form $I_Z^{\oplus 2}$, with $[Z] \in X^{[n]}$; therefore, $\mathrm{\Omega}_Q \sslash \mathrm{PGL}(N) \simeq X^{[n]}$.

The middle term parametrises polystable sheaves of the form $I_{Z_1} \oplus I_{Z_2}$, with $[Z_1],[Z_2] \in X^{[n]}$; thus, it equals the strictly semistable locus $M_n \setminus M_n^s$, as well as the image of the natural morphism
\[
X^{[n]} \times X^{[n]} \to M_n, \ \ \ ([Z_1], [Z_2]) \mapsto [I_{Z_1} \oplus I_{Z_2}].
\]
This morphism is invariant for the involution $\iota$ interchanging the factors of $X^{[n]} \times X^{[n]}$, and identifies $M_n \setminus M_n^s  = \mathrm{\Sigma}_Q \sslash \mathrm{PGL}(N)$ with the symmetric square $\mathrm{S}^2 X^{[n]} = (X^{[n]} \times X^{[n]}) / \iota$.

\smallskip

The singular locus of $M_n$ equals $M_n \setminus M_n^s \simeq \mathrm{S}^2 X^{[n]}$ (cf.\ \cite[Proposition 6.1]{KLSssms}). In turn, $\mathrm{S}^2 X^{[n]}$ is singular along the diagonal $X^{[n]}$. All in all, the stratification \eqref{stratification moduli space} can be written as
\[
 X^{[n]} \subset \mathrm{S}^2 X^{[n]} \subset M_n,
\]
with $X^{[n]}$, $\mathrm{S}^2 X^{[n]} \setminus X^{[n]}$ and $M_n \setminus \mathrm{S}^2 X^{[n]}$ non-singular.

\section{The desingularisation procedure}\label{sec:The desingularisation procedure}
In \cite{O'GdmsK3}, O'Grady exhibited an explicit resolution of singularities of $M_n=Q^{ss} \sslash \mathrm{PGL}(N)$. In this section, we detail his construction, with complements due to Choy and Kiem \cite{ChoyKiem}. Proofs are omitted, unless they present a technical tool that is repeatedly used, or help reconstructing the very concise arguments of  \cite{ChoyKiem}. Before starting, let us state a couple of general results concerning the delicate relation between quotients and blow-ups.
\subsection{Quotients and blow-ups}
We keep the terminology and conventions of \ref{Notation GIT}. Throughout this subsection, $Y$ denotes a scheme of finite type over $\mathbb{C}$, $\mathrm{G}$ a reductive affine algebraic group acting on $Y$, $V$ a $\mathrm{G}$-invariant closed subscheme of $Y$, $\pi \colon \widetilde{Y} \to Y$ the blow-up of $Y$ along $V$ and $\widetilde{V}= \pi^{-1}(V)$ its exceptional divisor. The action of $\mathrm{G}$ on $Y$ lifts naturally to an action on $\widetilde{Y}$, and the line bundle $\Oca(-\widetilde{V})$ carries a natural $\mathrm{G}$-equivariant structure. The following propositions relate the quotients of $Y$ and $\widetilde{Y}$ by $\mathrm{G}$.

\begin{proposition}\label{prop:Kirwan blow-up}
Assume that $Y$ is projective and that the $\mathrm{G}$-action on $Y$ is linearised in an ample line bundle $\Lca$. Then the following hold.
\begin{enumerate}
\item[(i)] For $d \gg 0$, the $\mathrm{G}$-equivariant line bundle $\pi^*\Lca^{\otimes d} \otimes \Oca(-\widetilde{V})$ on $\widetilde{Y}$ is very ample, and the corresponding semistable and stable loci do not depend on $d$. Let us denote them as $\widetilde{Y}^{ss}$ and $\widetilde{Y}^{s}$, respectively; they satisfy
\[
\widetilde{Y}^{ss} \subset \pi^{-1}(Y^{ss}) \quad \text{and} \quad \widetilde{Y}^s \supset \pi^{-1}(Y^s).
\]
\item[(ii)] For $d$ sufficiently divisible, the GIT quotient $\widetilde{Y}^{ss} \sslash \mathrm{G}$ is identified with a blow-up of $Y^{ss} \sslash \mathrm{G}$, whose centre has the same support as $V^{ss} \sslash \mathrm{G}$.
\end{enumerate}
\end{proposition}
\begin{proof}
For statement (i), see \cite[3.1--3.2]{KPdqnvBn}. For (ii), see \cite[Lemma 3.11]{KPdqnvBn} and its proof.
\end{proof}
In general, it is hard to determine the scheme structure of the centre appearing in \cref{prop:Kirwan blow-up}(ii), which may be non-reduced even if $V$ and hence $V^{ss} \sslash \mathrm{G}$ are reduced. Nevertheless, things are better behaved in the following situation.

\begin{proposition}\label{prop:blow-up and principal G bundles}
Assume that the $\mathrm{G}$-action on $Y$ is set-theoretically free and that a good quotient $Y \sslash  \mathrm{G}$ of $Y$ by $\mathrm{G}$ exists (in particular, a good quotient $V \sslash \mathrm{G}$ of $V$ by $\mathrm{G}$ exists, and the $\mathrm{G}$-action on $\widetilde{Y}$ is set-theoretically free). Then the blow-up of $Y \sslash \mathrm{G}$ along $V \sslash \mathrm{G}$ and its exceptional divisor are good quotients of $\widetilde{Y}$ and $\widetilde{V}$ by $\mathrm{G}$, respectively.
\end{proposition}
\begin{proof}
By Luna's \'etale slice theorem, we deduce from the assumptions that the quotient morphism $\psi \colon Y \to Y \sslash \mathrm{G}$ is a principal $\mathrm{G}$-bundle (\cf \cite[Corollary 4.2.13]{HLgmss}); the same holds true for $V \to V \sslash \mathrm{G}$. Now, the natural map $V \to Y \times _{Y \sslash \mathrm{G}}(V \sslash \mathrm{G})$ is a morphism of principal $\mathrm{G}$-bundles over $V \sslash \mathrm{G}$, hence it is an isomorphism; in other words, $\psi^{-1}(V \sslash \mathrm{G}) \simeq V$. Since $\psi$ is a universal geometric quotient and is flat, the conclusion follows by the commutativity of blowing-up with flat base change.
\end{proof}

\begin{remark}\label{Rem:set-theoretic free implies free}
For a $\mathrm{G}$-action on $Y$ satisfying the assumptions of \cref{prop:blow-up and principal G bundles}, we have just seen that the quotient morphism $Y \to Y \sslash \mathrm{G}$ is a principal $\mathrm{G}$-bundle. In particular, such an action is free (i.e.\ the morphism $Y \times \mathrm{G} \to Y \times Y, \  (y,g) \mapsto (yg,y)$ is a closed immersion) and, if $Y$ is smooth, so is $Y \sslash \mathrm{G}$.
\end{remark}

\subsection{Outline of the resolution}
In the light of \cref{prop:Kirwan blow-up}, to resolve $M_n=Q^{ss} \sslash \mathrm{PGL}(N)$, O'Grady works at the level of the Quot scheme. He takes first the blow-up $\pi_R \colon R \to Q^{ss}$ of $Q^{ss}$ along $\mathrm{\Omega}_Q$, then the blow-up $\pi_S \colon S \to R$ of $R$ along the strict transform $\mathrm{\Sigma}_R$ of $\mathrm{\Sigma}_Q$. The semistable locus of $S$ is smooth and coincides with the stable locus \cite[Claim 1.8.10]{O'GdmsK3}; in particular, all points of $S^{ss}=S^s$ have finite stabilisers. 

The induced morphism $S^s \sslash \mathrm{PGL}(N) \to M_n$ is not yet a resolution of singularities: to resolve the remaining finite quotient singularities of $S^s \sslash \mathrm{PGL}(N)$, a last blow-up is needed; once more, it will come from the blow-up $\pi_T \colon T \to S$ of $S$ along a suitable centre. 

\subsection{Double quotients} \label{rem:triple description of the centres}
In $Q^{ss}$, $R$, $S$ and $T$, we shall consider several closed subschemes that are geometrically significant for the desingularisation procedure; typically, they could be the centre or the exceptional divisor of some blow-up. For any such subscheme $Y$, we would like to have an explicit global description, and to deduce one for its GIT quotient by $\mathrm{PGL}(N)$.

This task gets easier when there exists a variety $U$ equipped with commuting actions of $\mathrm{PGL}(N)$ and of a reductive group $\mathrm{G}$, such that $Y = U \sslash \mathrm{G}$. The quotient of $U$ by $\mathrm{G}$ first, and then by $\mathrm{PGL}(N)$ yields $Y$ and $Y \sslash \mathrm{PGL}(N)$, respectively. But since the actions of $\mathrm{PGL}(N)$ and $\mathrm{G}$ on $U$ commute, one can also take the quotients in the reverse order:
\[
\begin{tikzcd}[column sep = 0.8em]
& U \ar[ld] \ar[rd] \arrow[loop left, "\mathrm{G}"] \arrow[loop right, "\mathrm{PGL}(N)"]& \\
U \sslash \mathrm{PGL}(N) \ar[rd] \arrow[loop left,looseness=3, "\mathrm{G}"] & & Y = U \sslash \mathrm{G} \ar[ld] \arrow[loop right,looseness=3, "\mathrm{PGL}(N)"] \\
& Y \sslash \mathrm{PGL}(N). &
\end{tikzcd}
\]
This picture is slightly inaccurate: indeed, for a GIT quotient, the given actions need to be linearised, and one has to take care of the corresponding semistable loci. To deal rigorously with the situations that shall occur in this paper, we will use the following

\begin{proposition}\label{prop:commuting actions} Let $\mathrm{G}_1$ and $\mathrm{G}_2$ be reductive affine algebraic groups. Let $Y$ be a quasi-projective $\mathrm{G}_1$-variety, and $U$ a quasi-projective variety equipped with commuting actions of $\mathrm{G}_1$ and $\mathrm{G}_2$; in particular $U$ is a $\mathrm{G}_1 \times \mathrm{G}_2$-variety. Let $\sigma \colon U \to Y$ be a $\mathrm{G}_1$-equivariant morphism, which is also a good quotient of $U$ by $\mathrm{G}_2$. 

Suppose that the $\mathrm{G}_1$-action on $Y$ is linearised in a line bundle $\Lca$. Assume the $\mathrm{G}_1 \times \mathrm{G}_2$-action on $U$ is linearised in $\sigma^*\Lca$, in such a way that the $\mathrm{G}_1$-equivariant structures on $\sigma^*\Lca$ induced by $\sigma$, and by restriction along $\mathrm{G}_1 \simeq \mathrm{G}_1 \times \{1_{\mathrm{G}_2}\} \hookrightarrow \mathrm{G}_1 \times \mathrm{G}_2$, respectively, are isomorphic. Then 
\[
Y^{ss}(\Lca) \sslash \mathrm{G}_1 \simeq U^{ss}(\sigma^*\Lca) \sslash (\mathrm{G}_1 \times \mathrm{G}_2).
\]
\end{proposition}
\begin{proof}
If $U$ and $Y$ are affine, and the linearisations are trivial, the statement clearly holds true. In the general case, we can reason as follows.
Since $\sigma \colon U \to Y$ is a good quotient of $U$ by $\mathrm{G}_2$, the map $\sigma^* \colon \Ho^0(Y,\Lca) \xrightarrow{\sim} \Ho^0(U,\sigma^*\Lca)^{\mathrm{G}_2}$ is an isomorphism; in particular, it remains an isomorphism after taking $\mathrm{G}_1$-invariants. This implies that $\sigma^{-1}(Y^{ss}(\Lca))=U^{ss}(\sigma^*\Lca)$, and so we can reduce to the affine situation considered above.
\end{proof}

We introduce now some notation that will help us describe $\mathrm{\Omega}_Q$ and $\mathrm{\Sigma}_R^{ss}$ from the perspective of \cref{rem:triple description of the centres}. 
\subsection{Notation}\label{Notation} 
Recall from \cref{sec: notation projections and base change} that, for any scheme $Y$ over $\mathbb{C}$, we shall denote as $p_Y \colon Y \times X \to Y$ the projection onto the first component; moreover, the base change of any morphism $\sigma \colon Y' \to Y$ along $p_Y$ will be denoted $\sigma_X \colon Y' \times X \to Y \times X$.
 
\subsubsection{} Let $\pi_1$ and $\pi_2$ be the projections from $X^{[n]} \times X^{[n]}$ onto the first and the second factor, respectively. Let $\delta \colon X^{[n]} \rightarrow X^{[n]} \times X^{[n]}$ be the diagonal imbedding, and $b \colon \Xca^{[n]} \rightarrow X^{[n]} \times X^{[n]}$ the blow-up of $X^{[n]} \times X^{[n]}$ along the diagonal. The exceptional divisor $E$ is isomorphic to the projective bundle $\PP(\Tca)$ associated with the tangent bundle $\Tca$ of $X^{[n]}$. Denote by $\varepsilon \colon E \rightarrow \Xca^{[n]}$ the inclusion morphism, and by $\beta \colon E \rightarrow X^{[n]}$ the restriction of $b$ to $E$. 

This geometric situation can be visualised in a commutative diagram with cartesian squares:
\[
\begin{tikzcd}[scale cd=0.8, column sep=1.0em]
& E \times X \ar[ld, "p_E"'] \ar[dd, "\beta_X"{yshift=15pt}] \ar[rr,hook,"\varepsilon_X"] & & \Xca^{[n]} \times X \ar[ld,"p_{\Xca^{[n]}}"] \ar[dd,"b_X"] \\
\quad E \simeq \PP(\Tca) \quad \ar[dd, "\beta"'] \ar[rr,"\varepsilon"{xshift=15pt},hook, crossing over]  & & \Xca^{[n]} \\
& X^{[n]} \times X \ar[ld, "p_{X^{[n]}}"'] \ar[rr,"\delta_X"{xshift=-15pt},hook] & & X^{[n]} \times X^{[n]} \times X \ar[ld,"p_{X^{[n]} \times X^{[n]}}"] \ar[rd,"\pi_{2,X}"] \\
 X^{[n]} \ar[rr,"\delta",hook] & & X^{[n]} \times X^{[n]} \ar[ld, "\pi_1"'] \ar[rd,"\pi_2"] \ar[from=uu, "b"{yshift=15pt}, crossing over] & & X^{[n]} \times X \ar[ld,"p_{X^{[n]}}"] \\
& X^{[n]} & & X^{[n]}.
\end{tikzcd}
\]
\subsubsection{}\label{sec:involution action} The involution $\iota$ interchanging the factors of $X^{[n]} \times X^{[n]}$ preserves the diagonal; hence $\iota$ lifts to an involution $\widetilde{\iota}$ of $\Xca^{[n]}$. The quotient of $X^{[n]} \times X^{[n]}$ by $\iota$ is the singular locus $\mathrm{S}^2X^{[n]}$ of $M_n$. The quotient of $\Xca^{[n]}$ by $\widetilde{\iota}$ is the Hilbert square $(X^{[n]})^{[2]}$ of $X^{[n]}$. 
\subsubsection{}\label{sec:PGL(N)-bundles} Let $\Ica$ be the tautological ideal sheaf on $X^{[n]} \times X$. For the $k$ chosen in \cref{sec:construction of moduli space}, $p_{{X^{[n]}},*} [\Ica (k)]$ is a vector bundle on $X^{[n]}$ of rank $N/2$. Define on $X^{[n]} \times X^{[n]} \times X$ the sheaves 
\[
\Ica_1 \coloneqq \pi_{1,X}^* \Ica \quad \text{and} \quad \Ica_2 \coloneqq \pi_{2,X}^* \Ica.
\]
Then $p_{{X^{[n]}\times X^{[n]}},*} [(\Ica_1 \oplus \Ica_2) (k)]$ is a vector bundle of rank $N$ on $X^{[n]} \times X^{[n]}$. 
Consider the principal $\mathrm{PGL}(N)$-bundle
\[
\PP\sIsom(\Oca^{\oplus N}, p_{X^{[n]}\times X^{[n]},*} [(\Ica_1 \oplus \Ica_2) (k) ]) \to X^{[n]}\times X^{[n]}.
\]
Its base change along $\delta$, $b$ and $b \circ \varepsilon = \delta \circ \beta$ yields the respective principal $\mathrm{PGL}(N)$-bundles
\begin{align*}
Y \coloneqq \PP\sIsom(\Oca^{\oplus N}, p_{{X^{[n]}},*} [\Ica(k)^{\oplus 2}]) \xrightarrow{\upsilon} X^{[n]}, \\
U \coloneqq \PP\sIsom(\Oca^{\oplus N}, p_{\Xca^{[n]},*} b_X^*[(\Ica_1 \oplus \Ica_2) (k)]) \xrightarrow{u} \Xca^{[n]}, \\
D \coloneqq \PP\sIsom(\Oca^{\oplus N}, p_{E,*} \beta_X^* [\Ica(k)^ {\oplus 2}]) \rightarrow E.
\end{align*}
They fit into a commutative diagram with cartesian squares:
\begin{equation}\label{principal PGL(N)-bundles}
\begin{tikzcd}[scale cd=0.8]
& Y \times X \ar[ld, "p_Y"'] \ar[dd,"\upsilon_X"{yshift=15pt}] & & D \times X \ar[ld, "p_D"'] \ar[ll] \ar[dd]  \ar[rr,hook]& & U \times X \ar[ld,"p_U"'] \ar[dd,"u_X"] \\
Y \ar[dd, "\upsilon"] & & D \ar[ll, crossing over] \ar[rr,hook,crossing over] & & U \\
& X^{[n]} \times X \ar[ld,"p_{X^{[n]}}"] & & E \times X \ar[ld, "p_E"] \ar[ll,"\beta_X"'{xshift=-15pt}] \ar[rr,"\varepsilon_X"{xshift=-15pt},hook] & & \Xca^{[n]} \times X \ar[ld,"p_{\Xca^{[n]}}"] \\
X^{[n]} & & E \ar[from=uu, crossing over] \ar[ll, "\beta"'] \ar[rr,"\varepsilon",hook] & & \Xca^{[n]}. \ar[from=uu,"u"{yshift=15pt}, crossing over]
\end{tikzcd}
\end{equation}

\subsection{The centre of the first blow-up} O'Grady's resolution starts with the blow-up of $Q^{ss}$ along $\mathrm{\Omega}_Q$. Let us recall the relation between $\mathrm{\Omega}_Q$ and $Y$. Consider the algebraic groups $\mathrm{GL}(N)$ and $\mathrm{GL}(2)$. On $X^{[n]}$, they act trivially. The standard representation of $\mathrm{GL}(N)$ on $\mathbb{C}^N$ induces a $\mathrm{GL}(N)$-equivariant structure on $\Oca^{\oplus N}_{X^{[n]}}$, while $p_{{X^{[n]}},*} [\Ica(k)^{\oplus 2}]$ inherits from $\Ica^{\oplus 2}$ a $\mathrm{GL}(2)$-equivariant structure. We deduce on $Y$ commuting actions of $\mathrm{GL}(N)$ and $\mathrm{GL}(2)$, that descend to set-theoretically free actions of $\mathrm{PGL}(N)$ and $\mathrm{PGL}(2)$. By \cref{Rem:set-theoretic free implies free}, these actions are in fact free: indeed, good quotients of $Y$ by $\mathrm{PGL}(N)$ and $\mathrm{PGL}(2)$ exist, because the morphism $\upsilon \colon Y \to X^{[n]}$ is affine and invariant for $\mathrm{PGL}(N)$ and $\mathrm{PGL}(2)$.

\begin{lemma}\cite[Proposition 1.5.1]{O'GdmsK3}\cite[Lemma 4.1(1)]{ChoyKiem}\label{lemma-GlobalOmegaQ} The scheme $\mathrm{\Omega}_Q$ is smooth and isomorphic to $Y \sslash \mathrm{PGL}(2)$.
\end{lemma}
\subsection{The first blow-up}\label{sec:the first blow-up}
Let 
\[
\pi_R \colon R \to Q^{ss}
\]
be the blow-up of $Q^{ss}$ along $\mathrm{\Omega}_Q$. Denote by $\mathrm{\Omega}_R$ the exceptional divisor, and by $\mathrm{\Sigma}_R$ the strict transform of $\mathrm{\Sigma}_Q$:
\[
\mathrm{\Omega}_R= \pi_R^{-1}(\mathrm{\Omega}_Q), \quad \quad \mathrm{\Sigma}_R \simeq \mathrm{Bl}_{\mathrm{\Omega}_Q} \mathrm{\Sigma}_Q.
\]
In this section we recall global geometric descriptions of $\mathrm{\Omega}_R$, $\mathrm{\Omega}_R \cap \mathrm{\Sigma}_R$ and $\mathrm{\Sigma}_R^{ss}$. 

\medskip

\subsubsection{} Let us start from $\mathrm{\Omega}_R$. The fibre of $\pi_R$ over a point $[q] \in \mathrm{\Omega}_Q$ is the projectivisation of $(C_{\mathrm{\Omega_Q}/Q^{ss}})_{[q]}$, whose expression was given in \cref{sec:Points in in OmegaQ}. We keep the notation introduced therein. Consider the fibre bundle $\PP \sHom^\omega(W,\Tca)$ (resp.\ $\PP \sHom_i^\omega(W,\Tca)$) on $X^{[n]}$ whose fibre over $[Z]$ is
\begin{gather*}
\{ [\varphi] \in \PP\mathrm{Hom}(W,\mathrm{E}_Z) \mid \im(\varphi) \ \text{is $\omega$-isotropic} \} \\
(\text{resp.} \ \{ [\varphi] \in \PP\mathrm{Hom}(W,\mathrm{E}_Z) \mid \im(\varphi) \ \text{is $\omega$-isotropic and} \ \mathrm{dim}(\im(\varphi)) \le i \});
\end{gather*}
note in particular that $\PP \sHom_1^\omega(W,\Tca) \simeq \PP(W^\vee) \times \PP(\Tca)$. The fibre bundle $\PP \sHom^\omega(W,\Tca)$ (resp.\ $\PP \sHom_i^\omega(W,\Tca)$) is endowed with an action of $\mathrm{SO}(W)$ by composition. Thus, its base change $\upsilon^*\PP\sHom^\omega(W,\Tca)$ (resp.\ $\upsilon^*\PP\sHom_i^\omega(W,\Tca)$) along $\upsilon \colon Y \to X^{[n]}$ is equipped with commuting actions of $\mathrm{SO}(W)\simeq \mathrm{PGL}(2)$ and $\mathrm{PGL}(N)$, which are both set-theoretically free.
\begin{lemma}\label{lem:exceptional divisor first blow-up} We have isomorphisms 
\begin{enumerate}
\item[(1)] $\mathrm{\Omega}_R \simeq [\upsilon^*\PP\sHom^\omega(W,\Tca)] \sslash \mathrm{SO}(W) $ and
\item[(2)] $\mathrm{\Omega}_R^{ss} \sslash \mathrm{PGL}(N) \simeq \PP\sHom^\omega(W,\Tca)^{ss} \sslash \mathrm{SO}(W)$.
\end{enumerate}
\end{lemma}
\begin{proof}
The first isomorphism is proved in  \cite[Lemma 4.1(2)]{ChoyKiem}. After \cref{rem:triple description of the centres}, the second statement can be deduced from the first. More precisely, we apply twice \cref{prop:commuting actions} to the $\mathrm{SO}(W)$- and $\mathrm{PGL}(N)$-actions on $\upsilon^*\PP\sHom^\omega(W,\Tca)$. On the one hand, the $\mathrm{PGL}(N)$-equivariant morphism $\upsilon^*\PP\sHom^\omega(W,\Tca) \to \mathrm{\Omega}_R$ is a good quotient by $\mathrm{SO}(W)$; on the other, the $\mathrm{SO}(W)$-equivariant morphism $\upsilon^*\PP\sHom^\omega(W,\Tca) \to \PP\sHom^\omega(W,\Tca)$ is a good quotient by $\mathrm{PGL}(N)$, because $\upsilon \colon Y \to X^{[n]}$ is a (universal) geometric quotient by $\mathrm{PGL}(N)$. Therefore, 
\[
\mathrm{\Omega}_R^{ss} \sslash \mathrm{PGL}(N)\simeq [\upsilon^*\PP\sHom^\omega(W,\Tca)]^{ss} \sslash (\mathrm{PGL}(N) \times \mathrm{SO}(W)) \simeq \PP\sHom^\omega(W,\Tca)^{ss} \sslash \mathrm{SO}(W). \qedhere
\]
\end{proof}
\subsubsection{} Next, we consider the scheme-theoretic intersection of $\mathrm{\Omega}_R$ and $\mathrm{\Sigma}_R$.
\begin{lemma}[{\cf \cite[Proposition 1.7.1(2)]{O'GdmsK3} and \cite[Corollary 4.3]{ChoyKiem}}] \label{lemma:OmegaR-SigmaR}
The scheme-theoretic intersection of $\mathrm{\Omega}_R$ and $\mathrm{\Sigma}_R$ is smooth. Moreover,
\begin{enumerate}
\item[(1)] $\mathrm{\Omega}_R \cap \mathrm{\Sigma}_R \simeq [\upsilon^*\PP\sHom_1^\omega(W,\Tca)] \sslash \mathrm{SO}(W)$ and
\item[(2)] $(\mathrm{\Omega}_R \cap \mathrm{\Sigma}_R)^{ss} \sslash \mathrm{PGL}(N) \simeq  E$.
\end{enumerate}
\end{lemma}
\begin{proof}
The scheme $\mathrm{\Omega}_R \cap \mathrm{\Sigma}_R$ is the exceptional divisor of the blow-up of $\mathrm{\Sigma}_Q$ along $\mathrm{\Omega}_Q$. Thus, it is isomorphic to $\PP(C_{\mathrm{\Omega}_Q/\mathrm{\Sigma}_Q})$, which is smooth: indeed, the projection morphism to the smooth $\mathrm{\Omega}_Q$ is flat with smooth fibres, as one sees from \cref{normal flatness of Q}(ii) and equality \eqref{fibres cone to OmegaQ in SigmaQ}.

To prove (1), consider in $\mathrm{\Omega}_R \simeq [\upsilon^*\PP\sHom^\omega(W,\Tca)] \sslash \mathrm{SO}(W) $ the closed subschemes $\mathrm{\Omega}_R \cap \mathrm{\Sigma}_R$ and $[\upsilon^*\PP\sHom_1^\omega(W,\Tca)] \sslash \mathrm{SO}(W)$: they are reduced and have the same support, so they are equal.

Let us show (2). Reasoning as in the proof of \cref{lem:exceptional divisor first blow-up}, we deduce from (1) the isomorphism
\[
(\mathrm{\Omega}_R \cap \mathrm{\Sigma}_R)^{ss} \sslash \mathrm{PGL}(N) \simeq \PP\sHom_1^\omega(W,\Tca)^{ss} \sslash \mathrm{SO}(W).
\]
Now, the GIT quotient of $\PP\sHom_1^\omega(W,\Tca)$ (identified to $\PP(W) \times \PP(\Tca)$ via the Killing form) by $\mathrm{SO}(W)$ can be computed as follows. Consider the tautological $\mathrm{SO}(W)$-action on $\PP(W)$. On its coordinate ring $\mathrm{S}^\bullet W^\vee$, we have then an induced $\mathrm{SO}(W)$-action; by the classical results on the invariant theory of the (special) orthogonal groups (see e.g. \cite[Section I.3]{LBtrg2x2m}), the ring of invariants is 
\[
(\mathrm{S}^\bullet W^\vee)^{\mathrm{SO}(W)} = \mathbb{C}[\kappa],
\] 
where $\kappa \in \mathrm{S}^2(W^\vee)$ is the Killing form of $W$. In particular, $\PP(W)^{ss}$ is the complement of the conic defined by $\kappa$, and $\PP(W)^{ss} \sslash \mathrm{SO}(W) =\mathrm{Proj}(\mathbb{C}[\kappa])$ equals $\mathrm{Spec}(\mathbb{C})$.
We conclude that $\PP\sHom_1^\omega(W,\Tca)^{ss} \sslash \mathrm{SO}(W) \simeq \PP(\Tca) = E$, as desired.
\end{proof}
\subsubsection{} For future reference, we consider in $\mathrm{\Omega}_R$ also the closed subscheme 
\[
\mathrm{\Delta}_R \coloneqq [\upsilon^*\PP\sHom^\omega_2(W,\Tca)] \sslash \mathrm{SO}(W).
\]
Its intersection with $\mathrm{\Sigma}_R$ coincides with $\mathrm{\Omega}_R \cap \mathrm{\Sigma}_R$, and $\mathrm{\Delta}_R^{ss} \sslash \mathrm{PGL}(N) \simeq \PP\sHom^\omega_2(W,\Tca)^{ss} \sslash \mathrm{SO}(W)$.

\medskip

\subsubsection{}\label{sec:O(2) linearisation} Let us turn to $\mathrm{\Sigma}_R$, and relate its semistable locus to $U$. 
Consider the algebraic group $\Orm(2) \simeq \mathbb{C}^\times \rtimes \bm{\mu}_2$, realised as the subgroup of $\mathrm{GL}(2)$ generated by 
\begin{equation} \label{O(2)}
\mathrm{SO}(2) = \left\{ \theta_\lambda = 
\begin{pmatrix}
\lambda & 0 \\
0 & \lambda^{-1}
\end{pmatrix}
\Bigm| \lambda \in \mathbb{C}^\times \right\} \ \ \ \text{and}
 \ \ \ \tau =
 \begin{pmatrix}
0 & 1 \\
1 & 0
\end{pmatrix}.
 \end{equation} 
On $\Xca^{[n]}$, the quotient $\Orm(2)/\mathrm{SO}(2) \simeq \bm{\mu}_2$ acts via $\widetilde{\iota}$ as in \ref{sec:involution action}, while $\mathrm{GL}(N)$ acts trivially. Give $\Oca^{\oplus N}_{\Xca^{[n]}}$ the $\mathrm{GL}(N)$-equivariant structure that comes from the standard representation of $\mathrm{GL}(N)$ on $\mathbb{C}^N$, and $p_{{\Xca^{[n]}},*}[b_X^*(\Ica_1\oplus \Ica_2)(k)]$ the $\Orm(2)$-equivariant structure induced from the following one on $\Ica_1 \oplus \Ica_2$:
$\theta_\lambda$ multiplies $\lambda$ (resp.\ $\lambda^{-1}$) to $\Ica_1$ (resp.\ to $\Ica_2$); $\tau$ can be regarded as an isomorphism between $\Ica_1 \oplus \Ica_2$ and $\iota^*(\Ica_1 \oplus \Ica_2) \simeq \Ica_2 \oplus \Ica_1$. We deduce on $U$ commuting actions of $\mathrm{GL}(N)$ and $\Orm(2)$; they descend to set-theoretically free actions of $\mathrm{PGL}(N)$ and $\mathrm{PO}(2) = \mathrm{O}(2)/\{ \pm \mathrm{Id} \}$\footnote{There is a little clash of notation with our references. The group $\mathrm{PO}(2)$ is abstractly isomorphic to $\mathrm{O}(2)$; the $\mathrm{PO}(2)$-action on $U$ that we consider here is written in \cite{O'GdmsK3} and \cite{ChoyKiem} as an action of $\mathrm{O}(2)$. Moreover, what we call here $U$ is  in \cite[Section 1.7.III]{O'GdmsK3} denoted as $\mathscr{\tilde{U}}$.}, which preserve the divisor $D$. 

\begin{proposition}\cite[Proposition~1.7.10]{O'GdmsK3}\label{prop:global description of SigmaR} The scheme $\mathrm{\Sigma}_R^{ss}$ is a good quotient of $U$ by $\mathrm{PO}(2)$; in particular, it is integral and smooth. Moreover, we have isomorphisms
\begin{enumerate} 
\item[(1)] $\mathrm{\Sigma}_R^{ss} \sslash \mathrm{PGL}(N) \simeq (X^{[n]})^{[2]}$,
\item[(2)] $(\mathrm{\Omega}_R \cap \mathrm{\Sigma}_R)^{ss} \simeq D \sslash \mathrm{PO}(2)$.
\end{enumerate}
\end{proposition}
\begin{proof}
Let us recall how to define a morphism from $U$ to $\mathrm{\Sigma}_R^{ss}$. The tautological inclusion 
\[
\Oca_u(-1) \rightarrow u^* \sHom(\Oca^{\oplus N}, p_{{\Xca^{[n]}},*} [b_X^*(\Ica_1\oplus \Ica_2)(k)])
\]
induces by adjunction an isomorphism $\Oca^{\oplus N} \simeq  u^*(p_{{\Xca^{[n]}},*}[b_X^*(\Ica_1\oplus \Ica_2)(k)]) \otimes \Oca_u(1)$ on $U$.
Its pullback by $p_U$ yields
\begin{multline*}
\Oca^{\oplus N} \simeq p_U^*\{ u^* ( p_{{\Xca^{[n]}},*}[b_X^*(\Ica_1\oplus \Ica_2)(k)]) \otimes \Oca_u(1) \} \\
\simeq u_X^* p_{\Xca^{[n]}}^* p_{{\Xca^{[n]}},*}[b_X^*(\Ica_1\oplus \Ica_2)(k)]\otimes p_U^* \Oca_u(1) 
\rightarrow u_X^*b_X^*(\Ica_1\oplus \Ica_2)(k) \otimes p_U^* \Oca_u(1).
\end{multline*}
Thus, we get an epimorphism of $\Oca_{U \times X}$-modules
\begin{equation}\label{quotient on U times X}
\Oca(-k)^{\oplus N} \rightarrow u_X^* b_X^* ( \Ica_1 \oplus \Ica_2 ) \otimes p_U^* \Oca_u(1).
\end{equation}
By the universal property of the Quot scheme, it induces a morphism $U \rightarrow Q^{ss}$, with image $\mathrm{\Omega}_Q^\circ \sqcup \mathrm{\Sigma}_Q^\circ$. As shown in the proof of \cite[Proposition~1.7.10]{O'GdmsK3}, this morphism lifts to a morphism $U \to \mathrm{\Sigma}_R^{ss}$, which is a good quotient by $\mathrm{PO}(2)$  and restricts to a good quotient $D \to (\mathrm{\Omega}_R \cap \mathrm{\Sigma}_R)^{ss}$ by $\mathrm{PO}(2)$. 

It remains to show Item 1. Reasoning as for \cref{lem:exceptional divisor first blow-up}, we obtain
\[
\mathrm{\Sigma}_R^{ss} \sslash \mathrm{PGL}(N) \simeq U \sslash (\mathrm{PGL}(N) \times \mathrm{PO}(2)) \simeq \Xca^{[n]} \sslash \mathrm{PO}(2).
\]
The last quotient can be computed in two steps: first by $\mathrm{PSO}(2) = \mathrm{SO}(2)/ \{ \pm \mathrm{Id} \}$, and then by the residual action of $\mathrm{PO}(2)  / \mathrm{PSO}(2) \simeq \bm{\mu}_2$. Since $\mathrm{PSO}(2)$ acts trivially on $\Xca^{[n]}$, we get $\Xca^{[n]} \sslash \mathrm{PO}(2) \simeq\Xca^{[n]} /\bm{\mu}_2 \simeq (X^{[n]})^{[2]}$, as desired.
\end{proof}
\begin{remark}\label{rem:GL2action on D}
On $D \simeq Y \times_{X^{[n]}}E$, there is also a free $\mathrm{PGL}(2)$-action coming from $Y$; it turns $D$ into a principal $\mathrm{PGL}(2)$-bundle over $\mathrm{\Omega}_Q \times_{X^{[n]}}E$.
\end{remark}

\subsubsection{The normal cone to $\mathrm{\Sigma}_R$ in $R$} The centre of the second blow-up is $\mathrm{\Sigma}_R$. Therefore, in order to determine the corresponding exceptional divisor, the following proposition is important.
\begin{proposition}\label{prop:normal cone SigmaR fibres} The scheme $R^{ss}$ is normally flat along $\mathrm{\Sigma}_R^{ss}$. Over any closed point of $\mathrm{\Sigma}_R^{ss}$, the fibre of the normal cone $C_{\mathrm{\Sigma}_R/R}$ to $\mathrm{\Sigma}_R$ in $R$ is the affine cone over a smooth quadric in $\PP^{4n-5}$.
\end{proposition} 
\begin{proof}
We prove only that $R^{ss}$ is normally flat along $\mathrm{\Sigma}_R^{ss}$, the rest being \cite[Proposition 1.7.1(3)]{O'GdmsK3}. Let $y$ be a closed point of $\mathrm{\Sigma}_R^{ss}$. Since $\mathrm{\Sigma}_R^{ss}\setminus \mathrm{\Omega}_R \simeq \mathrm{\Sigma}_Q^\circ$, we may assume by \cref{normal flatness of Q}(iii) that $y \in \mathrm{\Omega}_R$. Write $\pi_R(y) \in \mathrm{\Omega}_Q$ as $[q \colon \Hca \to I_Z^{\oplus 2}]$, for $[Z] \in X^{[n]}$. Thanks to Luna's theorem (applied to $[q]$ as in the proof of \cref{normal flatness of Q}) and since blowing-up commutes with flat base change, it suffices to check the normal flatness of $\mathrm{Bl}_o\mathrm{Hom}^\omega(W,\mathrm{E}_Z)$ along $\mathrm{Bl}_o\mathrm{Hom}_1(W, \mathrm{E}_Z)$. As the blow-up of a cone $Y \subset \AA^m$ at the origin $o$ is the total space of the tautological line subbundle on $\PP(Y)$, such a normal flatness results from that of $\PP \mathrm{Hom}^\omega(W,\mathrm{E}_Z)$ along $\PP \mathrm{Hom}_1(W, \mathrm{E}_Z)$, which \cref{lem:normal flatness}(ii) reduces to the local computation done on page 77 of \cite{O'GdmsK3}.
\end{proof}
\cref{prop:normal cone SigmaR fibres} extends to $\mathrm{\Sigma}_R^{ss} \cap \mathrm{\Omega}_R$ the \emph{fibrewise} description of $C_{\mathrm{\Sigma}_R/R}$ over $\mathrm{\Sigma}_R^{ss}\setminus \mathrm{\Omega}_R \simeq \mathrm{\Sigma}_Q^\circ$, which was given in \cref{Points in SigmaQ}.
 A \emph{global} description of $C_{\mathrm{\Sigma}_R/R}$ over ${\mathrm{\Sigma}_R^{ss}\setminus \mathrm{\Omega}_R}$ can be found in \cite{ChoyKiem} after Lemma 4.4, or in \cite[Section 3]{MABBccc}. The main technical contribution of the present paper is a global description of $C_{\mathrm{\Sigma}_R/R}$ that holds on the whole $\mathrm{\Sigma}_R^{ss}$ (see \cref{sec:The exceptional divisor of the second blow-up}).

\subsection{The second blow-up}\label{section:second blow-up}
Let 
\[
\pi_S \colon S \to R
\]
be the blow-up of $R$ along $\mathrm{\Sigma}_R$. Let $\mathrm{\Sigma}_S$ be its exceptional divisor, and $\mathrm{\Omega}_S$ (resp.\ $\mathrm{\Delta}_S$) the strict transform of $\mathrm{\Omega}_R$ (resp.\ $\mathrm{\Delta}_R$):
\[
\mathrm{\Sigma}_S = \pi_S^{-1}(\mathrm{\Sigma}_R), \quad \quad \mathrm{\Omega}_S \simeq \mathrm{Bl}_{\mathrm{\Sigma}_R \cap \mathrm{\Omega}_R} \mathrm{\Omega}_R, \quad \quad \mathrm{\Delta}_S \simeq \mathrm{Bl}_{\mathrm{\Sigma}_R \cap \mathrm{\Delta}_R} \mathrm{\Delta}_R.
\]
Let us say in passing that, for the action of $\mathrm{PGL}(N)$ on $S$, the stable locus $S^s$ equals the semistable locus $S^{ss}$ (see \cite[Claim 1.8.10]{O'GdmsK3}). 

In this section, we recall global geometric descriptions of $\mathrm{\Omega}_S$,  $\mathrm{\Omega}_S \cap \mathrm{\Sigma}_S$, $\mathrm{\Delta}_S$ and $\mathrm{\Delta}_S \cap \mathrm{\Sigma}_S$.

\subsubsection{}\label{notation blow-ups} We need some notation, that we borrow from \cite[Section 4]{ChoyKiem}. Let $\mathrm{Bl}^\Tca$ (resp.\ $\mathrm{Bl}^\Tca_2$) be the blow-up of $\PP\sHom^\omega(W,\Tca)$ (resp.\ $\PP\sHom^\omega_2(W,\Tca)$) along $\PP\sHom_1^\omega(W,\Tca)$. We denote by $E^\Tca$ (resp.\ $E^\Tca_2$) its exceptional divisor. 

The scheme $\mathrm{Bl}^\Tca$ (resp.\ $\mathrm{Bl}_2^\Tca$) is acted on by $\mathrm{SO}(W)$. Hence, its base change $\upsilon^*\mathrm{Bl}^\Tca$ (resp.\ $\upsilon^*\mathrm{Bl}_2^\Tca$) along $\upsilon \colon Y \to X^{[n]}$ is equipped with commuting actions of $\mathrm{SO}(W)$ and $\mathrm{PGL}(N)$, which are both set-theoretically free.
\begin{proposition} \cite[Proposition 1.8.4]{O'GdmsK3} \cite[Proof of Proposition 3.2(1),(4)]{ChoyKiem} \label{prop:OmegaS;OmegaS-SigmaS}The scheme $\mathrm{\Omega}_S$ is smooth, and we have isomorphisms
\begin{enumerate}
\item[(1)] $\mathrm{\Omega}_S \simeq (\upsilon^* \mathrm{Bl}^\Tca) \sslash \mathrm{SO}(W)$ and 
\item[(2)] $\mathrm{\Omega}_S^s \sslash \mathrm{PGL}(N) \simeq (\mathrm{Bl}^\Tca)^s \sslash \mathrm{SO}(W)$.
\end{enumerate}
The scheme-theoretic intersection of $\mathrm{\Omega}_S$ and $\mathrm{\Sigma}_S$ is smooth, and we have isomorphisms
\begin{enumerate}
\item[(3)] $\mathrm{\Omega}_S \cap \mathrm{\Sigma}_S \simeq (\upsilon^* E^\Tca) \sslash \mathrm{SO}(W)$ and 
\item[(4)] $(\mathrm{\Omega}_S \cap \mathrm{\Sigma}_S)^s \sslash \mathrm{PGL}(N) \simeq (E^\Tca)^s \sslash \mathrm{SO}(W)$.
\end{enumerate}

\end{proposition}
\begin{remark}
Items (1) and (3) result from \cref{lem:exceptional divisor first blow-up} and \cref{lemma:OmegaR-SigmaR} by \cref{prop:blow-up and principal G bundles}:
\[
\begin{tikzcd}[column sep= 0.7em, row sep = 1.3em, scale cd=0.8]
& E^\Tca \ar[dd] \ar[ld] & & \upsilon^* E^\Tca \ar[ll] \ar[rr] \ar[dd] \ar[ld] & & \mathrm{\Omega}_S \cap \mathrm{\Sigma}_S \ar[dd] \ar[ld] \\
\mathrm{Bl}^\Tca \ar[dd]   \ar[from=rr, crossing over]  & & \upsilon^*\mathrm{Bl}^\Tca& & \mathrm{\Omega}_S \ar[from=ll, crossing over] & \\
& \PP\sHom_1^\omega(W,\Tca) \ar[ld] & & \upsilon^*\PP\sHom_1^\omega(W,\Tca) \ar[ll] \ar[rr] \ar[ld] & & \mathrm{\Omega}_R \cap \mathrm{\Sigma}_R \ar[ld] \\
\PP\sHom^\omega(W,\Tca) & & \upsilon^*\PP\sHom^\omega(W,\Tca) \ar[ll,"\mathrm{PGL}(N)-\text{bundle}","\text{principal}"'] \ar[rr,"\text{principal}", "\mathrm{SO}(W)-\text{bundle}"'] \ar[from=uu, crossing over] & & \mathrm{\Omega}_R. \ar[from=uu, crossing over] &
\end{tikzcd}
\]
\end{remark}
\begin{proposition}\cite[Proposition 1.8.11]{O'GdmsK3} \cite[Proof of Proposition 3.2(3)]{ChoyKiem}\label{prop:DeltaS;DeltaS-SigmaS}
The scheme $\mathrm{\Delta}_S$ is smooth, and we have isomorphisms
\begin{enumerate}
\item[(1)] $\mathrm{\Delta}_S \simeq (\upsilon^* \mathrm{Bl}_2^\Tca) \sslash \mathrm{SO}(W)$ and
\item[(2)]  $\mathrm{\Delta}_S^s \sslash \mathrm{PGL}(N) \simeq (\mathrm{Bl}_2^\Tca)^s \sslash \mathrm{SO}(W)$;

\smallskip

\item[(3)] $\mathrm{\Delta}_S \cap \mathrm{\Sigma}_S \simeq (\upsilon^* E_2^\Tca) \sslash \mathrm{SO}(W)$ and 
\item[(4)] $(\mathrm{\Delta}_S \cap \mathrm{\Sigma}_S)^s \sslash \mathrm{PGL}(N) \simeq (E_2^\Tca)^s \sslash \mathrm{SO}(W)$.
\end{enumerate}
\end{proposition}
\subsubsection{}\label{sec:Grass(2,T)} In proving Proposition 3.2(3) of \cite{ChoyKiem}, Choy and Kiem gave an alternative description of $\mathrm{\Delta}_S$, which we now recall. 

Let $\alpha \colon \Grass^\omega(2,\Tca) \to X^{[n]}$ be the relative symplectic Grassmannian. Denote by $\Aca \subset \alpha^*\Tca$ its tautological rank-two bundle, by $\vartheta \colon \PP(\Aca) \to \Grass^\omega(2,\Tca)$ the associated projective bundle, and by $\Tca_\vartheta$ the relative tangent bundle of $\vartheta$. 

Over $\Grass^\omega(2,\Tca)$, consider also the projective bundle $\PP\sHom(W,\Aca) \coloneqq \PP (\Aca \otimes W^\vee)$ and its subbundle $\PP\sHom_1(W,\Aca) \simeq \PP (W^\vee) \times \PP (\Aca)$ of rank-one homomorphisms. 

On $\PP\sHom(W,\Aca)$ (resp.\ $\PP\sHom_1(W,\Aca)$), the group $\mathrm{SO}(W)$ acts by precomposition. The corresponding GIT quotient is the space $\PP(\mathrm{S}^2\Aca)$ (resp.\ $\PP(\mathrm{S}_1^2\Aca) \simeq \PP(\Aca)$) parametrising quadrics (resp.\ rank-one quadrics) in the fibres of $\PP(\Aca^\vee)$ over $\Grass^\omega(2,\Tca)$; the quotient morphism is
\[
\PP\sHom(W,\Aca)^{ss} \to \PP(\mathrm{S}^2\Aca), \quad [\varphi] \mapsto [\varphi \circ \varphi^\vee].\footnote{Here, $\varphi^\vee \colon \Aca^\vee \to \Oca \otimes W^\vee$ denotes the transpose of $\varphi$, and we identify $W^\vee$ with $W$ via the Killing form.}
\]
Let $\mathrm{Bl}^\Aca$ be the blow-up of $\PP\sHom(W,\Aca)$ along $\PP\sHom_1(W,\Aca)$, and $E^\Aca$ its exceptional divisor. 

The $\mathrm{SO}(W)$-action lifts to $\mathrm{Bl}^\Aca$ and $E^\Aca$; the corresponding GIT quotients are $\mathrm{Bl}_{\PP (\Aca)}\PP (\mathrm{S}^2\Aca)$ and $\PP (\mathrm{S}^2 \Tca_\vartheta)$, respectively. Since $\PP(\Aca)$ is a Cartier divisor on $\PP(\mathrm{S}^2\Aca)$ and $\Tca_\vartheta$ is a line bundle, these quotients can be identified with $\PP(\mathrm{S}^2\Aca)$ and $\PP (\Aca)$, respectively.

We have a natural $\mathrm{SO}(W)$-equivariant forgetful morphism $\PP\sHom(W,\Aca) \to \PP\sHom_2^\omega(W,\Tca)$ over $X^{[n]}$. The schematic preimage of $\PP\sHom_1^\omega(W,\Tca)$ is $\PP\sHom_1(W,\Aca)$; thus, we obtain an $\mathrm{SO}(W)$-equivariant morphism from $\mathrm{Bl}^\Aca \to \mathrm{Bl}_2^\Tca$, which turns out to be an isomorphism. It identifies $E^\Aca$ with $E^\Tca_2$.

The above discussion translates immediately \cref{prop:DeltaS;DeltaS-SigmaS} into the following
\begin{proposition}\label{prop:DeltaS}  We have isomorphisms
\begin{enumerate}
\item[(1)] $\mathrm{\Delta}_S \simeq (\upsilon^* \mathrm{Bl}^\Aca) \sslash \mathrm{SO}(W)$ and
\item[(2)]  $\mathrm{\Delta}_S^s \sslash \mathrm{PGL}(N) \simeq \mathrm{Bl}_{\PP (\Aca)}\PP (\mathrm{S}^2\Aca) \simeq \PP (\mathrm{S}^2\Aca)$.
\end{enumerate}
The scheme-theoretic intersection of  $\mathrm{\Delta}_S$ and $\mathrm{\Sigma}_S$ is smooth, and we have isomorphisms
\begin{enumerate}
\item[(3)] $\mathrm{\Delta}_S \cap \mathrm{\Sigma}_S \simeq (\upsilon^* E^\Aca) \sslash \mathrm{SO}(W)$ and 
\item[(4)] $(\mathrm{\Delta}_S \cap \mathrm{\Sigma}_S)^s \sslash \mathrm{PGL}(N) \simeq (E^\Aca)^s \sslash \mathrm{SO}(W) \simeq \PP (\mathrm{S}^2 \mathcal{T}_\vartheta) \simeq \PP (\Aca)$.
\end{enumerate}
\end{proposition}
\subsubsection{} Before moving on to the third blow-up, let us state a technical lemma.
\begin{lemma}
$\mathrm{\Sigma}_S$ is smooth at any point of $\mathrm{\Omega}_S \cap \mathrm{\Sigma}_S$, and $S$ at any point of $\mathrm{\Omega}_S$.
\end{lemma}
\begin{proof}
Let $y$ be a closed point of $\mathrm{\Omega}_S \cap \mathrm{\Sigma}_S$. Write $\pi_R \circ \pi_S(y) \in \mathrm{\Omega}_Q$ as $[q \colon \Hca \to I_Z^{\oplus 2}]$, for $[Z] \in X^{[n]}$. Using Luna's theorem as in the proof of \cref{prop:normal cone SigmaR fibres}, the smoothness of $\mathrm{\Sigma}_S$ at $y$ is deduced from the smoothness of the exceptional divisor of the blow-up of $\PP \mathrm{Hom}^\omega(W, \mathrm{E}_Z)$ along $\PP \mathrm{Hom}_1(W, \mathrm{E}_Z)$. As $\mathrm{\Sigma}_S$ is a Cartier divisor, $S$ is smooth at $y$, too. Away from $\mathrm{\Sigma}_S$, $\pi_S$ is an isomorphism and $\mathrm{\Omega}_S$ a Cartier divisor; the smoothness of $\mathrm{\Omega}_S$ implies that of $S$ at any point of $\mathrm{\Omega}_S \setminus \mathrm{\Sigma}_S$.
\end{proof}
\subsection{The third blow-up}\label{section:third blow-up}
Let 
\[
\pi_T \colon T \to S
\]
be the blow-up of $S$ along $\mathrm{\Delta}_S$. Let $\mathrm{\Delta}_T$ be the exceptional divisor, and $\mathrm{\Omega}_T$ (resp.\ $\mathrm{\Sigma}_T$) be the strict transform of $\mathrm{\Omega}_S$ (resp.\ $\mathrm{\Sigma}_S$): 
\[
\mathrm{\Delta}_T = \pi_T^{-1}(\mathrm{\Delta}_S), \quad \quad \mathrm{\Omega}_T \simeq \mathrm{Bl}_{\mathrm{\Delta}_S} \mathrm{\Omega}_S, \quad \quad \mathrm{\Sigma}_T \simeq \mathrm{Bl}_{\mathrm{\Delta}_S \cap \mathrm{\Sigma}_S} \mathrm{\Sigma}_S.
\]
For the action of $\mathrm{PGL}(N)$ on $T$,   the stable locus $T^s$ equals the semistable locus $T^{ss}$, and coincides with $\pi_T^{-1}(S^{ss}) = \pi_T^{-1}(S^s)$ (see \cite[(1.8.13)]{O'GdmsK3}). The GIT quotient of $T$ by $\mathrm{PGL}(N)$ is smooth, and provides a resolution of singularities
\begin{equation*}
\widehat{\pi}\colon \widehat{M}_n \coloneqq T^s \sslash \mathrm{PGL}(N) \to M_n
\end{equation*}
(see \cite[Proposition 1.8.3]{O'GdmsK3}). The irreducible components of the exceptional locus of $\widehat{\pi}$ are
\[
\widehat{\mathrm{\Delta}} \coloneqq \mathrm{\Delta}_T^s \sslash \mathrm{PGL}(N), \quad \widehat{\mathrm{\Sigma}} \coloneqq \mathrm{\Sigma}_T^s \sslash \mathrm{PGL}(N), \quad \widehat{\mathrm{\Omega}} \coloneqq \mathrm{\Omega}_T^s \sslash \mathrm{PGL}(N).  
\]
In this section, we shall recall global geometric descriptions of $\mathrm{\Omega}_T$, $\mathrm{\Omega}_T \cap \mathrm{\Delta}_T$, $\mathrm{\Omega}_T \cap \mathrm{\Sigma}_T$ and $\mathrm{\Omega}_T \cap \mathrm{\Delta}_T \cap \mathrm{\Sigma}_T$, and of their respective GIT quotients by $\mathrm{PGL}(N)$, which are $\widehat{\mathrm{\Omega}}$, $\widehat{\mathrm{\Omega}}\cap\widehat{\mathrm{\Delta}}$, $\widehat{\mathrm{\Omega}}\cap\widehat{\mathrm{\Sigma}}$ and $\widehat{\mathrm{\Omega}}\cap\widehat{\mathrm{\Delta}}\cap\widehat{\mathrm{\Sigma}}$.

\medskip

First, though, we state a technical lemma, which we will only need later, in \cref{sec:description of Sigma hat and Delta hat}.
\begin{lemma}\label{lem:OmegaT-DeltaT}
The scheme-theoretic intersection of $\mathrm{\Omega}_T$ and $\mathrm{\Sigma}_T$ is smooth, and 
\begin{equation*}
\mathrm{\Omega}_T \cap \mathrm{\Sigma}_T = \mathrm{Bl}_{\mathrm{\Delta}_S \cap \mathrm{\Sigma}_S}(\mathrm{\Omega}_S \cap \mathrm{\Sigma}_S).
\end{equation*}
\end{lemma}
\begin{proof}
By \cref{prop:OmegaS;OmegaS-SigmaS}, the statement is clear away from $\mathrm{\Delta}_T$. Let us prove that $\mathrm{\Omega}_T \cap \mathrm{\Sigma}_T$ is smooth at any closed point $y$ of $\mathrm{\Omega}_T \cap \mathrm{\Sigma}_T \cap \mathrm{\Delta}_T$. By the results of \cref{section:second blow-up}, $\mathrm{\Sigma}_T$, $\mathrm{\Omega}_T$ and $T$ are smooth at $y$; hence, it is enough to show that $\mathrm{\Omega}_T$ and $\mathrm{\Sigma}_T$ intersect transversely at $y$, i.e.\
\begin{equation*}
\mathrm{T}_{\mathrm{\Sigma}_T,y} + \mathrm{T}_{\mathrm{\Omega}_T,y} = \mathrm{T}_{T,y}.
\end{equation*}
As $\mathrm{\Sigma}_T$ is a divisor, it suffices to check that $\mathrm{T}_{\mathrm{\Omega}_T,y} \setminus \mathrm{T}_{\mathrm{\Sigma}_T,y}$ is non-empty. The point $x = \pi_T(y)$ lies in $\mathrm{\Delta}_S \cap \mathrm{\Sigma}_S$; as $\mathrm{\Delta}_S$ and $\mathrm{\Sigma}_S$ intersect transversely, there exists $v \in \mathrm{T}_{\mathrm{\Delta}_S,x} \setminus \mathrm{T}_{\mathrm{\Sigma}_S,x}$. By the smoothness of $\mathrm{\Omega}_T \cap \mathrm{\Delta}_T \to \mathrm{\Delta}_S$, we can find $v' \in \mathrm{T}_{\mathrm{\Omega}_T \cap \mathrm{\Delta}_T,y}$ with $\pi_{T,*}(v') = v$. In particular, $v' \in \mathrm{T}_{\mathrm{\Omega}_T,y} \setminus \mathrm{T}_{\mathrm{\Sigma}_T,y}$.

Let us prove that $\mathrm{Bl}_{\mathrm{\Delta}_S \cap \mathrm{\Sigma}_S}(\mathrm{\Omega}_S \cap \mathrm{\Sigma}_S) $ equals $\mathrm{\Omega}_T \cap \mathrm{\Sigma}_T$. As both are smooth (hence reduced) subschemes of $T$, this can be checked set-theoretically. By the universal property of blow-ups,
\[
\mathrm{Bl}_{\mathrm{\Delta}_S \cap \mathrm{\Sigma}_S}(\mathrm{\Omega}_S \cap \mathrm{\Sigma}_S)  \subset \mathrm{\Omega}_T \cap \mathrm{\Sigma}_T.
\] 
This containment is an equality away from $\mathrm{\Delta}_T$; it remains to show that it is so over $\mathrm{\Delta}_S \cap \mathrm{\Sigma}_S$ as well. For any $x \in \mathrm{\Delta}_S \cap \mathrm{\Sigma}_S$, the natural inclusion $(N_{\mathrm{\Delta}_S\cap \mathrm{\Sigma}_S/\mathrm{\Omega}_S\cap \mathrm{\Sigma}_S})_x \subset (N_{\mathrm{\Delta}_S/\mathrm{\Omega}_S})_x$ is an equality: indeed, these are fibres of vector bundles having the same rank. Therefore,
\[
\mathrm{Bl}_{\mathrm{\Delta}_S \cap \mathrm{\Sigma}_S}(\mathrm{\Omega}_S \cap \mathrm{\Sigma}_S) \cap \pi_T^{-1}(x) = \mathrm{\Omega}_T \cap \pi_T^{-1}(x) \supset  \mathrm{\Omega}_T \cap \mathrm{\Sigma}_T \cap \pi_T^{-1}(x). \qedhere
\]
\end{proof}
\subsubsection{} Let us globally describe $\mathrm{\Omega}_T$ and $\widehat{\mathrm{\Omega}}$. Keep the notation as in \ref{notation blow-ups}.
\begin{lemma}\cite[Proof of Proposition 3.2(1)]{ChoyKiem} The scheme $\mathrm{\Omega}_T$ is smooth, and we have isomorphisms
\begin{enumerate}
\item[(1)] $\mathrm{\Omega}_T \simeq [\upsilon^* (\mathrm{Bl}_{\mathrm{Bl}_2^\Tca} \mathrm{Bl}^\Tca)] \sslash \mathrm{SO}(W)$ and 
\item[(2)] $\widehat{\mathrm{\Omega}}=\mathrm{\Omega}_T^s \sslash \mathrm{PGL}(N) \simeq (\mathrm{Bl}_{\mathrm{Bl}_2^\Tca} \mathrm{Bl}^\Tca)^s \sslash \mathrm{SO}(W)$.
\end{enumerate}
\end{lemma}

\subsubsection{}\label{sec:complete conics} In \cite[Section~3.1]{O'GdmsK3I}, O'Grady provided an alternative description of $\mathrm{\Omega}_T$ (and of its intersections with $\mathrm{\Sigma}_T$ and $\mathrm{\Delta}_T$), which we now recall. 

Let $\gamma \colon \Grass^\omega(3, \Tca) \to X^{[n]}$ be the relative symplectic Grassmannian. Denote by $\Bca \subset \gamma^*\Tca$ its tautological rank-three bundle, by $\chi \colon \PP(\Bca) \to \Grass^\omega(3,\Tca)$ the associated projective bundle, and by $\Tca_\chi$ the relative tangent bundle of $\chi$. 

Over $\Grass^\omega(3,\Tca)$, consider also the projective bundle $\PP\sHom(W,\Bca) \coloneqq \PP (\Bca \otimes W^\vee)$ and the subbundles $\PP\sHom_i(W,\Bca)$ corresponding to the homomorphisms of rank at most $i$; in particular, $\PP\sHom_1(W,\Bca) \simeq \PP (W^\vee) \times \PP (\Bca)$. 

On $\PP\sHom(W,\Bca)$ (resp.\ $\PP\sHom_i(W,\Bca)$ for $i=1,2$) the group $\mathrm{SO}(W)$ acts by precomposition. The corresponding GIT quotient is the space $\PP(\mathrm{S}^2\Bca)$ (resp.\ $\PP(\mathrm{S}^2_i\Bca)$) parametrising conics (resp.\ conics of rank at most $i$) in the fibres of $\PP(\Bca^\vee)$ over $\Grass^\omega(3,\Tca)$; the quotient morphism is
\[
\PP\sHom(W,\Bca)^{ss} \to \PP(\mathrm{S}^2\Bca), \quad [\varphi] \mapsto [\varphi \circ \varphi^\vee].\footnote{Here, $\varphi^\vee \colon \Bca^\vee \to \Oca \otimes W^\vee$ denotes the transpose of $\varphi$, and we identify $W^\vee$ with $W$ via the Killing form.}
\]
Let $\mathrm{Bl}^\Bca$ be the blow-up of $\PP\sHom(W,\Bca)$ along $\PP\sHom_1(W,\Bca)$, and $E^\Bca$ its exceptional divisor. The strict transform of $\PP\sHom_2(W,\Bca)$ is the blow-up $\mathrm{Bl}_2^\Bca$ of $\PP\sHom_2(W,\Bca)$ along $\PP\sHom_1(W,\Bca)$; we denote its exceptional divisor as $E^\Bca_2$. 

The $\mathrm{SO}(W)$-action lifts to $\mathrm{Bl}^\Bca$ and $E^\Bca$; the corresponding GIT quotients are the space $\mathbf{CC}(\Bca) \coloneqq \mathrm{Bl}_{\PP (\mathrm{S}^2_1\Bca)}\PP (\mathrm{S}^2\Bca)$ parametrising complete conics in the fibres of $\PP (\Bca^\vee)$ over $\Grass^\omega(3,\Tca)$, and $\PP(\mathrm{S}^2\Tca_\chi)$, respectively. Moreover, $\mathrm{SO}(W)$ acts on $\mathrm{Bl}_2^\Bca$ and $E^\Bca_2$; the corresponding GIT quotients are $\mathrm{Bl}_{\PP (\mathrm{S}^2_1 \Bca)}\PP (\mathrm{S}^2_2\Bca)$, and the space $\PP(\mathrm{S}^2_1\Tca_\chi)$ parametrising rank-1 quadrics in the fibres of $\PP(\Tca_\chi^\vee)$ over $\PP(\Bca)$, respectively.

Consider the natural $\mathrm{SO}(W)$-equivariant forgetful morphism $\PP \sHom(W,\Bca) \to \PP\sHom^\omega(W,\Tca)$ over $X^{[n]}$. Since the schematic preimage of $\PP\sHom_1^\omega(W,\Tca)$ is $\PP \sHom_1(W,\Bca)$, we obtain an $\mathrm{SO}(W)$-equivariant morphism $\mathrm{Bl}^\Bca \to \mathrm{Bl}^\Tca$, and a morphism $E^\Bca \to E^\Tca$ between exceptional divisors. Under the morphism $\mathrm{Bl}^\Bca \to \mathrm{Bl}^\Tca$, the schematic preimage of $\mathrm{Bl}_2^\Tca$ is the Cartier divisor $\mathrm{Bl}_2^\Bca$; so, we obtain a morphism $\mathrm{Bl}^\Bca \to \mathrm{Bl}_{\mathrm{Bl}_2^\Tca}\mathrm{Bl}^\Tca$, which turns out to be an isomorphism. It identifies $\mathrm{Bl}^\Bca_2$ with the exceptional divisor of the blow-up of $\mathrm{Bl}^\Tca$ along $\mathrm{Bl}_2^\Tca$, and $E^\Bca$ with $\mathrm{Bl}_{E^\Tca_2}E^\Tca$.
\begin{proposition}\cite[(Proof of) Proposition~3.0.1]{O'GdmsK3I}\label{prop:OmegaT;OmegaT-SigmaT;DeltaT-SigmaT;OmegaT-SigmaT-DeltaT} We have isomorphisms
\begin{enumerate}
\item[(1)] $\mathrm{\Omega}_T \simeq (\upsilon^* \mathrm{Bl}^\Bca) \sslash \mathrm{SO}(W)$ and 
\item[(2)] $\widehat{\mathrm{\Omega}} = \mathrm{\Omega}_T^s \sslash \mathrm{PGL}(N) \simeq \mathbf{CC}(\Bca)$;

\medskip

\item[(3)] $\mathrm{\Omega}_T \cap \mathrm{\Delta}_T  \simeq (\upsilon^* \mathrm{Bl}^\Bca_2) \sslash \mathrm{SO}(W)$ and 
\item[(4)]$\widehat{\mathrm{\Omega}}\cap\widehat{\mathrm{\Delta}} = (\mathrm{\Omega}_T \cap \mathrm{\Delta}_T)^s \sslash \mathrm{PGL}(N) \simeq \mathrm{Bl}_{\PP(\mathrm{S}^2_1\Bca)}\PP(\mathrm{S}^2_2\Bca)$;

\medskip

\item[(5)] $\mathrm{\Omega}_T \cap \mathrm{\Sigma}_T \simeq (\upsilon^* E^\Bca) \sslash \mathrm{SO}(W)
$ and 
\item[(6)] $\widehat{\mathrm{\Omega}}\cap\widehat{\mathrm{\Sigma}} = (\mathrm{\Omega}_T \cap \mathrm{\Sigma}_T)^s \sslash \mathrm{PGL}(N) \simeq \PP(\mathrm{S}^2\Tca_\chi)$;

\medskip

\item[(7)] $\mathrm{\Omega}_T \cap \mathrm{\Sigma}_T \cap \mathrm{\Delta}_T \simeq (\upsilon^*E^\Bca_2) \sslash \mathrm{SO}(W)$ and 
\item[(8)] $\widehat{\mathrm{\Omega}}\cap\widehat{\mathrm{\Sigma}} \cap \widehat{\mathrm{\Delta}} = (\mathrm{\Omega}_T \cap \mathrm{\Sigma}_T \cap \mathrm{\Delta}_T)^s \sslash \mathrm{PGL}(N) \simeq \PP(\mathrm{S}^2_1\Tca_\chi)$.
\end{enumerate}
\end{proposition}
\subsubsection{}\label{sec:notation dual conics} Keep the notation as in \ref{sec:complete conics}. Let $\PP(\mathrm{S}^2\Bca^\vee)$ (resp.\ $\PP(\mathrm{S}^2_i\Bca^\vee)$) be the space parametrising conics (resp.\ conics of rank at most $i$) in the fibres of $\PP(\Bca)$ over $\Grass^\omega(3,\Tca)$. The schemes $\mathrm{Bl}_{\PP (\mathrm{S}^2_1\Bca^\vee)}\PP (\mathrm{S}^2\Bca^\vee)$ and $\mathbf{CC}(\Bca)$ are classically known to be isomorphic. From \cref{prop:OmegaT;OmegaT-SigmaT;DeltaT-SigmaT;OmegaT-SigmaT-DeltaT}, one can then deduce the following 
\begin{proposition}\label{prop:dual conics} The following hold true:
\begin{enumerate}
\item[(1)] $\widehat{\mathrm{\Omega}} \simeq \mathrm{Bl}_{\PP(\mathrm{S}^2_1\Bca^\vee)}\PP(\mathrm{S}^2\Bca^\vee)$;
\item[(2)] $\widehat{\mathrm{\Omega}}\cap\widehat{\mathrm{\Delta}}$ is the exceptional divisor of the blow-up of $\PP(\mathrm{S}^2\Bca^\vee)$ along $\PP(\mathrm{S}^2_1\Bca^\vee) \simeq \PP(\Bca^\vee)$;
\item[(3)] $\widehat{\mathrm{\Omega}}\cap\widehat{\mathrm{\Sigma}} \simeq \mathrm{Bl}_{\PP(\mathrm{S}^2_1 \Bca^\vee)} \PP(\mathrm{S}^2_2\Bca^\vee)$;
\item[(4)] $\widehat{\mathrm{\Omega}} \cap \widehat{\mathrm{\Sigma}}\cap \widehat{\mathrm{\Delta}}$ is the exceptional divisor of the blow-up of $\PP(\mathrm{S}^2_2\Bca^\vee)$ along $\PP(\mathrm{S}^2_1\Bca^\vee) \simeq \PP(\Bca^\vee)$.
\end{enumerate}
\end{proposition}
\subsubsection{} We conclude this section with a look at $\widehat{\mathrm{\Delta}}$. We keep the notation of \ref{sec:Grass(2,T)}; moreover, we denote as $\Aca^{\perp_\omega}$ the orthogonal complement of $\Aca \subset \alpha^*\Tca$ with respect to $\alpha^*\omega$.
\begin{proposition}\cite[Proposition 3.5.1]{O'GdmsK3I} \label{prop:Delta hat}
The morphism
\begin{equation}\label{fibration from Delta hat}
\widehat{\mathrm{\Delta}} =\mathrm{\Delta}_T^s \sslash \mathrm{PGL}(N) \to \mathrm{\Delta}_S^s \sslash \mathrm{PGL}(N) \simeq \PP (\mathrm{S}^2\Aca)
\end{equation}
is the projective bundle associated to a locally free sheaf on $\PP(\mathrm{S}^2\Aca)$. Moreover, 
\begin{equation}\label{Delta hat Omega hat}
\widehat{\mathrm{\Omega}}\cap \widehat{\mathrm{\Delta}} \simeq \PP(\mathrm{S}^2\Aca) \times_{\Grass^\omega(2, \Tca)} \PP (\Aca^{\perp_\omega}/\Aca).
\end{equation}
\end{proposition}
\begin{remark}\label{rem:Delta hat and Omega hat}
Our formulation of \cref{prop:Delta hat} seems a refinement of O'Grady's, but everything can be deduced from his work. The morphism \eqref{fibration from Delta hat} is not only a  $\PP^{2n-4}$-fibration, because the divisor $\widehat{\mathrm{\Delta}} \cap \widehat{\mathrm{\Omega}}$ of $\widehat{\mathrm{\Delta}}$ intersects each fibre in a linear hyperplane (see the proof of \cite[Lemma 3.5.4]{O'GdmsK3I}). The isomorphism \eqref{Delta hat Omega hat} does not hold only over any point $[Z] \in X^{[n]}$, because we have a global morphism $\widehat{\mathrm{\Omega}}\cap \widehat{\mathrm{\Delta}} \to \PP(\mathrm{S}^2\Aca) \times_{\Grass^\omega(2, \Tca)} \PP (\Aca^{\perp_\omega}/\Aca)$: by \cref{prop:dual conics}(2), there is indeed a morphism $\widehat{\mathrm{\Omega}}\cap \widehat{\mathrm{\Delta}} \to \PP (\Bca^\vee)\simeq \Grass(2,\Bca)$; its target is the variety parametrising partial $\omega$-isotropic flags of type $(2,3)$ in the fibres of $\Tca$ over $X^{[n]}$, and such a variety can be described also as the projective bundle $\PP(\Aca^{\perp_\omega}/\Aca)$ over $\Grass^\omega(2,\Tca)$. 
\end{remark}

\section{A construction by Markman}\label{sec:Markman}
Let us keep the notation introduced in \cref{Notation}. Consider on $X^{[n]} \times X^{[n]}$ the complexes of sheaves 
\begin{equation*}
\Eca_{ij} \coloneqq \Rder p_{X^{[n]} \times X^{[n]},*} \Rder \sHom (\Ica_i, \Ica_j)[1] \quad \text{for} \ 1 \le i,j \le 2.
\end{equation*}

The goal of this section is to study the homological properties of the $\Eca_{ij}$'s and of their derived pullbacks along $\delta$, $\delta \circ \beta$ and $b$.

\smallskip

Let us start from the easy case of $\Eca_{ii}$ for $i=1,2$.
\begin{proposition}\label{prop:E_ii} The following hold for $i=1,2$ and any $m \in \mathbb{Z}$.
\begin{enumerate}
\item[(1)] On $X^{[n]} \times X^{[n]}$ we have $\Hca^m(\Eca_{ii}) \simeq \pi_i^* \sExt_{X^{[n]}}^{m+1}(\Ica, \Ica) $. In particular, $\iota^*\Hca^{m}(\Eca_{11}) \simeq \Hca^{m}(\Eca_{22})$.
\item[(2)] On $X^{[n]}$ we have $\Hca^m( \Lder \delta^* \Eca_{ii} ) \simeq \sExt_{X^{[n]}}^{m+1}(\Ica, \Ica)$.
\item[(3)] On $E$ we have $\Hca^m (\Lder \beta^* \Lder \delta^* \Eca_{ii}) \simeq \beta^*\sExt_{X^{[n]}}^{m+1}(\Ica, \Ica)$.
\item[(4)] On $\Xca^{[n]}$ we have $\Hca^m (\Lder b^* \Eca_{ii}) \simeq (\pi_i \circ b)^* \sExt_{X^{[n]}}^{m+1}(\Ica, \Ica)$.
\end{enumerate}
\end{proposition}
\begin{proof}
The flatness of $\pi_i$ allows us to write
\begin{align*}
\Eca_{ii} &= \Rder p_{X^{[n]} \times X^{[n]},*} \Rder \sHom (\Lder \pi_{i,X}^* \Ica, \Lder \pi_{i,X}^* \Ica)[1] \\
&\simeq \Rder p_{X^{[n]} \times X^{[n]},*} \Lder \pi_{i,X}^* \Rder \sHom ( \Ica, \Ica)[1] \\
&\simeq \Lder \pi_{i}^* \Rder p_{X^{[n]},*} \Rder \sHom ( \Ica, \Ica) [1],
\end{align*}
and thence to deduce part 1. Part 2 results from the equality $\pi_i \circ \delta=\mathrm{id}_{X^{[n]}}$, and implies part 3 thanks to the flatness of $\beta$. The last claim is a consequence of the flatness of $\pi_i \circ b \colon \Xca^{[n]} \to X^{[n]}$, which holds by \cite[Ex.\ III.10.9]{Rag}.
\end{proof}

Let us move on to $\Eca_{12}$ and $\Eca_{21}$. Their study, due to Markman, relies on the following
\begin{proposition}\cite[Section 2]{Mgcrmssss}\label{prop:Markman complex}
There exists a complex 
\[
\begin{tikzcd}
V_{-1} \ar[r, "g"] & V_0 \ar[r,"f"] &V_1
\end{tikzcd}
\]
of locally free sheaves on $X^{[n]} \times X^{[n]}$ that is quasi-isomorphic to $\Eca_{12}$.
\end{proposition}

\begin{proposition}\label{prop:E:ij}
On $X^{[n]} \times X^{[n]}$ we have
\begin{equation*}
\Hca^m(\Eca_{12}) \simeq \begin{cases} \sExt_{X^{[n]} \times X^{[n]}}^1(\Ica_1, \Ica_2) &\text{if}\ m = 0 \\ \delta_* \sExt_{X^{[n]}}^2(\Ica, \Ica) \simeq \delta_*\Oca &\text{if}\ m = 1 \\
0 &\text{otherwise}.
    \end{cases}
\end{equation*}
Mutatis mutandis, the same holds for $\Eca_{21}$.
\end{proposition}
\begin{proof}
The only non-trivial statements concern the sheaves
\[
\Hca^{-1}(\Eca_{12}) = \sHom_{X^{[n]} \times X^{[n]}}(\Ica_1, \Ica_2) \quad \text{and} \quad  \Hca^{1}(\Eca_{12}) = \sExt_{X^{[n]} \times X^{[n]}}^2(\Ica_1, \Ica_2).
\]
By Cohomology and Base Change, they are both \emph{set-theoretically} supported on the diagonal of $X^{[n]} \times X^{[n]}$. By \cref{prop:Markman complex}, the former is isomorphic to the torsion-free sheaf $\ker(g)$, so it is zero. The latter is isomorphic to $\delta_*\Lca$, for $\Lca$ a line bundle on $X^{[n]}$: this fact, stated in \cite[Remark 3, Item 2]{Mgcrmssss}, was explained to me by Markman in a personal communication. Therefore, 
\[
\Lca \simeq \delta^*\delta_*\Lca \simeq \delta^*\sExt_{X^{[n]} \times X^{[n]}}^2(\Ica_1, \Ica_2) \simeq \sExt_{X^{[n]} \times X^{[n]}}^2(\delta_X^*\pi_{1,X}^*\Ica, \delta_X^*\pi_{2,X}^*\Ica) \simeq \sExt_{X^{[n]}}^2(\Ica, \Ica),
\]
where the third isomorphism comes once more from Cohomology and Base Change. We conclude that $ \Hca^{1}(\Eca_{12}) \simeq \delta_*\sExt_{X^{[n]}}^2(\Ica, \Ica)$. Since $\sExt_{X^{[n]}}^2(\Ica, \Ica)$ is isomorphic to $\Oca_{X^{[n]}} \otimes \Ho^2(X,\Oca_X)$ (see \cite[Theorem 10.2.1]{HLgmss}), we also get $ \Hca^{1}(\Eca_{12}) \simeq \delta_*\Oca$.
\end{proof}
Let us set 
\begin{align*}
\Eca & \coloneqq \Hca^0(\Eca_{12}) = \sExt_{X^{[n]} \times X^{[n]}}^1(\Ica_1, \Ica_2), \\
\Eca' &\coloneqq \Hca^0(\Eca_{21}) = \sExt_{X^{[n]} \times X^{[n]}}^1(\Ica_2, \Ica_1).
\end{align*}
\begin{lemma}\label{lem:involution and F_ij}
On $X^{[n]} \times X^{[n]}$, $\iota^*\Eca$ is isomorphic to $\Eca'$.
\end{lemma}
\begin{proof}
Since $\Lder \iota^*\Eca_{12} \simeq \Eca_{21}$, the statement results from the flatness of $\iota$.
\end{proof}
\begin{lemma}\label{lem:dual F_ij}
We have an isomorphism $\mathbb{D}(\Eca_{12}) \simeq \Eca_{21}$. Moreover, $\Eca^\vee  \simeq \Eca'$.
\end{lemma}
\begin{proof}
The first statement comes from Grothendieck-Verdier duality:
\begin{align*}
\mathbb{D}(\Eca_{12}) &= \mathbb{D}( \Rder p_{X^{[n]} \times X^{[n]},*} \Rder \sHom ( \Ica_1, \Ica_2)[1]) \\
&\simeq \Rder p_{X^{[n]} \times X^{[n]},*} \Rder \sHom ( \Rder \sHom (\Ica_1, \Ica_2)[1], \Oca [2] ) \\
&\simeq \Rder p_{X^{[n]} \times X^{[n]},*} \Rder \sHom ( \Ica_2, \Ica_1 )[1] = \Eca_{21}.
\end{align*}
For the second claim, use the spectral sequence $E^{p,q}_2 = \Rder ^p \sHom(\Hca^{-q}(\Eca_{12}),\Oca) \Rightarrow \Hca^{p+q}(\mathbb{D}(\Eca_{12}))$: again by Grothendieck-Verdier duality, $\Rder^{p} \sHom (\delta_*\Oca,\Oca)$ vanishes for every $p < \mathrm{dim} \ X^{[n]}$. 
\end{proof}

\begin{remark}
By \cref{prop:Markman complex},  $\Eca_{21}$ is quasi-isomorphic to the complex $V_1^\vee \xrightarrow{f^\vee} V_0^\vee \xrightarrow{g^\vee} V_{-1}^\vee$.
\end{remark}

\begin{proposition}\label{prop:F_ij pullback under delta}
On $X^{[n]}$ we have $\Hca^m(\Lder \delta^* \Eca_{ij}) \simeq \sExt_{X^{[n]}}^{m+1}(\Ica, \Ica)$ for any $m \in \mathbb{Z}$.
\end{proposition}
\begin{proof} By the flatness of $p_{X^{[n]} \times X^{[n]}}$, we have
\begin{align*}
\Lder \delta^* \Eca_{ij} &= \Lder \delta^* \Rder p_{X^{[n]} \times X^{[n]},*} \Rder \sHom (\Ica_i, \Ica_j)[1] \\
&\simeq \Rder p_{X^{[n]},*} \Lder \delta^*_X \Rder \sHom (\Lder \pi_{i, X}^* \Ica, \Lder \pi_{j, X}^* \Ica)[1] \\
&\simeq \Rder p_{X^{[n]},*} \Rder \sHom (\Lder \delta^*_X \Lder \pi_{i, X}^* \Ica, \Lder \delta^*_X \Lder \pi_{j, X}^* \Ica)[1] \\
&= \Rder p_{X^{[n]},*} \Rder \sHom ( \Ica, \Ica )[1]. \qedhere
\end{align*}
\end{proof}

\begin{corollary}\label{cor:F_ij pullback under delta circ beta}
On $E$ we have $\Hca^m (\Lder \beta^* \Lder \delta^* \Eca_{ij}) \simeq \beta^*\sExt_{X^{[n]}}^{m+1}(\Ica, \Ica)$  for any $m \in \mathbb{Z}$. In particular, $\Hca^0 (\Lder \varepsilon^* \Lder b^* \Eca_{ij})$ is the pullback along $\beta$ of the tangent bundle $\Tca$ of $X^{[n]}$.
\end{corollary}
\begin{proof}
The claim follows immediately from \cref{prop:F_ij pullback under delta} thanks to the flatness of $\beta$.
\end{proof}

\begin{proposition}\label{prop:Fij pullback beta}
On $\Xca^{[n]}$ we have 
\[
\Hca^m ( \Lder b^* \Eca_{12}) \simeq \begin{cases} \ker (b^* f)/\im (b^*g) &\text{if}\ m = 0 \\
b^*\delta_* \sExt_{X^{[n]}}^2(\Ica, \Ica) \simeq \varepsilon_*\Oca_E &\text{if}\ m = 1 \\
0 &\text{otherwise}.
\end{cases}
\]
Mutatis mutandis, the same holds for $\Hca^m ( \Lder b^* \Eca_{21})$.
\end{proposition}
\begin{proof}
By \cref{prop:Markman complex}, $\Lder b^* \Eca_{12}$ is given by $b^* V_{-1} \xrightarrow{b^*g} b^* V_0 \xrightarrow{b^*f} b^* V_1$.
Since $b^* g$ is a generically injective homomorphism between locally free sheaves, it is injective, hence $\Hca^{-1} ( \Lder b^* \Eca_{12})$ vanishes.
For $m=1$, we have by \cref{prop:E:ij}
\[
\Hca^{1} ( \Lder b^* \Eca_{12}) = \coker (b^* f) \simeq b^* \coker (f) \simeq b^* \delta_* \sExt_{X^{[n]}}^2(\Ica, \Ica). \qedhere
\]
\end{proof}

\subsection{A vector bundle on $\Xca^{[n]}$}\label{sec:construction by Markman} 
In the complex of locally free sheaves 
\begin{equation}\label{pullback Markman complex to blow up}
\begin{tikzcd}
b^* V_{-1} \ar[r,"b^*g"] & b^* V_0 \ar[r, "b^*f"] & b^* V_1,
\end{tikzcd}
\end{equation}
the morphisms $b^*g$ and $b^*f$ are of constant rank on $\Xca^{[n]} \setminus E$, but not on the whole of the reduced scheme $\Xca^{[n]}$. As explained in \cite[Section 4]{MBBccc}, we can consider instead of \eqref{pullback Markman complex to blow up} a complex of locally free sheaves 
\[
\begin{tikzcd}
U_{-1} \ar[r,"\widetilde{g}"] & b^* V_0 \ar[r,"\widetilde{f}"] & U_1,
\end{tikzcd}
\]
whose morphisms $\widetilde{g}$ and $\widetilde{f}$ are of constant rank on $\Xca^{[n]}$ and agree on $\Xca^{[n]} \setminus E$ with $b^*g$ and $b^*f$, respectively. The construction, whose outcome will be important for this paper, goes as follows.

\smallskip

Define $U_1$ as the image of $b^*f$, and denote by $\widetilde{f} \colon b^*V_0 \to U_1$ the surjective homomorphism induced by $b^*f$. We have two short exact sequences
\begin{equation}
\begin{tikzcd}[row sep=0.3em]
0 \ar[r] & U_1 \ar[r] & b^*V_1 \ar[r] & \varepsilon_*\Oca_E \ar[r] & 0, \label{definition U_1} \\
0 \ar[r] & \ker(b^*f) \ar[r] & b^*V_0 \ar[r, "\widetilde{f}"] & U_1 \ar[r] & 0.
\end{tikzcd}
\end{equation}
From the first (which is given by \cref{prop:Fij pullback beta}), we read the local freeness of $U_1$, because $E$ is a Cartier divisor. The second implies that $\ker(\widetilde{f})$ equals $\ker(b^*f)$ and that it is a subbundle of $b^*V_0$. 

Set $U'_{-1} \coloneqq U_1^\vee$ and $\widetilde{f^\vee} \coloneqq (\widetilde{f})^\vee$. The functor $\mathbb{D}$ applied to \eqref{definition U_1} yields the short exact sequences
\[
\begin{tikzcd}[row sep=0.3em]
0 \ar[r] & b^*V_1^\vee \ar[r] & U'_{-1} \ar[r] & \varepsilon_*(\varepsilon^*\Oca(E)) \ar[r] & 0, \\
0 \ar[r] & U'_{-1}  \ar[r,"\widetilde{f^\vee}"] & b^*V_0^\vee \ar[r] & \ker(b^*f)^\vee \ar[r] & 0
\end{tikzcd}
\]
(recall that $\mathbb{D} \circ \varepsilon_* (-) \simeq \varepsilon_* \Rder \sHom(-,\varepsilon^*\Oca(E))[-1]$ by Grothendieck-Verdier duality). In particular, $U'_{-1}$ is a subbundle of $b^*V^\vee_0$ via $\widetilde{f^\vee}$.

\smallskip

Now, replace \eqref{pullback Markman complex to blow up} with its dual $b^*V_1^\vee \xrightarrow{b^*f^\vee} b^*V_0^\vee \xrightarrow{b^*g^\vee} b^*V_{-1}^\vee$ and repeat this construction. We get a locally free sheaf $U_1' \coloneqq \im(b^*g^\vee)$ with a surjective morphism $\widetilde{g^\vee} \colon b^*V^\vee_0 \to U'_1$. Set $U_{-1} \coloneqq (U'_1)^\vee$ and $\widetilde{g} \coloneqq (\widetilde{g^\vee})^\vee \colon U_{-1} \to b^*V_0$. Then, $U_{-1}$ is a subbundle of $b^*V_0$ via $\widetilde{g}$ and sits into a short exact sequence
\begin{equation}\label{b^*V_{-1} vs U_{-1}}
\begin{tikzcd}
0 \ar[r] & b^*V_{-1} \ar[r] & U_{-1} \ar[r] & \varepsilon_*(\varepsilon^*\Oca(E)) \ar[r] & 0.
\end{tikzcd}
\end{equation}
The morphism $\widetilde{f} \circ \widetilde{g}$ between the locally free sheaves $U_{-1}$ and $U_1$ vanishes: indeed, it vanishes on $\Xca^{[n]} \setminus E$, where $\widetilde{g}$ and $\widetilde{f}$ agree with $b^*g$ and $b^*f$, respectively. Analogously, $\widetilde{g^\vee} \circ \widetilde{f^\vee}=0$. Thus, we have obtained two complexes of vector bundles
\[
\begin{tikzcd}
U_{-1} \ar[r,"\widetilde{g}"] & b^*V_0 \ar[r,"\widetilde{f}"] & U_1 &\text{and}& U_{-1}' \ar[r,"\widetilde{f^\vee}"] & b^*V_0^\vee \ar[r,"\widetilde{g^\vee}"] & U_1'.
\end{tikzcd}
\]
By construction, they are dual to each other, and quasi-isomorphic to their cohomology sheaves
\[
V \coloneqq \ker(\widetilde{f})/\im(\widetilde{g}) \quad \text{and} \quad  V' \coloneqq \ker(\widetilde{g^\vee})/\im(\widetilde{f^\vee}),
\]
respectively. Here we collect the main properties of $V$.
\begin{theorem}[{\cf \cite[Proposition 4.1]{MBBccc}}]\label{thm:properties of V} The following hold true.
\begin{enumerate}
\item[(1)] The sheaves $V$ and $V'$ are locally free, and $V^\vee \simeq V'$.
\item[(2)] We have $V \simeq [\ker(b^*f)/\im(b^*g)]_\tf$ and $V' \simeq [\ker(b^*f^\vee)/\im(b^*g^\vee)]_\tf$.
\item[(3)] Let $\Lca$ be the tautological line subbundle of $ \beta^*\Tca$ on $E = \PP(\Tca)$ and $\Lca^{\perp_\omega}$ its orthogonal complement with respect to the symplectic form $\beta^*\omega$. Then $V|_E \simeq \Lca^{\perp_\omega }/\Lca$, $V'|_E\simeq \Lca^{\perp_\omega }/\Lca$.
\item[(4)] We have $(b^*\Eca)_\tf \simeq V(-E)$ and $[b^*(\Eca')]_\tf \simeq V'(-E)$.
\item[(5)] $\widetilde{\iota}^*V \simeq V'$ and $\widetilde{\iota}^*(V') \simeq V$.
\end{enumerate}
\end{theorem}
\begin{proof}
We prove only the statements concerning $V$: for $V'$ one reasons identically.

Part 1 holds true, because both  $U_{-1} \simeq \im(\widetilde{g})$ and $\ker(\widetilde{f})$ are vector subbundles of $b^*V_0$.

Let us prove part 2. Consider the commutative diagram with exact rows
\[
\begin{tikzcd}
0 \ar[r] & b^*V_{-1} \ar[r] \ar[d] & \ker(b^*f) \ar[r] \ar[d,equal] & \ker(b^*f)/\im(b^*g) \ar[d] \ar[r] & 0 \\
0 \ar[r] & U_{-1} \ar[r] & \ker(\widetilde{f}) \ar[r] & V \ar[r] & 0.
\end{tikzcd}
\]
By the snake lemma, we deduce from \eqref{b^*V_{-1} vs U_{-1}} the short exact sequence
\begin{equation}\label{V=ker(b*f)/im(b*g)/ torsion}
\begin{tikzcd}
0 \ar[r] & \varepsilon_*\Lca \ar[r] & \ker(b^*f)/\im(b^*g) \ar[r] & V \ar[r] & 0.
\end{tikzcd}
\end{equation}
As $\varepsilon_*\Lca $ is a torsion sheaf, while $V$ is torsion-free, we get the desired isomorphism.

Let us show part 3. Since $V$ is locally free, the restriction of \eqref{V=ker(b*f)/im(b*g)/ torsion} to $E$ remains exact:
\[
\begin{tikzcd}
0 \ar[r] & \Lca \ar[r] & \varepsilon^*\left[ \ker(b^*f)/\im(b^*g) \right]\ar[r] & \varepsilon^*V \ar[r] & 0.
\end{tikzcd}
\]
To compute the middle term, we make use of the spectral sequence $E^{p,q}_2 = \Lder ^p \varepsilon^* (\Hca^q(\Lder b^*\Eca_{12})) \Rightarrow \Lder^{p+q} \varepsilon^*(\Lder b^* \Eca_{12})$; by \cref{prop:Fij pullback beta} and \cite[Corollary 11.4(i)]{huybrechts2006fourier} it degenerates at the second page. In particular, we deduce from $\Lder^1\varepsilon^*(\varepsilon_*\Oca_E) \simeq \Lca^\vee$ and \cref{cor:F_ij pullback under delta circ beta} a short exact sequence
\[
\begin{tikzcd}
0 \ar[r] & \varepsilon^*\left[ \ker(b^*f)/\im(b^*g) \right] \ar[r] & \beta^*\Tca \ar[r] & \Lca^\vee \ar[r] & 0.
\end{tikzcd}
\]
Since $\mathrm{Hom}(\beta^* \Tca,\Lca^\vee) \simeq \mathrm{Hom}(\Tca,\Tca^\vee) \simeq \mathbb{C}$ by the stability of $\Tca$, $\varepsilon^*\left[ \ker(b^*f)/\im(b^*g) \right]$ must be isomorphic to $\Lca^{\perp_\omega}$. We conclude that $V|_E \simeq \Lca^{\perp_\omega}/\Lca$.

Let us move on to part 4. The natural transformation $\mathrm{id} \to \varepsilon_* \varepsilon^*$ applied to \eqref{V=ker(b*f)/im(b*g)/ torsion} gives the commutative diagram with exact rows
\[
\begin{tikzcd}
0 \ar[r] & \varepsilon_*\Lca \ar[r] \isoarrow{d} & \ker(b^*f)/\im(b^*g) \ar[r] \ar[d,"\nu"] & V \ar[r] \ar[d] & 0 \\
0 \ar[r] & \varepsilon_* \varepsilon^*(\varepsilon_*\Lca) \ar[r] & \varepsilon_* \varepsilon^*\left[ \ker(b^*f)/\im(b^*g) \right] \ar[r] & \varepsilon_* \varepsilon^*V \ar[r] & 0. 
\end{tikzcd}
\]
By the snake lemma, $\ker(\nu) \simeq V(-E)$. We compute now $\ker(\nu)$ in a different way. Consider the spectral sequence $E^{p,q}_2 = \Lder ^p b^* (\Hca^q(\Eca_{12})) \Rightarrow \Lder^{p+q} b^* \Eca_{12}$, which degenerates at the third page. Thanks to \cref{prop:Fij pullback beta} and the isomorphisms $E^{p,1}_2 \simeq \varepsilon_*(\Kahler^{-p}_\beta \otimes \Lca^{\otimes p})$ (see \cite[Proposition 11.12]{huybrechts2006fourier}), we get an exact sequence
\[
\begin{tikzcd}
0 \ar[r] & \varepsilon_*{[ \Kahler^2_\beta \otimes (\Lca^\vee)^{\otimes 2}]} \ar[r] & b^* \Eca \ar[r] & \ker(b^* f)/ \im (b^* g) \ar[r,"\mu"] & \varepsilon_* (\Kahler^1_\beta \otimes \Lca^\vee) \ar[r] & 0.
\end{tikzcd}
\]
If we prove that $\ker(\mu) = \ker(\nu)$, we are done: indeed, this would identify the torsion-free sheaf $\ker(\nu) \simeq V(-E)$ with the quotient of $b^* \Eca$ by the torsion sheaf $ \varepsilon_*[ \mathrm{\Omega}^2_\beta \otimes (\Lca^\vee)^{\otimes 2}]$.
The natural transformation $\mathrm{id} \to \varepsilon_* \varepsilon^*$ applied to the epimorphism $\mu$ yields the commutative square
\[
\begin{tikzcd}
\ker(b^* f)/ \im (b^* g) \ar[r, two heads,"\mu"] \ar[d,"\nu"] & \varepsilon_* (\Kahler^1_\beta \otimes \Lca^\vee) \ar[d,] \\
\varepsilon_* \varepsilon^*{[\ker(b^* f)/ \im (b^* g)]} \ar[r,"\varepsilon_* \varepsilon^*\mu"] & \varepsilon_* \varepsilon^*\varepsilon_* (\Kahler^1_\beta \otimes \Lca^\vee).
\end{tikzcd}
\]
As the vertical arrow on the right is an isomorphism, it suffices to show that $\varepsilon^*\mu$ (and so $\varepsilon_* \varepsilon^*\mu$) is an isomorphism. Now, $\varepsilon^*[\ker(b^* f)/ \im (b^* g)]$ and $ \varepsilon^*\varepsilon_* (\Kahler^1_\beta \otimes \Lca^\vee) \simeq \Kahler^1_\beta \otimes \Lca^\vee$ are both isomorphic to $\Lca^{\perp_\omega}$. Thus, $\varepsilon^* \mu$ is a surjective endomorphism of the locally free sheaf $\Lca^{\perp_\omega}$, so it is an isomorphism.

It remains to prove part 5. As $\Xca^{[n]}$ is integral and $\widetilde{\iota}$ is flat, $\widetilde{\iota}^*$ commutes with $(-)_\tf$; moreover $\widetilde{\iota}^*\Oca(E) \simeq \Oca(E)$. Using part 4 and \cref{lem:involution and F_ij}, we get
\begin{equation*}
\widetilde{\iota}^* V \simeq \widetilde{\iota}^*[(b^* \Eca(E))_\tf] \simeq [\widetilde{\iota}^* (b^* \Eca(E))]_\tf \simeq [(b^* \iota^* \Eca)(E)]_\tf \simeq [(b^*\Eca')(E)]_\tf \simeq V'. \qedhere
\end{equation*}
\end{proof}

\section{The exceptional divisor of the blow-up of $R^{ss}$ along $\mathrm{\Sigma}_R^{ss}$}\label{sec:The exceptional divisor of the second blow-up}

In this section, technical core of the paper, we compute the normal cone $C_{\mathrm{\Sigma}_R^{ss}/R^{ss}}$ to $\mathrm{\Sigma}_R^{ss}$ in $R^{ss}$. 
Inspired by the computation of the normal cone to $\mathrm{\Omega}_Q$ in $Q^{ss}$ done in \cite[Lemma 4.1(2)]{ChoyKiem}, we achieve our goal in two steps. First, we determine the conormal sheaf to $\mathrm{\Sigma}_R^{ss}$ in $R^{ss}$. Here, the crucial ingredient is the global description \eqref{cotangent sheaf Quot scheme} of the cotangent sheaf of $Q^{ss}$ as a relative Ext-sheaf: this enables a comparison between a long exact sequence of relative Ext-sheaves on the one hand, and exact sequences of differentials and conormal sheaves on the other. Second, we compute $C_{\mathrm{\Sigma}_R^{ss}/R^{ss}}$.

Building on these computations, we globally describe the exceptional divisor of the blow-up of $R^{ss}$ along $\mathrm{\Sigma}_R^{ss}$, and its GIT quotient by $\mathrm{PGL}(N)$.

\medskip

Recall that we denoted by $\pi_R \colon R \to Q^{ss}$ the blow-up of $Q^{ss}$ along $\mathrm{\Omega}_Q$, by $\mathrm{\Omega}_R$ its exceptional divisor and by $\mathrm{\Sigma}_R$ the strict transform of $\mathrm{\Sigma}_Q$. The restriction $p \colon \mathrm{\Sigma}_R \to \mathrm{\Sigma}_Q$ of $\pi_R$ is the blow-up of $\mathrm{\Sigma}_Q$ along $\mathrm{\Omega}_Q$, and $\mathrm{\Sigma}_R \cap \mathrm{\Omega}_R$ is its exceptional divisor. Let $\jmath \colon \mathrm{\Omega}_R \to R$ and $\imath \colon \mathrm{\Sigma}_R \cap \mathrm{\Omega}_R \to \mathrm{\Sigma}_R$ be the inclusion morphisms; we shall abuse the notation, and keep the same letters when passing to the respective semistable loci.

We stick to the notation introduced in \cref{Notation}, and we enrich it with the divisorial imbedding $\zeta \colon D \to U$. We denote as $h \colon U \to \mathrm{\Sigma}_R^{ss}$ the morphism constructed in the proof of \cref{prop:global description of SigmaR}, and as $\eta \colon D \to \mathrm{\Sigma}_R^{ss} \cap \mathrm{\Omega}_R$ its restriction to $D$; they are the quotient morphisms for the free action of $\mathrm{PO}(2)$ on the smooth $U$ and $D$, respectively; in particular, $h$ and $\eta$ are flat. 

This geometric setting is summarised in the following commutative diagram:
\[
\begin{tikzcd}[row sep=1.4em, column sep = 1.6em, scale cd=0.8]
& & D \ar[dd, "\zeta", hook] \ar[lldd] \ar[rrdd, "\eta"] & & &  \\[3 ex]
& & & & & \\
E \ar[rd, "\varepsilon", hook] \ar[dd,"\beta"] & & U\ar[ld, "u"'] \ar[rd, "h"] & & \mathrm{\Sigma}_R^{ss} \cap \mathrm{\Omega}_R \ar[ld, hook', "\imath"] \ar[rd, hook] &\\
 & \Xca^{[n]} \ar[dd, "b"] \ar[dr] & & \mathrm{\Sigma}_R^{ss} \ar[ld] \ar[dd] \ar[rd, hook] & & \mathrm{\Omega}_R \ar[ld, hook', "\jmath"] \ar[dd] \\
X^{[n]} \ar[rd,"\delta", hook] & & (X^{[n]})^{[2]} \ar[dd] & & R &\\
& X^{[n]} \times X^{[n]} \ar[rd] & & \mathrm{\Sigma}_Q \ar[ld] \ar[rd, hook] & & \mathrm{\Omega}_Q \ar[ld, hook'] \ar[ll, hook'] \\
& & S^2X^{[n]} & &  Q^{ss}. \ar[from=uu, crossing over, "\pi_R" {yshift=15pt}] &
\end{tikzcd}
\]
As we have seen in the proof of \cref{prop:global description of SigmaR}, the composite morphism
\begin{equation}\label{morphism from U to Quot}
\begin{tikzcd}
U \ar[r,"h"] & \mathrm{\Sigma}_R^{ss} \ar[r,hook] &  R \ar[r,"\pi_R"] & Q^{ss}
\end{tikzcd}
\end{equation}
is determined by the surjection \eqref{quotient on U times X}; thus, up to an isomorphism of $\Oca_{U \times X}$-modules
\begin{equation}\label{isomorphism SigmaR FU}
\begin{tikzcd}
u_X^* b_X^* (\Ica_1 \oplus \Ica_2) \otimes p_U^* \Oca_u(1) \ar[r,"\sim"] & \Fca_U,
\end{tikzcd}
\end{equation}
the surjection \eqref{quotient on U times X} coincides with the pullback to $U \times X$ of the tautological epimorphism on $Q^{ss} \times X$ considered in \cref{sec:cotangent sheaf}. This pullback sits into a short exact sequence
\[
\begin{tikzcd}
0 \ar[r] & \Kca_U \ar[r] & \Oca(-k)^{\oplus N} \ar[r] & \Fca_U \ar[r] & 0.
\end{tikzcd}
\]
Applying the functor $p_{U,*}\sHom(\Fca_U , -)$, we get a long exact sequence in cohomology
\[
\begin{tikzcd}[row sep=0.4em]
& \sExt_U^1(\Fca_U,\Oca(-k)^{\oplus N}) \ar[r] & \sExt_U^1(\Fca_U,\Fca_U) \ar[r] & {\color{white} 0} \\
\sExt_U^2(\Fca_U,\Kca_U) \ar[r] & \sExt_U^2(\Fca_U,\Oca(-k)^{\oplus N}) \ar[r] & \sExt_U^2(\Fca_U,\Fca_U) \ar[r] & 0.
\end{tikzcd}
\]
Lemmata \ref{lem:SigmaR ExtEE} and \ref{SigmaR-Ext(E,O(-k))} describe the sheaves involved in this sequence. We keep the notation and use results of \cref{sec:Markman} about the complex \eqref{pullback Markman complex to blow up} and its dual.
\begin{lemma}\label{lem:SigmaR ExtEE} 
We have 
isomorphisms
\begin{equation*}
    \sExt_U^m(\Fca_U,\Fca_U) \simeq \begin{cases} u^* \left[(\ker (b^* g^\vee)/\im (b^* f^\vee)) \oplus (\ker (b^*f)/\im (b^*g))\right] \oplus u^*b^*\Kahler^1_{X^{[n]} \times X^{[n]} } &\text{if}\ m = 1 \\[0.5em]
    (b \circ u)^*[ \oplus_{i=1}^2 \pi_i^* \sExt_{X^{[n]}}^2(\Ica,\Ica)] \oplus (b \circ u)^* [ \delta_* \sExt_{X^{[n]}}^2(\Ica,\Ica)^{\oplus 2}] &\text{if}\ m=2.
    \end{cases}
\end{equation*}
\end{lemma}
\begin{proof} 
We have
\begin{align*}
\Rder & p_{U,*} \Rder \sHom (\Fca_U, \Fca_U)  && \\
&\simeq \Rder p_{U,*} \Rder \sHom ((b \circ u)_X^* ( \Ica_1 \oplus \Ica_2 ), (b \circ u)_X^* ( \Ica_1 \oplus \Ica_2 ))  && \text{(isomorphism \eqref{isomorphism SigmaR FU})} \\
&\simeq \Rder p_{U,*} \Rder \sHom ( \Lder (b \circ u)_X^* ( \Ica_1 \oplus \Ica_2 ), \Lder (b \circ u)_X^* ( \Ica_1 \oplus \Ica_2 ))  && \text{(flatness of $u$, $\pi_i \circ b$ and $\pi_i$)} \\
&\simeq \bigoplus _{1 \le i,j \le 2} \Rder p_{U,*} \Lder (b \circ u)_X^* \Rder \sHom (\Ica_i, \Ica_j)  && \\
&\simeq \bigoplus _{1 \le i,j \le 2} \Lder (b \circ u)^* \Rder p_{X^{[n]} \times X^{[n]},*}  \Rder \sHom (\Ica_i, \Ica_j)  && \text{(flatness of $p_{X^{[n]} \times X^{[n]}}$)} \\
& \simeq \bigoplus _{1 \le i,j \le 2} \Lder u^* \Lder b^* \Rder p_{X^{[n]} \times X^{[n]},*}  \Rder \sHom (\Ica_i, \Ica_j).  &&
\end{align*}
By the flatness of $u$, the conclusion now follows from \cref{prop:E_ii}(4) and \cref{prop:Fij pullback beta}.
\end{proof}
\begin{lemma}\label{lem:SigmaR-ExtFF}
We have an isomorphism $\sExt_U^1(\Fca_U,\Fca_U)_\tf \simeq u^*(V' \oplus V) \oplus u^*b^* \Kahler^1_{X^{[n]} \times X^{[n]}}$.
\end{lemma}
\begin{proof} As $u \colon U \to \Xca^{[n]}$ is a flat morphism of integral schemes, $u^*$ commutes with $(-)_\tf$. The statement now results from \cref{lem:SigmaR ExtEE} by \cref{thm:properties of V}(2).
\end{proof}
\begin{lemma}\label{SigmaR-Ext(E,O(-k))}
We have
\begin{equation*}
    \sExt_U^m(\Fca_U,\Oca(-k)^{\oplus N}) \simeq \begin{cases} u^*\sHom (\Oca^{\oplus N}, p_{\Xca^{[n]},*} b_X^*[(\Ica_1 \oplus \Ica_2)(k)])^\vee \otimes \Oca_u(-1) &\text{if}\ m = 2 \\[0.1em]
    0 &\text{otherwise}.
    \end{cases}
\end{equation*}
\end{lemma}
\begin{proof}
We have
\begin{align*}
\Rder &p_{U,*} \Rder \sHom ( \Fca_U, \Oca(-k)^{\oplus N}) &&\\
&\simeq \Rder p_{U,*} \Rder \sHom ( \Rder \sHom (\Oca(-k)^{\oplus N}, \Fca_U), \Oca ) && \\
&\simeq \mathbb{D}[\Rder p_{U,*} \Rder \sHom (\Oca(-k)^{\oplus N}, \Fca_U)][-2] && \text{(Grothendieck-Verdier duality)} \\
&\simeq \mathbb{D}[\Rder p_{U,*} \Rder \sHom (\Oca(-k)^{\oplus N}, u_X^* b_X^* (\Ica_1 \oplus \Ica_2) \otimes p_U^* \Oca_u(1))][-2] && \text{(isomorphism \eqref{isomorphism SigmaR FU})}\\
&\simeq \mathbb{D}[\Rder p_{U,*} \Rder \sHom (\Oca^{\oplus N}, u_X^* b_X^* [(\Ica_1 \oplus \Ica_2)(k)])\otimes \Oca_u(1)][-2] && \\
&\simeq \mathbb{D}[\Rder p_{U,*} u_X^* \Rder \sHom(\Oca^{\oplus N}, b_X^* [(\Ica_1 \oplus \Ica_2)(k)])] \otimes \Oca_u(-1)[-2] && \\
&\simeq u^* \mathbb{D}[\Rder p_{\Xca^{[n]},*} \Rder \sHom ( \Oca^{\oplus N}, b_X^*[(\Ica_1 \oplus \Ica_2)(k)])] \otimes \Oca_u(-1)[-2]&& \text{(flat base change)}\\
&\simeq u^* \mathbb{D}[\Rder \sHom ( \Oca^{\oplus N}, \Rder p_{\Xca^{[n]},*} b_X^*[(\Ica_1 \oplus \Ica_2)(k)])] \otimes \Oca_u(-1)[-2] && \\
&\simeq u^* [\sHom ( \Oca^{\oplus N},p_{\Xca^{[n]},*} b_X^*[(\Ica_1 \oplus \Ica_2)(k)])]^\vee \otimes \Oca_u(-1)[-2]. && 
\end{align*}
The last equality holds because $\Rder p_{\Xca^{[n]},*} b_X^*[(\Ica_1 \oplus \Ica_2)(k)] \simeq p_{\Xca^{[n]},*} b_X^*[(\Ica_1 \oplus \Ica_2)(k)]$ is a locally free sheaf. We deduce that $\sExt_U^m(\Fca_U,\Oca(-k)^{\oplus N})$ vanishes for every $m \neq 2$. For $m=2$, we get
\[
\sExt_U^2(\Fca_U,\Oca(-k)^{\oplus N}) \simeq u^* \sHom (\Oca^{\oplus N}, p_{\Xca^{[n]},*} b_X^*[(\Ica_1 \oplus \Ica_2)(k)])^\vee \otimes \Oca_u(-1). \qedhere
\]
\end{proof}
In particular, we have obtained the exact complex
\begin{equation*}
\begin{tikzcd}[column sep=0.7em]
0 \ar[r] &\sExt_U^1(\Fca_U,\Fca_U) \ar[r] &\sExt_U^2(\Fca_U,\Kca_U) \ar[r] &\sExt_U^2(\Fca_U,\Oca(-k)^{\oplus N}) \ar[r] &\sExt_U^2(\Fca_U,\Fca_U) \ar[r] &0.
\end{tikzcd}
\end{equation*}
Since $\sExt_U^2(\Fca_U,\Oca(-k)^{\oplus N})$ is locally free, we have also the exact complex
\begin{equation}\label{SigmaR-Ext complex}
\begin{tikzcd}[column sep=0.55em]
0 \ar[r] &\sExt_U^1(\Fca_U,\Fca_U)_\tf \ar[r] &\sExt_U^2(\Fca_U,\Kca_U)_\tf \ar[r] &\sExt_U^2(\Fca_U,\Oca(-k)^{\oplus N}) \ar[r] &\sExt_U^2(\Fca_U,\Fca_U) \ar[r] &0.
\end{tikzcd}
\end{equation}

\subsection{Computation of the conormal sheaf}
In this section, we compute the conormal sheaf to $\mathrm{\Sigma}_R^{ss}$ in $R^{ss}$, or rather its pullback along $h$. By \cref{prop:normal cone SigmaR fibres}, $R^{ss}$ is normally flat along $\mathrm{\Sigma}_R^{ss}$; in particular, the conormal sheaf to $\mathrm{\Sigma}_R^{ss}$ in $R^{ss}$ is a locally free $\Oca_{\mathrm{\Sigma}_R^{ss}}$-module, hence isomorphic to the dual of the normal bundle $\Nca_{\mathrm{\Sigma}_R^{ss}/R^{ss}}$ to $\mathrm{\Sigma}_R^{ss}$ in $R^{ss}$.

\smallskip

As a first step, we relate $\Nca^\vee_{\mathrm{\Sigma}_R^{ss}/R^{ss}}$ to the cotangent sheaf of $Q^{ss}$. This requires the following
\begin{lemma}\label{lem:ses of differentials for blow up}
We have exact sequences
\begin{gather}
\begin{tikzcd}[ampersand replacement=\&]
\pi_R^* \Kahler^1_{Q^{ss}} \ar[r] \& \Kahler^1_R \ar[r] \& \jmath_*\Kahler^1_{\mathrm{\Omega}_R/\mathrm{\Omega}_Q} \ar[r] \& 0,
\end{tikzcd} \label{ses-blow-up:eq1}\\
\begin{tikzcd}[ampersand replacement=\&]
p^* \Kahler^1_{\mathrm{\Sigma}_Q} \ar[r] \& \Kahler^1_{\mathrm{\Sigma}_R} \ar[r] \& \imath_*\Kahler^1_{\mathrm{\Omega}_R \cap \mathrm{\Sigma}_R/\mathrm{\Omega}_Q} \ar[r] \& 0.
\end{tikzcd}\label{ses-blow-up:eq2}
\end{gather}
\end{lemma}
\begin{proof}
Let us prove that the first sequence is exact. It suffices to show that it is exact at every closed point $y \in R$. Clearly, this holds true if $y \notin \mathrm{\Omega}_R$. For $y \in \mathrm{\Omega}_R$, let $[q] =  \pi_R(y) \in \mathrm{\Omega}_Q$ and pick $\Vca \subset Q^{ss}$ an étale slice in $[q]$. Set $\Wca \coloneqq \Vca \cap \mathrm{\Omega}_Q$. Denote by $\sigma \colon \widetilde{\Vca} \to \Vca$ the blow-up of $\Vca$ along $\Wca$ and by $i \colon \widetilde{\Wca} \to \widetilde{\Vca}$ the imbedding of the exceptional divisor. Since $(\mathrm{PGL}(N) \times \Vca) \sslash \mathrm{St}([q]) \to Q^{ss}$ is étale and the fibration $(\mathrm{PGL}(N) \times \Vca) \sslash \mathrm{St}([q]) \to \mathrm{PGL}(N)/\mathrm{St}([q])$ with fibre $\Vca$ is \'etale-locally trivial, the exactness of \eqref{ses-blow-up:eq1} at $y$ can be deduced from the exactness of the sequence
\[
\begin{tikzcd}
\sigma^* \Kahler^1_{\Vca} \ar[r] & \Kahler^1_{\widetilde{\Vca}} \ar[r] & i_*\Kahler^1_{\widetilde{\Wca}/\Wca} \ar[r] & 0
\end{tikzcd}
\]
at any closed point of $\widetilde{\Wca}$ that lies over $[q]$. Using the isomorphisms \eqref{Q local description} and \eqref{OmegaQ local description}, we can further reduce to the case of an intersection of quadrics in an affine space, and its blow-up at the origin; this case is handled in \cref{lemma:relative differentials blow up}, which we leave in the Appendix not to interrupt the flow of the exposition.

The exactness of \eqref{ses-blow-up:eq2} can be proved along the same lines: replace $\Vca \subset Q^{ss}$ with $\Vca \cap \mathrm{\Sigma}_Q \subset \mathrm{\Sigma}_Q$ and use the identifications \eqref{SigmaQ local description} and \eqref{OmegaQ local description}.
\end{proof}

\begin{lemma}\label{lemma:SigmaR-ses conormal sheaf}
There is a short exact sequence of sheaves on $\mathrm{\Sigma}_R^{ss}$ 
\[
\begin{tikzcd}
0 \ar[r] & \Nca^\vee_{\mathrm{\Sigma}_R^{ss}/R^{ss}}(-(\mathrm{\Omega}_R \cap \mathrm{\Sigma}_R^{ss})) \ar[r] & \left[ (\pi_R^* \Kahler^1_{Q^{ss}})|_{\mathrm{\Sigma}_R^{ss}} \right]_\tf \ar[r] & \left[ (p^* \Kahler^1_{\mathrm{\Sigma}_Q})|_{\mathrm{\Sigma}_R^{ss}} \right]_\tf \ar[r] & 0.
\end{tikzcd}
\]
\end{lemma}
\begin{proof}
Let us start from the conormal sequence
\[
\begin{tikzcd}
0 \ar[r] & \Nca^\vee_{\mathrm{\Sigma}_R^{ss}/R^{ss}} \ar[r] & \Kahler^1_R|_{\mathrm{\Sigma}_R^{ss}} \ar[r] & \Kahler^1_{\mathrm{\Sigma}_R^{ss}} \ar[r] & 0.
\end{tikzcd}
\]
It is left exact because $\mathrm{\Sigma}_R^{ss}$ is smooth.
Since $\Nca^\vee_{\mathrm{\Sigma}_R^{ss}/R^{ss}}$ and $\Kahler^1_{\mathrm{\Sigma}_R^{ss}}$ are locally free, so is $\Kahler^1_R|_{\mathrm{\Sigma}_R^{ss}}$. When we apply the functor $\imath^*$, this sequence remains exact and sits into the commutative diagram
\begin{equation}\label{SigmaR-6 diagram}
\begin{tikzcd}
0 \ar[r] & \imath^*\Nca^\vee_{\mathrm{\Sigma}_R^{ss}/R^{ss}} \ar[r] \ar[d, dashed] & \imath^*(\Kahler^1_R|_{\mathrm{\Sigma}_R^{ss}}) \ar[r] \ar[d] & \imath^* \Kahler^1_{\mathrm{\Sigma}_R^{ss}} \ar[r] \ar[d] & 0 \\
0 \ar[r] & \Jca/\Jca^{2}  \ar[r] & \Kahler^1_{\mathrm{\Omega}_R/\mathrm{\Omega}_Q}|_{\mathrm{\Omega}_R \cap \mathrm{\Sigma}_R^{ss}} \ar[r] & \Kahler^1_{\mathrm{\Omega}_R \cap \mathrm{\Sigma}_R^{ss}/\mathrm{\Omega}_Q} \ar[r] & 0,
\end{tikzcd}
\end{equation}
where $\Jca$ is the ideal sheaf of $\mathrm{\Omega}_R \cap \mathrm{\Sigma}_R^{ss}$ in $\mathrm{\Omega}_R^{ss}$. The second row is exact as $\mathrm{\Omega}_R \cap \mathrm{\Sigma}_R^{ss} \to \mathrm{\Omega}_Q$ is a smooth morphism. The dotted arrow is an isomorphism. To prove this, it suffices to check that the ideal sheaf of $\mathrm{\Sigma}_R^{ss}$ in $R^{ss}$ pullbacks on $\mathrm{\Omega}_R^{ss}$ to $\Jca$. One can show that $\mathrm{Tor}_1(\Oca_{\mathrm{\Omega}_R^{ss}}, \Oca_{\mathrm{\Sigma}_R^{ss}})$ vanishes: indeed, the Cartier divisor $\mathrm{\Omega}_R$ remains a divisor when restricted to the integral scheme $\mathrm{\Sigma}_R^{ss}$.

Now, the natural transformation $\mathrm{id} \to \imath_* \circ \imath^*$ and the pushforward by $\imath$ of the diagram \eqref{SigmaR-6 diagram} yield the following morphism of short exact sequences
\begin{equation*}
\begin{tikzcd}
0 \ar[r] & \Nca^\vee_{\mathrm{\Sigma}_R^{ss}/R^{ss}} \ar[r] \ar[d, "\lambda'"] & \Kahler^1_R|_{\mathrm{\Sigma}_R^{ss}} \ar[r] \ar[d, "\lambda"] &  \Kahler^1_{\mathrm{\Sigma}_R^{ss}} \ar[r] \ar[d, "\lambda''"] & 0 \\
0 \ar[r] & \imath_*\imath^*\Nca^\vee_{\mathrm{\Sigma}_R^{ss}/R^{ss}} \ar[r] & \imath_*(\Kahler^1_{\mathrm{\Omega}_R/\mathrm{\Omega}_Q}|_{\mathrm{\Omega}_R \cap \mathrm{\Sigma}_R^{ss}}) \ar[r] & \imath_*\Kahler^1_{\mathrm{\Omega}_R \cap \mathrm{\Sigma}_R^{ss}/\mathrm{\Omega}_Q} \ar[r] & 0.
\end{tikzcd}
\end{equation*}
Clearly, $\lambda'$ is surjective and has kernel $\Nca^\vee_{\mathrm{\Sigma}_R^{ss}/R^{ss}}(-(\mathrm{\Omega}_R \cap \mathrm{\Sigma}_R^{ss}))$. As for $\lambda$, by \cref{lem:ses of differentials for blow up} we have the exact sequence 
\[
\begin{tikzcd}
(\pi_R^* \Kahler^1_{Q^{ss}}) |_{\mathrm{\Sigma}_R^{ss}} \ar[r] & \Kahler^1_R|_{\mathrm{\Sigma}_R^{ss}} \ar[r,"\lambda"] & \imath_*(\Kahler^1_{\mathrm{\Omega}_R/\mathrm{\Omega}_Q}|_{\mathrm{\Omega}_R \cap \mathrm{\Sigma}_R^{ss}}) \ar[r] & 0.
\end{tikzcd}
\]
Since $(\pi_R^* \Kahler^1_{Q^{ss}}) |_{\mathrm{\Sigma}_R^{ss}}$ coincides with the torsion-free sheaf $\Kahler^1_R|_{\mathrm{\Sigma}_R^{ss}}$ away from $\mathrm{\Omega}_R \cap \mathrm{\Sigma}_R^{ss}$, the kernel of $\lambda$ equals $[(\pi_R^* \Kahler^1_{Q^{ss}} )|_{\mathrm{\Sigma}_R^{ss}}]_\tf$. The same argument applied to $\lambda''$ yields $\ker(\lambda'') \simeq [(p ^* \Kahler^1_{\mathrm{\Sigma}_Q})|_{\mathrm{\Sigma}_R^{ss}}]_\tf$.

The snake lemma allows us to conclude.
\end{proof}
Since $h \colon U \to \mathrm{\Sigma}_R^{ss}$ is a flat morphism of integral schemes, the functor $h^*$ is exact and commutes with $(-)_\tf$. We deduce from Lemma \ref{lemma:SigmaR-ses conormal sheaf} an exact sequence of sheaves on $U$
\begin{equation}\label{SigmaR-ses conormal sheaf}
\begin{tikzcd}[column sep = 1.2em]
0 \ar[r] & (h^*\Nca^\vee_{\mathrm{\Sigma}_R^{ss}/R^{ss}})(-D) \ar[r] & \left[ h^*(\pi_R^* \Kahler^1_{Q^{ss}})|_{\mathrm{\Sigma}_R^{ss}}\right]_\tf \ar[r] & \left[h^* (p^* \Kahler^1_{\mathrm{\Sigma}_Q})|_{\mathrm{\Sigma}_R^{ss}}\right]_\tf \ar[r] & 0.
\end{tikzcd}
\end{equation}
On $U$ we can also consider the exact sequence of sheaves of differentials
\[
\begin{tikzcd}
h^* \left[(p^* \Kahler^1_{\mathrm{\Sigma}_Q})|_{\mathrm{\Sigma}_R^{ss}}\right] \ar[r] & \Kahler^1_U \ar[r] & \Kahler^1_{U/\mathrm{\Sigma}_Q} \ar[r] & 0.
\end{tikzcd}
\]
Off $D$, the first arrow is injective, because $(p \circ h)|_{U \setminus D}$ is a smooth morphism. Since $\Kahler^1_U$ is torsion-free, we get the short exact sequence
\begin{equation}\label{U-ses}
\begin{tikzcd}
0 \ar[r] & \left[h^* (p^* \Kahler^1_{\mathrm{\Sigma}_Q})|_{\mathrm{\Sigma}_R^{ss}}\right]_\tf \ar[r] & \Kahler^1_U \ar[r] & \Kahler^1_{U/\mathrm{\Sigma}_Q} \ar[r] & 0.
\end{tikzcd}
\end{equation}
Glueing \eqref{SigmaR-ses conormal sheaf} with \eqref{U-ses} we obtain the exact complex
\[
\begin{tikzcd}
0 \ar[r] & (h^*\Nca^\vee_{\mathrm{\Sigma}_R^{ss}/R^{ss}})(-D) \ar[r] & \left[ h^*(\pi_R^* \Kahler^1_{Q^{ss}})|_{\mathrm{\Sigma}_R^{ss}}\right]_\tf \ar[r]& \Kahler^1_U \ar[r] & \Kahler^1_{U/\mathrm{\Sigma}_Q} \ar[r] & 0.
\end{tikzcd}
\]
We shall compare it with \eqref{SigmaR-Ext complex}.
\begin{lemma}\label{lemma:SigmaR-commutative square}
There is a commutative diagram
\begin{equation}\label{commutative diagram-SigmaR}
\begin{tikzcd}
h^*\left[(\pi_R^* \Kahler^1_{Q^{ss}}) |_{\mathrm{\Sigma}_R^{ss}}\right] \ar[r] \isoarrow{d} & \Kahler^1_U \ar[d] \\
\sExt_U^2(\Fca_U,\Kca_U) \ar[r] & \sExt_U^2(\Fca_U,\Oca(-k)^{\oplus N}).
\end{tikzcd}
\end{equation}
\end{lemma}
\begin{proof}
The top horizontal arrow of the diagram \eqref{commutative diagram-SigmaR} is the differential of the morphism \eqref{morphism from U to Quot} from $U$ to $Q^{ss}$. Since this morphism is $\mathrm{GL}(N)$-equivariant, it gives a commutative diagram 
\[
\begin{tikzcd}
h^*\left[(\pi_R^* \Kahler^1_{Q^{ss}}) |_{\mathrm{\Sigma}_R^{ss}}\right] \ar[r] \ar[d] & \Kahler^1_U \ar[d] \\
\mathfrak{gl}(N)^\vee \otimes \Oca_U \ar[r, equal] & \mathfrak{gl}(N)^\vee \otimes \Oca_U,
\end{tikzcd}
\]
whose vertical arrows describe the infinitesimal action of $\mathrm{GL}(N)$ (\cf \cite[Proof of Theorem 4.2]{LcsQs}). Let us provide more explicit expressions for these arrows.

The homomorphism $h^*\left[(\pi_R^* \Kahler^1_{Q^{ss}}) |_{\mathrm{\Sigma}_R^{ss}}\right] \rightarrow \mathfrak{gl}(N)^\vee \otimes \Oca_U$ is the composite map
\begin{multline*}
\quad \quad h^*\left[(\pi_R^* \Kahler^1_{Q^{ss}}) |_{\mathrm{\Sigma}_R^{ss}}\right] \simeq \sExt_U^2(\Fca_U,\Kca_U) \rightarrow \sExt_U^2(\Fca_U,\Oca(-k)^{\oplus N}) \rightarrow \\
\rightarrow \sExt_U^2(\Oca(-k)^{\oplus N},\Oca(-k)^{\oplus N}) \simeq u^* \sHom(\Oca^{\oplus N}, \Oca^{\oplus N})^\vee = \mathfrak{gl}(N)^\vee \otimes \Oca_U,
\end{multline*}
where the arrows between the relative Ext-sheaves are given by composition, and the second isomorphism by Grothendieck-Verdier duality.

The homomorphism $\Kahler^1_U \rightarrow \mathfrak{gl}(N)^\vee \otimes \Oca_U$ is the composition
\[
\quad \quad  \Kahler^1_U \rightarrow \Kahler^1_u
\rightarrow u^*\sHom (\Oca^{\oplus N}, p_{\Xca^{[n]},*} b_X^*[(\Ica_1 \oplus \Ica_2)(k)])^\vee \otimes \Oca_u(-1)  \simeq \mathfrak{gl}(N)^\vee \otimes \Oca_U,
\]
where the second arrow comes from the relative Euler sequence for $u$.

Both these expressions can be checked at the closed points of the reduced scheme $U$, because $\mathfrak{gl}(N)^\vee \otimes \Oca_U$ is locally free\footnote{Actually, the first expression holds true, possibly up to a sign coming from the deformation-theoretic description of the tangent space of the Quot scheme at a closed point (see \cite[Proposition 2.2.7]{HLgmss}). If necessary, multiply the isomorphism \eqref{cotangent sheaf Quot scheme} with $-1$ to fix the problem.}. These two homomorphisms fit into the diagram
\[
\begin{tikzcd}[column sep=0.7 em, scale cd=0.9]
h^*\left[(\pi_R^* \Kahler^1_{Q^{ss}}) |_{\mathrm{\Sigma}_R^{ss}}\right]  \ar[rr] \isoarrow{d} & & \Kahler^1_U \ar[d] \\
\sExt_U^2(\Fca_U,\Kca_U) \ar[r] &\sExt_U^2(\Fca_U,\Oca(-k)^{\oplus N}) \ar[r, "\sim", "GV"'] \ar[d] & u^*\sHom (\Oca^{\oplus N}, p_{\Xca^{[n]},*} b_X^*[(\Ica_1 \oplus \Ica_2)(k)])^\vee \otimes \Oca_u(-1) \isoarrow{d} \\
& \sExt_U^2(\Oca(-k)^{\oplus N},\Oca(-k)^{\oplus N}) \ar[r, "\sim", "GV"'] & u^* \sHom(\Oca^{\oplus N}, \Oca^{\oplus N})^\vee = \mathfrak{gl}(N)^\vee \otimes \Oca_U.
\end{tikzcd}
\]
By functoriality of Grothendieck-Verdier duality, the bottom-right square commutes. We deduce the commutativity of the upper rectangle, which gives \eqref{commutative diagram-SigmaR}.
\end{proof}
Since $\Kahler^1_U$ is torsion-free, from Lemma \ref{lemma:SigmaR-commutative square} we deduce a commutative diagram
\[
\begin{tikzcd}
\left[h^*(\pi_R^* \Kahler^1_{Q^{ss}}) |_{\mathrm{\Sigma}_R^{ss}}\right]_\tf \ar[r] \isoarrow{d} & \Kahler^1_U \ar[d] \\
\sExt_U^2(\Fca_U,\Kca_U)_\tf \ar[r] & \sExt_U^2(\Fca_U,\Oca(-k)^{\oplus N}),
\end{tikzcd}
\]
and thus a morphism of complexes
\begin{equation}\label{SigmaR-4 terms complex morphism}
\begin{tikzcd}[column sep=0.7em ]
0 \ar[r] & (h^* \Nca^\vee_{\mathrm{\Sigma}_R^{ss}/R^{ss}})(-D) \ar[r] \ar[d] & \left[h^*(\pi_R^* \Kahler^1_{Q^{ss}}) |_{\mathrm{\Sigma}_R^{ss}}\right]_\tf \ar[r] \isoarrow{d} &\Kahler^1_U \ar[r] \ar[d] &\Kahler^1_{U/\mathrm{\Sigma}_Q} \ar[r] \ar[d] &0 \\
0 \ar[r] &\sExt_U^1(\Fca_U,\Fca_U)_\tf \ar[r] 
&\sExt_U^2(\Fca_U,\Kca_U)_\tf \ar[r] &\sExt_U^2(\Fca_U,\Oca(-k)^{\oplus N}) \ar[r] &\sExt_U^2(\Fca_U,\Fca_U) \ar[r] &0.
\end{tikzcd}
\end{equation}
\begin{proposition}\label{prop:conormal sheaf SigmaR in R}
The pullback along $h$ of the conormal sheaf to $\mathrm{\Sigma}_R^{ss}$ in $R^{ss}$ is given by
\begin{equation}\label{description conormal sheaf SigmaR in R}
h^* \Nca^\vee_{\mathrm{\Sigma}_R^{ss}/R^{ss}} \simeq [u^*(V' \oplus V)](D).
\end{equation}
Geometrically, this corresponds to an isomorphism of $U$-schemes
\begin{equation}\label{geometric normal bundle SigmaR in R}
h^* N_{\mathrm{\Sigma}_R^{ss}/R^{ss}} \simeq u^*\mathbf{Spec}(\mathrm{S}^\bullet[(V' \oplus V)(E)]).
\end{equation}
\end{proposition}
\begin{proof}
Let us split \eqref{SigmaR-4 terms complex morphism} into two morphisms of short exact sequences. The first is given by
\[
\begin{tikzcd}
0 \ar[r] & (h^* \Nca^\vee_{\mathrm{\Sigma}_R^{ss}/R^{ss}})(-D) \ar[d,"\mu'"] \ar[r] & \left[h^*(\pi_R^* \Kahler^1_{Q^{ss}}) |_{\mathrm{\Sigma}_R^{ss}}\right]_\tf \isoarrow{d} \ar[r] & \left[h^* (p^* \Kahler^1_{\mathrm{\Sigma}_Q})|_{\mathrm{\Sigma}_R^{ss}}\right]_\tf \ar[d,"\mu''"] \ar[r] & 0 \\
0 \ar[r] & \sExt_U^1(\Fca_U,\Fca_U)_\tf \ar[r] & \sExt_U^2(\Fca_U,\Kca_U)_\tf \ar[r] & \Cca \ar[r] & 0.
\end{tikzcd}
\]
By the snake lemma, $\mu'$ is injective, $\mu''$ is surjective and $\coker (\mu') \simeq \ker(\mu'')$. The second corresponds to
\begin{equation*}
\begin{tikzcd}
0 \ar[r] & \left[h^* (p^* \Kahler^1_{\mathrm{\Sigma}_Q})|_{\mathrm{\Sigma}_R^{ss}}\right]_\tf \ar[d, "\mu''"] \ar[r] & \Kahler^1_U \ar[d] \ar[r] & \Kahler^1_{U/\mathrm{\Sigma}_Q} \ar[d, "\nu"] \ar[r] & 0 \\
0 \ar[r] & \Cca \ar[r] & \sExt_U^2(\Fca_U,\Oca(-k)^{\oplus N}) \ar[r] & \sExt_U^2(\Fca_U,\Fca_U) \ar[r] & 0.
\end{tikzcd}
\end{equation*}
Its central vertical arrow is obtained glueing the short exact sequence
\[
\begin{tikzcd}
0 \ar[r] & u^*\Kahler^1_{\Xca^{[n]}} \ar[r] & \Kahler^1_U \ar[r] & \Kahler^1_u \ar[r] &0
\end{tikzcd}
\]
with the relative Euler sequence
\[
\begin{tikzcd}
0 \ar[r] &  \Kahler^1_u \ar[r] & u^*\sHom (\Oca^{\oplus N}, p_{\Xca^{[n]},*} b_X^*[(\Ica_1 \oplus \Ica_2)(k)])^\vee \otimes \Oca_u(-1) \ar[r] & \Oca \ar[r] & 0;
\end{tikzcd}
\]
in particular, its kernel and cokernel are $u^*\Kahler^1_{\Xca^{[n]}}$ and $\Oca$, respectively. By the snake lemma, from the surjectivity of $\mu''$ we deduce the isomorphism $\coker(\nu) \simeq \Oca$ and the short exact sequence
\begin{equation}\label{blowup conormal sequence for Xca^{[n]}}
\begin{tikzcd}
0 \ar[r] & \ker(\mu'') \ar[r] & u^*\Kahler^1_{\Xca^{[n]}} \ar[r] & \ker(\nu) \ar[r] & 0.
\end{tikzcd}
\end{equation}
Let us compute $\ker(\nu)$. Consider the solid diagram
\begin{equation}\label{diagram-diagonal and antidiagonal}
\begin{tikzcd}[column sep=1em ]
0 \ar[r] & h^* \imath_* \Kahler^1_{\mathrm{\Omega}_R \cap \mathrm{\Sigma}_R^{ss}/\mathrm{\Omega}_Q} \ar[d, "\nu'", dashed] \ar[r] & \Kahler^1_{U/\mathrm{\Sigma}_Q} \ar[d, "\nu"] \ar[r] & \Kahler^1_{U/\mathrm{\Sigma}_R^{ss}} \ar[d, "\nu''", dashed] \ar[r] & 0 \\
0 \ar[r] & (b \circ u)^*  \delta_*  \sExt_{X^{[n]}}^2(\Ica,\Ica)^{\oplus 2} \ar[r] & \sExt_U^2(\Fca_U,\Fca_U) \ar[r] & (b \circ u)^* \bigoplus_{i=1}^2 \pi_i^* \sExt_{X^{[n]}}^2(\Ica,\Ica) \ar[r] & 0.
\end{tikzcd}
\end{equation}
The two rows are exact: the first by the smoothness of $h$ together with \eqref{ses-blow-up:eq2}, the second by \cref{lem:SigmaR ExtEE}. The diagram \eqref{diagram-diagonal and antidiagonal} can be completed to a morphism of short exact sequences: indeed, from the torsion sheaf $h^* \imath_* \Kahler^1_{\mathrm{\Omega}_R \cap \mathrm{\Sigma}_R^{ss}/\mathrm{\Omega}_Q}$ to the locally free sheaf $(b \circ u)^* \bigoplus_{i=1}^2 \pi_i^* \sExt_{X^{[n]}}^2(\Ica,\Ica)$ there are no non-zero morphisms.

The source and the target of $\nu''$ are locally free sheaves of rank one and two, respectively; moreover, its cokernel is a quotient of $\coker (\nu) \simeq \Oca$. We easily deduce that $\nu''$ is injective and that its cokernel is isomorphic to $\Oca$. By the snake lemma, $\nu'$ is surjective, and $\ker (\nu') \simeq \ker (\nu)$.
Now, recall that $\zeta$ denotes the imbedding of $D$ in $U$. The equality $\nu' = \zeta_* \zeta^* \nu'$ suggests to consider the pullback of \eqref{diagram-diagonal and antidiagonal} by $\zeta$; it is still a morphism of short exact sequences. Since $\nu''$ has locally free cokernel, its pullback by $\zeta$ remains injective, so that $\ker (\zeta^*\nu') \simeq \ker (\zeta^*\nu)$. Hence,
\[
\ker(\nu) \simeq \ker(\nu') \simeq \ker(\zeta_*\zeta^*\nu') \simeq \zeta_* \ker(\zeta^*\nu')\simeq \zeta_* \ker(\zeta^*\nu).
\]
Let us determine $\ker(\zeta^*\nu)$. Recall from \cref{rem:GL2action on D} that we have on $D$ a $\mathrm{GL}(2)$-action. Consider the induced action of the group scheme $(\mathrm{GL}(2) \times \mathrm{\Omega}_Q)/\mathrm{\Omega}_Q$ on $D/\mathrm{\Omega}_Q$; infinitesimally, it is described by $\zeta^*\nu$. In particular, $\zeta^*\nu$ factors as the composition
\[
\zeta^*\Kahler^1_{U/\mathrm{\Sigma}_Q} \xrightarrow{\sim} \Kahler^1_{D/\mathrm{\Omega}_Q} \twoheadrightarrow \Kahler^1_{D/\mathrm{\Omega}_Q \times_{X^{[n]}}E} \hookrightarrow \mathfrak{gl}(2)^\vee \otimes \Oca_D \simeq \sExt_D^2(\Fca_D,\Fca_D);
\]
the third map is injective, because it identifies with $\Kahler^1_{D/\mathrm{\Omega}_Q \times_{X^{[n]}}E} \simeq \mathfrak{pgl}(2)^\vee \otimes \Oca_D \hookrightarrow  \mathfrak{gl}(2)^\vee \otimes \Oca_D$. Therefore, $\ker(\zeta^* \nu)$ is the pullback of $\Kahler^1_{E/X^{[n]}}$ along $u|_D \colon D \to E$. We conclude that 
\[
\ker(\nu) \simeq \zeta_* \ker(\zeta^* \nu) \simeq u^*\varepsilon_*\Kahler^1_{E/X^{[n]}}.
\]
This last isomorphism used in \eqref{blowup conormal sequence for Xca^{[n]}} shows that $\ker(\mu'')$ (and hence $\coker(\mu')$) is isomorphic to $u^* b^* \Kahler^1_{X^{[n]}\times X^{[n]}}$. Thus, we have obtained a short exact sequence
\[
\begin{tikzcd}
0 \ar[r] & (h^* \Nca^\vee_{\mathrm{\Sigma}_R^{ss}/R^{ss}})(-D) \ar[r, "\mu'"] & \sExt_U^1(\Fca_U,\Fca_U)_\tf \ar[r] & u^* b^* \Kahler^1_{X^{[n]}\times X^{[n]}} \ar[r] & 0;
\end{tikzcd}
\]
by \cref{lem:SigmaR-ExtFF}, we get $(h^* \Nca^\vee_{\mathrm{\Sigma}_R^{ss}/R^{ss}})(-D) \simeq u^*(V' \oplus V)$, as desired.
\end{proof}

\subsection{Computation of the normal cone} Consider on $X^{[n]} \times X^{[n]}$ the quadratic map 
\[
\mathrm{\Upsilon} \colon \sExt_{X^{[n]} \times X^{[n]}}^1(\Ica_1,\Ica_2) \oplus \sExt_{X^{[n]} \times X^{[n]}}^1(\Ica_2,\Ica_1) \to \Oca_{X^{[n]} \times X^{[n]}}
\] 
given as follows: choose a 2-form $\omega_X \in \Ho^0(X,\Kahler^2_X)$, and compose the Yoneda square map
\begin{equation*}\label{relative Yoneda}
\sExt_{X^{[n]} \times X^{[n]}}^1(\Ica_1,\Ica_2) \oplus \sExt_{X^{[n]} \times X^{[n]}}^1(\Ica_2,\Ica_1) \to \sExt_{X^{[n]} \times X^{[n]}}^2(\Ica_2,\Ica_2), \quad (e_{12}, e_{21}) \mapsto e_{12} \cup e_{21}
\end{equation*}
with the isomorphism
\[
\sExt_{X^{[n]} \times X^{[n]}}^2(\Ica_2,\Ica_2) \xrightarrow{\mathrm{Tr}} \Oca \otimes \Ho^2(X,\Oca) \xrightarrow{\cup \omega_X} \Oca \otimes \Ho^2(X,\Kahler^2_X) \xrightarrow{\int_X} \Oca.
\]
In this section, we describe the normal cone $C_{\mathrm{\Sigma}_R^{ss}/R^{ss}}$ to $\mathrm{\Sigma}_R^{ss}$ in $R^{ss}$ in terms of $\mathrm{\Upsilon}$. Note that $\mathrm{\Upsilon}$ is completely determined by the $\Oca_{X^{[n]} \times X^{[n]}}$-linear map
\begin{equation}\label{Morphism Phi}
\mathrm{\Psi} \colon \sExt_{X^{[n]} \times X^{[n]}}^1(\Ica_1, \Ica_2) \otimes \sExt_{X^{[n]} \times X^{[n]}}^1(\Ica_2, \Ica_1) \rightarrow \Oca, \quad e_{12}\otimes e_{21} \mapsto \int_X  \mathrm{Tr}(e_{12} \cup e_{21})  \cup \omega_X.
\end{equation}

\medskip

We begin by studying the pullback of $\mathrm{\Psi}$ along $b \colon \Xca^{[n]} \to X^{[n]} \times X^{[n]}$. To this aim, we shall need from \cref{sec:construction by Markman} the vector bundles $V$ and $V'$ on $\Xca^{[n]}$ constructed by Markman and the tautological line subbundle $\Lca$ of $\beta^*\Tca$ on $E$. Recall from \cref{thm:properties of V} that
\begin{equation}\label{V(-E) and V'(-E)}
V(-E) = [b^*\sExt_{X^{[n]} \times X^{[n]}}^1(\Ica_1, \Ica_2)]_\tf \quad \text{and} \quad V'(-E) = [b^*\sExt_{X^{[n]} \times X^{[n]}}^1(\Ica_2, \Ica_1)]_\tf,
\end{equation}
as well as $V|_E \simeq \Lca^{\perp_\omega}/\Lca \simeq V'|_E$. We denote $\omega_\Lca$ the symplectic form induced by $\beta^*\omega$ on $\Lca^{\perp_\omega}/\Lca$.
\begin{proposition}\label{proposition:Yoneda restricted} The $\Oca_{\Xca^{[n]}}$-linear homomorphism 
\[
b^* \mathrm{\Psi} \colon b^* \sExt_{X^{[n]} \times X^{[n]}}^1(\Ica_1, \Ica_2) \otimes b^* \sExt_{X^{[n]} \times X^{[n]}}^1(\Ica_2, \Ica_1) \rightarrow \Oca
\]
factors through 
\[
\overline{b^* \mathrm{\Psi}} \colon (V \otimes V')(-2E) \rightarrow \Oca(-2E).
\]
After tensorisation with $\Oca(2E)$ and via the identification $V' \simeq V^\vee$ of \cref{thm:properties of V}(1), $\overline{b^* \mathrm{\Psi}}$ corresponds to the canonical evaluation morphism 
\begin{equation*}
ev \colon V \otimes V^\vee \to \Oca, \quad v \otimes \xi \mapsto \xi(v).
\end{equation*}
Moreover, the restriction of $\overline{b^* \mathrm{\Psi}}$ to $E$ coincides, up to a non-zero multiplicative factor, with 
\[
\omega_\Lca \otimes \mathrm{id} \colon (\Lca^{\perp_\omega}/\Lca)^{\otimes 2} \otimes (\Lca^\vee)^{\otimes 2} \to (\Lca^\vee)^{\otimes 2}.
\] 
\end{proposition}
To prove \cref{proposition:Yoneda restricted}, we shall need the following
\begin{lemma}\label{lem:cohomologies of Lperp/L}
We have the isomorphisms
\[
\mathrm{Hom}((\Lca^{\perp_\omega}/\Lca)^{\otimes 2}, \Lca^{\otimes l}) \simeq \begin{cases} \mathbb{C} \omega_\Lca &\text{if} \ l = 0 \\
0 &\text{if} \ l =1,2.
\end{cases}
\]
\end{lemma}
\begin{proof}
By \cite[Lemma~4.3]{MBBccc}, $\Ho^0(E,(\Lca^{\perp_\omega} \otimes \Lca^{\perp_\omega})^\vee) = \mathrm{Hom}((\Lca^{\perp_\omega})^{\otimes 2}, \Oca) \simeq \mathbb{C}$.
The functor $\mathrm{Hom}(-,\Oca)$ applied to the epimorphism $(\Lca^{\perp_\omega})^
{\otimes 2} \rightarrow (\Lca^{\perp_\omega}/\Lca)^{\otimes 2}$ yields the injective map
\[
\mathrm{Hom}((\Lca^{\perp_\omega}/\Lca)^{\otimes 2}, \Oca) \rightarrow \mathrm{Hom}((\Lca^{\perp_\omega})^
{\otimes 2}, \Oca).
\]
Since $\Lca^{\perp_\omega}/\Lca$ carries the symplectic form $\omega_\Lca$, we deduce $\mathrm{Hom}((\Lca^{\perp_\omega}/\Lca)^{\otimes 2}, \Oca) = \mathbb{C}\omega_\Lca$.

The other statements result in the same way, once we know that $\Ho^0(E,(\Lca^{\perp_\omega} \otimes \Lca^{\perp_\omega})^\vee \otimes \Lca^{\otimes l})=0$ for $l = 1,2$. The argument proving \cite[Lemma~4.3]{MBBccc} carries over: from the vanishing of $\Rder^m \beta_*( \Lca^{\otimes l})$ for any $0 \le m \le 2n-2$ and $l \ge 1$, we can even show that $\beta_*[(\Lca^{\perp_\omega} \otimes \Lca^{\perp_\omega})^\vee \otimes \Lca^{\otimes l}] \simeq 0$ for all $l \ge 1$.
\end{proof}
\begin{proof}[Proof (of \cref{proposition:Yoneda restricted}).]
Since $\Oca_{\Xca^{[n]}}$ is a torsion-free sheaf, $b^* \mathrm{\Psi}$ factors through the morphism $\overline{b^* \mathrm{\Psi} } \colon (V \otimes V')(-2E) \rightarrow \Oca_{\Xca^{[n]}}$. To prove that the image of $\overline{b^* \mathrm{\Psi} }$ lies in $\Oca(-2E)$, consider the commutative diagram with exact rows
\begin{equation}\label{factorisation V Vv}
\begin{tikzcd}
& & (V \otimes V')(-2E) \ar[d, "\overline{b^* \mathrm{\Psi}}"] \ar[r] & \varepsilon_*\varepsilon^* [(V \otimes V')(-2E)] \ar[d, "\varepsilon_*\varepsilon^*(\overline{b^* \mathrm{\Psi}})"] \ar[r] & 0 \\
0 \ar[r] & \Oca(-E) \ar[r] & \Oca_{\Xca^{[n]}}\ar[r] & \varepsilon_* \Oca_E \ar[r] & 0.
\end{tikzcd}
\end{equation}
By \cref{lem:cohomologies of Lperp/L}, $\varepsilon^*( \overline{b^* \mathrm{\Psi}}) \colon [(\Lca^{\perp_\omega}/\Lca) \otimes \Lca^\vee]^{\otimes 2} \to \Oca_E$ vanishes, so the image of $\overline{b^* \mathrm{\Psi}}$ lies in $\Oca(-E)$. Thus, we can repeat the argument, with the second row of \eqref{factorisation V Vv} twisted by $\Oca(-E)$; we conclude that $\im(\overline{b^* \mathrm{\Psi}}) \subset \Oca(-2E)$. 

This gives the desired factorisation $\overline{b^* \mathrm{\Psi}} \colon V \otimes V' \otimes \Oca(-2E) \rightarrow \Oca(-2E)$. Now, tensor it with $\Oca(2E)$ and identify $V' \simeq V^\vee$ as in \cref{thm:properties of V}(1) (i.e.\ using Grothendieck-Verdier duality). What we obtain is precisely $ev$: indeed, these two maps coincide on $\Xca^{[n]} \setminus E$, and in fact on the whole of $\Xca^{[n]}$ because $\Oca$ is torsion-free.

We deduce in particular that $\overline{b^* \mathrm{\Psi}}$ is surjective; therefore, its restriction to $E$ is a non-zero element of $\mathrm{Hom}((\Lca^{\perp_\omega}/\Lca )^{\otimes 2} \otimes (\Lca^\vee)^{\otimes 2}, (\Lca^\vee)^{\otimes 2}) \simeq \mathrm{Hom}((\Lca^{\perp_\omega}/\Lca )^{\otimes 2}, \Oca) = \mathbb{C}\omega_\Lca$.
\end{proof}

\cref{proposition:Yoneda restricted} shows that the pullback of $\mathrm{\Upsilon}$ along $b$ factors through a quadratic map of locally free $\Oca_{\Xca^{[n]}}$-modules
\[
\overline{b^* \mathrm{\Upsilon}} \colon (V \oplus V')(-E) \to \Oca(-2E).
\]
This map coincides with $b^*\mathrm{\Upsilon}$ on $\Xca^{[n]} \setminus E$, and its restriction to $E$ gives a quadratic map 
\[
\overline{b^* \mathrm{\Upsilon}}|_E \colon (\Lca^{\perp_\omega} / \Lca)^{\oplus 2} \otimes \Lca^\vee \to (\Lca^\vee)^{\otimes 2}.
\]
We abuse the notation, and write with the same letters the corresponding morphisms between the total spaces of the involved locally free sheaves.
\begin{proposition}\label{proposition:SigmaR-normal cone}
The isomorphism \eqref{geometric normal bundle SigmaR in R} restricts to an isomorphism of $U$-schemes
\begin{equation}\label{description normal cone SigmaR in R}
h^* C_{\mathrm{\Sigma}_R^{ss}/R^{ss}} \simeq u^*[(\overline{b^*\mathrm{\Upsilon}})^{-1} (0)].
\end{equation}
\end{proposition}
\begin{proof}
We need to prove that, under the isomorphism $h^*\mathrm{S}^\bullet (\Ica_{\mathrm{\Sigma}_R^{ss}} / \Ica_{\mathrm{\Sigma}_R^{ss}}^2) \simeq u^*\mathrm{S}^\bullet ([V' \oplus V](E))$ of $\Oca_U$-algebras induced by \eqref{description conormal sheaf SigmaR in R}, the ideal sheaves defining $h^*C_{\mathrm{\Sigma}_R^{ss}/R^{ss}}$ and $u^*[(\overline{b^*\mathrm{\Upsilon}})^{-1} (0)]$ coincide. These ideal sheaves are obtained by pulling back along the flat morphisms $h$ and $u$ the ideal sheaf of $\mathrm{S}^\bullet (\Ica_{\mathrm{\Sigma}_R^{ss}} / \Ica_{\mathrm{\Sigma}_R^{ss}}^2)$ defining $C_{\mathrm{\Sigma}_R^{ss}/R^{ss}}$ and the ideal sheaf of $\mathrm{S}^\bullet ([V' \oplus V](E))$ defining $(\overline{b^*\mathrm{\Upsilon}})^{-1} (0)$, respectively.

\smallskip

Consider first the closed immersion $C_{\mathrm{\Sigma}_R^{ss}/R^{ss}} \hookrightarrow N_{\mathrm{\Sigma}_R^{ss}/R^{ss}}$ and its graded ideal sheaf
\[
\Jca \coloneqq \ker \left[ \oplus_{m \ge 0} \mathrm{S}^m(\Ica_{\mathrm{\Sigma}_R^{ss}} / \Ica_{\mathrm{\Sigma}_R^{ss}}^2) \twoheadrightarrow \oplus_{m \ge 0} (\Ica_{\mathrm{\Sigma}_R^{ss}}^m / \Ica_{\mathrm{\Sigma}_R^{ss}}^{m+1}) \right].
\]
By the normal flatness of $R^{ss}$ along $\mathrm{\Sigma}_R^{ss}$, $\Jca$ is a flat $\Oca_{\mathrm{\Sigma}_R^{ss}}$-module; in particular, its degree-$m$ part $\Jca_m$ is locally free for any $m$. Moreover, $\Jca$ is generated by $\Jca_2$ (so that $h^* \Jca$ is generated by $h^*(\Jca_2)$): this results from the equality of the Hessian and normal cones to $\mathrm{\Sigma}_R^{ss}$ in $R^{ss}$ (see \ref{normal cones}), which the aforementioned normal flatness reduces to the fibrewise verification done on pages 76--77 of \cite{O'GdmsK3}.

\smallskip

Consider now the ideal sheaf of $\mathrm{S}^\bullet ([V' \oplus V](E))$ defining $(\overline{b^*\mathrm{\Upsilon}})^{-1}(0)$; it is generated by the image of the $\Oca_{\Xca^{[n]}}$-linear composition
\begin{equation}\label{Yoneda dual}
\nu \colon \Oca(2E) \to (V \otimes V')^\vee(2E) \simeq (V' \otimes V)(2E) \hookrightarrow \mathrm{S}^2 ([V' \oplus V](E)),
\end{equation}
which is fibrewise injective by \cref{proposition:Yoneda restricted}. We claim that there is a commutative diagram
\[
\begin{tikzcd}
& \Oca_U(2D) \ar[r, "u^*\nu"] \ar[d, dashed, "\mu"] & \mathrm{S}^2 (u^*[V' \oplus V](D)) \isoarrow{d} & &  \\
0 \ar [r] & h^* (\Jca_2) \ar[r] & \mathrm{S}^2 [h^* ( \Ica_{\mathrm{\Sigma}_R^{ss}} / \Ica_{\mathrm{\Sigma}_R^{ss}}^2 )] \ar[r] & h^* (\Ica_{\mathrm{\Sigma}_R^{ss}}^2 / \Ica_{\mathrm{\Sigma}_R^{ss}}^3) \ar[r] & 0.
\end{tikzcd}
\]
To obtain $\mu$, it suffices to show that the morphism $\Oca_U(2D) \to h^* (\Ica_{\mathrm{\Sigma}_R^{ss}}^2 / \Ica_{\mathrm{\Sigma}_R^{ss}}^3)$ given by composition of solid arrows vanishes. As $h^* ( \Ica_{\mathrm{\Sigma}_R^{ss}}^2 / \Ica_{\mathrm{\Sigma}_R^{ss}}^3 )$ is locally free and the Yoneda map is compatible with base change, such a vanishing can be checked on $U \setminus D$, and even at the closed points of this reduced scheme; there, it holds by \eqref{fibres normal cone SigmaQ}, which shows also the rank of $h^* (\Jca_2)$ to be 1. Being a fibrewise injective homomorphism of invertible sheaves, $\mu$ is an isomorphism. 

We conclude that the ideal sheaves defining $h^* C_{\mathrm{\Sigma}_R^{ss}/R^{ss}}$ and $u^*[(\overline{b^*\mathrm{\Upsilon}})^{-1} (0)]$ coincide: indeed, they are generated by their degree-2 parts, which are isomorphic via $\mu$.
\end{proof}

\subsection{Group actions and quotients}
Recall from \cref{section:second blow-up} that $\pi_S \colon S \to R$ denotes the blow-up of $R$ along $\mathrm{\Sigma}_R$, $\mathrm{\Sigma}_S$ its exceptional divisor and $\mathrm{\Omega}_S$ (resp.\ $\mathrm{\Delta}_S$) the strict transform of $\mathrm{\Omega}_R$ (resp.\ $\mathrm{\Delta}_R$). In this section, we globally describe the GIT quotient of $\mathrm{\Sigma}_S$ by $\mathrm{PGL}(N)$; in the next section, we do the same for $\mathrm{\Sigma}_S \cap \mathrm{\Omega}_S$ and $\mathrm{\Sigma}_S \cap \mathrm{\Delta}_S$. The main ingredient is a natural $\mathrm{PO}(2)$-action on $(\overline{b^*\mathrm{\Upsilon}})^{-1} (0)$, which we now introduce.

\smallskip

Recall from \ref{sec:O(2) linearisation} the natural $\Orm(2)$-equivariant structure on $\Ica_1 \oplus \Ica_2$. It induces an $\mathrm{O}(2)$-equivariant structure on $\sExt_{X^{[n]} \times X^{[n]}}^1(\Ica_1, \Ica_2) \otimes \sExt_{X^{[n]} \times X^{[n]}}^1(\Ica_2, \Ica_1)$, which descends to a $\mathrm{PO}(2)$-equivariant structure. The homomorphism $\mathrm{\Psi}$ defined by \eqref{Morphism Phi} becomes $\mathrm{PO}(2)$-equivariant, if we give $\Oca$ the $\mathrm{PO}(2)$-equivariant structure that descends to the following $\bm{\mu}_2 \simeq \mathrm{PO}(2)/\mathrm{PSO}(2)$-equivariant structure: we pick the natural $\bm{\mu}_2$-equivariant structure on $\Oca$ coming from the action of $\iota$ on $X^{[n]} \times X^{[n]}$ and we twist it by the non-trivial character of $\bm{\mu}_2$ (mind that $\mathrm{Tr} \circ \cup$ is anticommutative, \cf \cite[Section 10.1.7]{HLgmss}). 

As a consequence, the homomorphism $b^*\mathrm{\Psi}$ is also $\mathrm{PO}(2)$-equivariant. Now, the action of $\widetilde{\iota}$ on $\Xca^{[n]}$ endows the sheaf $\Oca(-E)$ with a natural $\bm{\mu}_2$-equivariant structure, hence with an induced $\mathrm{PO}(2)$-equivariant structure; in particular, the sheaves $\Oca(-2E)$ and, because of \eqref{V(-E) and V'(-E)}, $V \otimes V'$ are $\mathrm{PO}(2)$-equivariant as well. The homomorphism $\overline{b^*\mathrm{\Psi}}$ in \cref{proposition:Yoneda restricted} is $\mathrm{PO}(2)$-equivariant.

This provides a natural $\mathrm{PO}(2)$-action on the scheme $(\overline{b^*\mathrm{\Upsilon}})^{-1} (0)$, thence commuting actions of $\mathrm{PO}(2)$ and $\mathrm{PGL}(N)$ on $U \times_{\Xca^{[n]}} (\overline{b^*\mathrm{\Upsilon}})^{-1} (0)$.
 \begin{theorem}\label{thm:SigmaS}
\begin{enumerate}
\item[(1)] The exceptional divisor of the blow-up of $R^{ss}$ along $\mathrm{\Sigma}_R^{ss}$ is $\PP(C_{\mathrm{\Sigma}_R^{ss}/R^{ss}}) \simeq \PP(u^*[(\overline{b^* \mathrm{\Upsilon}})^{-1} (0)])\sslash \mathrm{PO}(2)$.
\item[(2)] The GIT quotient of $\mathrm{\Sigma}_S$ by $\mathrm{PGL}(N)$ is isomorphic to $\PP((\overline{b^* \mathrm{\Upsilon}})^{-1} (0))^s \sslash \mathrm{PO}(2)$.
\end{enumerate}
\end{theorem}
\begin{proof}
Claim (2) follows from (1) by the argument of \cref{rem:triple description of the centres}.

Let us prove (1). Since $h \colon U \to \mathrm{\Sigma}_R^{ss}$ is a (universal) geometric quotient by $\mathrm{PO}(2)$, we have an isomorphism $C_{\mathrm{\Sigma}_R^{ss}/R^{ss}} \simeq (h^*C_{\mathrm{\Sigma}_R^{ss}/R^{ss}}) \sslash \mathrm{PO}(2)$. We are done if we prove that the isomorphism \eqref{description normal cone SigmaR in R} is $\mathrm{PO}(2) \times \mathrm{PGL}(N)$-equivariant.

The commuting actions of $\mathrm{O}(2)$ and $\mathrm{GL}(N)$ on $U=\PP\sIsom(\Oca^{\oplus N}, p_{\Xca^{[n]},*} b_X^*[\Ica_1 (k) \oplus \Ica_2 (k)])$ were deduced from the natural $\Orm(2)$-equivariant structure on $\Ica_1 \oplus \Ica_2$ and $\mathrm{GL}(N)$-equivariant structure on $\Oca^{\oplus N}$. The tautological line bundle $\Oca_u(1)$ on $U$ is in particular an $\mathrm{O}(2) \times \mathrm{GL}(N)$-equivariant sheaf.

Hence, the sheaf $u_X^* b_X^* (\Ica_1 \oplus \Ica_2) \otimes p_U^* \Oca_u(1)$ on $U \times X$ is naturally $\mathrm{O}(2) \times \mathrm{GL}(N)$-equivariant, as is also $\Fca_U$: indeed, the tautological sheaf $\Fca$ on $Q^{ss} \times X$ is $\mathrm{GL}(N)$-equivariant, and the morphism $U \to Q^{ss}$ constructed in the proof of \cref{prop:global description of SigmaR} is $\mathrm{GL}(N)$-equivariant and $\mathrm{O}(2)$-invariant. The isomorphism \eqref{isomorphism SigmaR FU} relating these two $\Orm(2) \times \mathrm{GL}(N)$-equivariant sheaves is equivariant. 

As a consequence, the isomorphisms of \cref{lem:SigmaR ExtEE} and \cref{lem:SigmaR-ExtFF} are $\mathrm{PO}(2) \times \mathrm{PGL}(N)$-equivariant, and so is the induced isomorphism \eqref{description conormal sheaf SigmaR in R}. Since the map \eqref{Yoneda dual} is $\mathrm{PO}(2)$-equivariant, the isomorphism \eqref{description normal cone SigmaR in R} is $\mathrm{PO}(2) \times \mathrm{PGL}(N)$-equivariant, as desired.
\end{proof}
Let us express $\PP((\overline{b^* \mathrm{\Upsilon}})^{-1} (0))^s \sslash \mathrm{PO}(2)$ -- and hence $\mathrm{\Sigma}_S^s \sslash \mathrm{PGL}(N)$ -- more explicitly. Under the $\mathrm{PO}(2)$-equivariant\footnote{\cref{thm:properties of V}(1) yields an isomorphism $V^\vee \simeq V'$. The $\mathrm{PO}(2)$-equivariant structure of $V \oplus V^\vee$ is then deduced from that of $V \oplus V'$. In particular, we get an isomorphism $\widetilde{\iota}^*(V \oplus V^\vee) \xrightarrow{\sim} V\oplus V^\vee$, whose non-trivial components are $\varphi \colon \widetilde{\iota}^*V \xrightarrow{\sim} V' \xrightarrow{\sim} V^\vee$ and $(\widetilde{\iota}^*\varphi)^{-1} \colon \widetilde{\iota}^*(V^\vee) \xrightarrow{\sim} \widetilde{\iota}^*V' \xrightarrow{\sim} V$. A word of caution: since in our situation the natural action of $\widetilde{\iota}^*$ on $\Oca$ is twisted by the non-trivial character of $\bm{\mu}_2$, $\widetilde{\iota}^*\varphi$ differs from $\varphi^\vee$ (in fact, $\widetilde{\iota}^*\varphi = -\varphi^\vee$).} isomorphisms
\begin{equation*}
\PP(V \oplus V^\vee) \xrightarrow{\sim} \PP(V \oplus V') \xrightarrow{\sim} \PP((V \oplus V')(-E)),
\end{equation*}
the inverse image of $\PP((\overline{b^* \mathrm{\Upsilon}})^{-1} (0))$ is $\PP(ev^{-1}(0))$, where $ev$ is the canonical evaluation map
\[
ev \colon V \oplus V^\vee \to \Oca, \quad (v, \xi) \mapsto \xi(v).
\]
Now, the quotient of $\PP(ev^{-1}(0))$ by $\mathrm{PO}(2)$ can be computed in two steps: first by $\mathrm{PSO}(2)$, then by $\mathrm{PO}(2)/\mathrm{PSO}(2) \simeq \bm{\mu}_2$. The former gives the incidence divisor over $\Xca^{[n]}$
\begin{equation}\label{definition I}
\mathbb{I} \hookrightarrow \PP (V) \times_{\Xca^{[n]}} \PP (V^\vee) \simeq \PP(V) \times_{\Xca^{[n]}} \PP(V') \simeq \PP(V(-E)) \times _{\Xca^{[n]}} \PP(V'(-E)).
\end{equation}
The generator of $\bm{\mu}_2$ acts on $\mathbb{I}$ via the isomorphism $\widetilde{\iota}^*\left[\PP (V) \times_{\Xca^{[n]}} \PP (V^\vee)\right] \simeq \PP (V^\vee) \times_{\Xca^{[n]}} \PP (V)$.
\begin{corollary}\label{cor:SigmaS modulo PGL(N)}
The GIT quotient of $\mathrm{\Sigma}_S$ by $\mathrm{PGL}(N)$ is isomorphic to $\mathbb{I}/\bm{\mu}_2$.
\end{corollary}

\subsection{Consequences}\label{sec:SigmaS-OmegaS;SigmaS-DeltaS} We provide now a global description of $(\mathrm{\Omega}_S \cap \mathrm{\Sigma}_S)^s \sslash \mathrm{PGL}(N)$. Comparing it with the results of \cref{section:second blow-up}, we deduce a global description of $(\mathrm{\Delta}_S \cap \mathrm{\Sigma}_S)^s \sslash \mathrm{PGL}(N)$.

\subsubsection{Description of $(\mathrm{\Omega}_S \cap \mathrm{\Sigma}_S)^s \sslash \mathrm{PGL}(N)$} Since $\mathrm{\Omega}_R^{ss}$ (resp.\ $\mathrm{\Omega}_R \cap \mathrm{\Sigma}_R^{ss}$) is a Cartier divisor on $R^{ss}$ (resp.\ $\mathrm{\Sigma}_R^{ss}$) and $R^{ss}$ is normally flat along $\mathrm{\Sigma}_R^{ss}$, we have an isomorphism $C_{\mathrm{\Sigma}_R^{ss} \cap \mathrm{\Omega}_R/\mathrm{\Omega}_R^{ss}} \simeq \imath^*C_{\mathrm{\Sigma}_R^{ss}/R^{ss}}$ (see \cite[Theorem 21.10]{HOIebu}), which is $\mathrm{PGL}(N)$-equivariant; hence, we have a $\mathrm{PO}(2) \times \mathrm{PGL}(N)$-equivariant isomorphism of $D$-schemes
\[
\eta^*(C_{\mathrm{\Sigma}_R^{ss} \cap \mathrm{\Omega}_R/\mathrm{\Omega}_R^{ss}}) \simeq \zeta^*u^*[(\overline{b^*\mathrm{\Upsilon}})^{-1} (0)].
\]
As $D$ is the base change of the $X^{[n]}$-scheme $E$ along $\upsilon \colon Y \to X^{[n]}$ (see \eqref{principal PGL(N)-bundles}), the right-hand term can be written as $\upsilon^*[(\overline{b^*\mathrm{\Upsilon}}|_E)^{-1} (0)]$. Reasoning exactly as in the previous section, we get the isomorphisms
\begin{gather*}
\mathrm{\Omega}_S \cap \pi_S^{-1}(\mathrm{\Sigma}_R^{ss})  \simeq \PP( \upsilon^*[(\overline{b^*\mathrm{\Upsilon}}|_E)^{-1} (0)]) \sslash \mathrm{PO}(2), \\ 
(\mathrm{\Omega}_S \cap \mathrm{\Sigma}_S)^s \sslash \mathrm{PGL}(N)  \simeq \PP((\overline{b^* \mathrm{\Upsilon}}|_E)^{-1} (0))^s \sslash \mathrm{PO}(2).
\end{gather*}
Again, we can perform the latter quotient in two steps: first, we get the incidence divisor over $E$
\begin{multline}\label{definition J}
\mathbb{J} \hookrightarrow \PP (\Lca^{\perp_\omega}/\Lca) \times_E \PP ((\Lca^{\perp_\omega}/\Lca)^\vee)\\
 \simeq \PP (\Lca^{\perp_\omega}/\Lca) \times_E \PP (\Lca^{\perp_\omega}/\Lca) \simeq \PP((\Lca^{\perp_\omega} / \Lca) \otimes \Lca^\vee) \times_E \PP((\Lca^{\perp_\omega} / \Lca) \otimes \Lca^\vee);
\end{multline}
second, the quotient of $\mathbb{J}$ by the involution interchanging the factors of $\PP (\Lca^{\perp_\omega}/\Lca) \times_E \PP (\Lca^{\perp_\omega}/\Lca)$. 
\begin{corollary}\label{cor:SigmaS-OmegaS modulo PGL(N)} We have isomorphisms
\begin{enumerate}
\item[(1)] $\mathrm{\Omega}_S \cap \pi_S^{-1}(\mathrm{\Sigma}_R^{ss}) \simeq \PP(\upsilon^*[(\overline{b^*\mathrm{\Upsilon}}|_E)^{-1} (0)]) \sslash \mathrm{PO}(2)$ and
\item[(2)] $(\mathrm{\Omega}_S \cap \mathrm{\Sigma}_S)^s \sslash \mathrm{PGL}(N) \simeq \mathbb{J}/\bm{\mu}_2$.
\end{enumerate}
\end{corollary}

Let us explain how \cref{cor:SigmaS-OmegaS modulo PGL(N)} compares with the global descriptions, given in \cref{prop:OmegaS;OmegaS-SigmaS}, of $\mathrm{\Omega}_S \cap \mathrm{\Sigma}_S$  as $\upsilon^* E^\Tca \sslash \mathrm{SO}(W)$, and of $(\mathrm{\Omega}_S \cap \mathrm{\Sigma}_S)^s \sslash \mathrm{PGL}(N)$ as $(E^\Tca)^s \sslash \mathrm{SO}(W)$ (recall from \ref{notation blow-ups} that $E^\Tca$ denotes the exceptional divisor of the blow-up of $\PP\sHom^\omega(W,\Tca)$ along $\PP\sHom_1(W,\Tca)$; the stable points of $E^\Tca$ are fibrewise computed in \cite[Lemma 1.8.6(1)]{O'GdmsK3}). 

\smallskip

 Identifying $\PP(W)$ with $\PP(W^\vee)$ via the Killing form $\kappa$ of $W$, we have a natural $\mathrm{SO}(W)$-equivariant morphism
\begin{equation}\label{exceptional divisor ET}
E^\Tca \to \PP\sHom_1(W,\Tca) \simeq \PP(W) \times \PP(\Tca);
\end{equation}
by \cref{prop:Kirwan blow-up}(i), it sends $(E^\Tca)^s=(E^\Tca)^{ss}$ to $\PP\sHom_1(W,\Tca)^{ss} \simeq \PP(W)^{ss} \times \PP(\Tca)$.
As shown in the proof of \cref{lemma:OmegaR-SigmaR}, the semistable locus for the tautological action of $\mathrm{SO}(W)$ on $\PP(W)$ is the complement of the smooth conic defined by $\kappa$. Pick the point $[w] \in \PP(W)^{ss}$ given by
 \begin{equation}\label{distinguished point}
w = 
\begin{pmatrix}
1 & 0 \\
0 & -1
\end{pmatrix} \in W.
\end{equation}
The stabiliser of $[w]$ is $\mathrm{PO}(2)$, with imbedding in $\mathrm{SO}(W) \simeq \mathrm{PGL}(2)$ induced by \eqref{O(2)}.
If we base change \eqref{exceptional divisor ET} along $\PP(\Tca) \simeq [w] \times \PP(\Tca) \hookrightarrow \PP(W) \times \PP(\Tca)$, we get a morphism 
\[
E_{[w]}^\Tca \to \PP(\Tca),
\]
together with a $\mathrm{PO}(2)$-action on $E_{[w]}^\Tca$. Since $\mathrm{SO}(W)$ acts transitively on $\PP(W)^{ss}$, and $(E^\Tca)^s$ equals $(E^\Tca)^{ss}$, we obtain
\begin{gather*}
(\mathrm{\Omega}_S \cap \mathrm{\Sigma}_S)^s \sslash \mathrm{PGL}(N) \simeq (E^\Tca)^{s} \sslash \mathrm{SO}(W)  \simeq (E^\Tca_{[w]})^{s} \sslash \mathrm{PO}(2),\\
(\mathrm{\Omega}_S \cap \mathrm{\Sigma}_S)^{s}  \simeq \upsilon^*[(E^\Tca)^{s}] \sslash \mathrm{SO}(W)  \simeq \upsilon^*[(E^\Tca_{[w]})^{s}] \sslash \mathrm{PO}(2).
\end{gather*}
The desired comparison follows then from the $\mathrm{PO}(2)$-equivariant isomorphism over $E=\PP(\Tca)$
\[
E_{[w]}^\Tca \simeq \PP((\overline{b^* \mathrm{\Upsilon}}|_E)^{-1} (0)).
\]
\begin{remark}\label{rem:tangent space of P(W) at distinguished point}
Taking $[w] \in \PP(W)^{ss}$ as distinguished point for the transitive action of $\mathrm{SO}(W) \simeq \mathrm{PGL}(2)$, we can write $ \PP(W)^{ss} \simeq \mathrm{PGL}(2)/\mathrm{PO}(2)$. The inclusion $\mathrm{PO}(2) \subset \mathrm{PGL}(2)$ gives an inclusion $\mathbb{C}w \subset W$ of their respective Lie algebras, and yields a natural $\mathrm{PO}(2)$-equivariant isomorphism $\mathrm{T}_{\PP(W)^{ss},[w]} \simeq W/\mathbb{C}w$.

In particular, the $\mathrm{SO}(W)$-equivariant imbedding 
\[
E^\Tca \hookrightarrow  \PP(N_{\PP(W) \times \PP(\Tca) / \PP\sHom^\omega(W,\Tca)}) = \PP(\Tca_{\PP (W)} \boxtimes [(\Lca^{\perp_\omega} / \Lca) \otimes \Lca^\vee])
\]
restricts to the $\mathrm{PO}(2)$-equivariant imbedding 
\[
E_{[w]}^\Tca \hookrightarrow \PP((W/\mathbb{C}w) \otimes [(\Lca^{\perp_\omega} / \Lca) \otimes \Lca^\vee]).
\]
\end{remark}

\subsubsection{Description of $(\mathrm{\Delta}_S \cap \mathrm{\Sigma}_S)^s \sslash \mathrm{PGL}(N)$} Consider now the exceptional divisor $E_2^\Tca$ of the blow-up of $\PP\sHom_2^\omega(W,\Tca)$ along $\PP\sHom_1(W,\Tca)$. We reason as above: pick $w \in W$ as in \eqref{distinguished point}; the restriction of 
\[
E_2^\Tca \to \PP\sHom_1(W,\Tca) \simeq \PP(W) \times \PP(\Tca)
\]
along $\PP(\Tca) \simeq [w] \times \PP(\Tca) \hookrightarrow \PP(W) \times \PP(\Tca)$ yields a $\mathrm{PO}(2)$-invariant morphism 
\[
E_{2,[w]}^\Tca \to \PP(\Tca);
\]
from \cref{prop:DeltaS;DeltaS-SigmaS}, we obtain
\begin{gather*}
(\mathrm{\Delta}_S \cap \mathrm{\Sigma}_S)^s \sslash \mathrm{PGL}(N)  \simeq  (E_2^\Tca)^{s} \sslash \mathrm{SO}(W)  \simeq (E^\Tca_{2,[w]})^{s} \sslash \mathrm{PO}(2),\\
(\mathrm{\Delta}_S \cap \mathrm{\Sigma}_S)^{s}  \simeq \upsilon^*[(E_2^\Tca)^{s}] \sslash \mathrm{SO}(W)  \simeq \upsilon^*[(E^\Tca_{2,[w]})^{s}] \sslash \mathrm{PO}(2).
\end{gather*}
The quotient of $(E^\Tca_{2,[w]})^{s}$ by $\mathrm{PSO}(2)$ is isomorphic to $\PP((\Lca^{\perp_\omega}/\Lca) \otimes\Lca^\vee) \simeq \PP(\Lca^{\perp_\omega}/\Lca)$, imbedded in $\mathbb{J}$ diagonally. The residual $\bm{\mu}_2$-action is trivial. In the end, we get the following
\begin{corollary}\label{cor:DeltaS-SigmaS modulo PGLN}
The GIT quotient of $\mathrm{\Delta}_S \cap \mathrm{\Sigma}_S$ by $\mathrm{PGL}(N)$ is isomorphic to $\PP(\Lca^{\perp_\omega}/\Lca)$.
\end{corollary}
\begin{remark}\label{rem:principal bundles for Sigma S}
The good quotients considered in this section fit into a commutative diagram
\begin{equation}\label{SigmaS-OmegaS-DeltaS Quot and moduli}
\begin{tikzcd}[scale cd=0.9]
\PP (\Lca^{\perp_\omega}/\Lca) \ar[d,hook] & (E_{2,[w]}^\Tca)^{s} \ar[d,hook] \ar[l] & \upsilon^*[(E_{2,[w]}^\Tca)^{s}] \ar[d,hook] \ar[l] \ar[r] &  (\mathrm{\Delta}_S \cap \mathrm{\Sigma}_S)^{s} \ar[d, hook]\\
\mathbb{J} \ar[d,hook] & (E_{[w]}^\Tca)^{s} \ar[d,hook] \ar[l] & \upsilon^*[(E_{[w]}^\Tca)^{s}] \ar[d,hook] \ar[l] \ar[r] &  (\mathrm{\Omega}_S \cap \mathrm{\Sigma}_S)^{s} \ar[d,hook] \\
\mathbb{I} & \PP(\overline{b^* \mathrm{\Upsilon}}^{-1} (0))^{s} \ar[l] & u^*[\PP(\overline{b^* \mathrm{\Upsilon}}^{-1} (0))^{s}] \ar[l] \ar[r] &  \mathrm{\Sigma}_S^{s}.
\end{tikzcd}
\end{equation}
The actions of $\mathrm{PSO}(2)$ on $\PP(\overline{b^* \mathrm{\Upsilon}}^{-1} (0))^{s}$, and of $\mathrm{PGL}(N)$ and $\mathrm{PO}(2)$ on $u^*[\PP(\overline{b^* \mathrm{\Upsilon}}^{-1} (0))^{s}]$ are all set-theoretically free. In particular, by \cref{Rem:set-theoretic free implies free}, the horizontal arrows on the left, in the middle and on the right of \eqref{SigmaS-OmegaS-DeltaS Quot and moduli} are principal $\mathrm{PSO}(2)$-, $\mathrm{PGL}(N)$- and $\mathrm{PO}(2)$-bundles, respectively.
\end{remark}

\subsection{Summary}
Let us denote by $r \colon \mathbb{I} \to \Xca^{[n]}$ the natural morphism, and by $\rho \colon \mathbb{J} \to E$ and $\varrho \colon \PP (\Lca^{\perp_\omega}/\Lca) \to E$ its restrictions to $\mathbb{J}$ and to the diagonal $\PP (\Lca^{\perp_\omega}/\Lca)$, respectively. At the level of the moduli space, the results of this long section can be summed up in the following diagram:
\[
\begin{tikzcd}[column sep= 1.0em, row sep = 1.0 em, scale cd=0.9]
\PP (\Lca^{\perp_\omega}/\Lca) \ar[r, hook] \ar[rdd,"\varrho"'] & \mathbb{J} \ar[rr, hook] \ar[rd] \ar[dd,"\rho"] & & \mathbb{I}\ar[dd,"r"{yshift=12pt}] \ar[rd] &\\
& & \mathbb{J}/\bm{\mu}_2 \ar[rr, hook, crossing over]  & & \mathbb{I}/\bm{\mu}_2 \ar[dd] \ar[r,hook]& S^s \sslash \mathrm{PGL}(N) \ar[dd] \\
& E \ar[rd,equal] \ar[rr, hook] \ar[dd,"\beta"]  & & \Xca^{[n]} \ar[rd] \ar[dd,"b" {yshift=12pt}]& \\
& & E \ar[rr, hook, crossing over]  \ar[from=uu, crossing over] & & (X^{[n]})^{[2]} \ar[dd] \ar[r,hook] & R^{ss} \sslash \mathrm{PGL}(N) \ar[dd] & \\
& X^{[n]} \ar[rd,equal] \ar[rr, hook]  & & X^{[n]} \times X^{[n]} \ar[rd] \\
& & X^{[n]} \ar[rr, hook] \ar[from=uu, crossing over] & & \mathrm{S}^2 X^{[n]} \ar[r,hook] & M_n. &
\end{tikzcd}
\]

\section{Global descriptions of $\widehat{\mathrm{\Sigma}}$ and $\widehat{\mathrm{\Delta}}$}\label{sec:description of Sigma hat and Delta hat}
Recall from \cref{section:third blow-up} that $\pi_T \colon T \to S$ denotes the blow-up of $S$ along $\mathrm{\Delta}_S$, $\mathrm{\Delta}_T$ its exceptional divisor, and $\mathrm{\Sigma}_T\simeq \mathrm{Bl}_{\mathrm{\Delta}_S \cap \mathrm{\Sigma}_S} \mathrm{\Sigma}_S$ (resp.\ $\mathrm{\Omega}_T \simeq \mathrm{Bl}_{\mathrm{\Delta}_S} \mathrm{\Omega}_S$) the strict transform of $\mathrm{\Sigma}_S$ (resp.\ $\mathrm{\Omega}_S$). In this section, we provide global descriptions of 
\[
\widehat{\mathrm{\Sigma}} \coloneqq \mathrm{\Sigma}_T^s \sslash \mathrm{PGL}(N) \ \ \ \text{and} \ \ \ \widehat{\mathrm{\Delta}} \coloneqq \mathrm{\Delta}_T^s \sslash \mathrm{PGL}(N).
\]
These are, together with $\widehat{\mathrm{\Omega}} \coloneqq \mathrm{\Omega}_T^s \sslash \mathrm{PGL}(N)$, the irreducible components of the exceptional locus of the resolution of singularities $\widehat{\pi}\colon \widehat{M}_n \to M_n$.
\subsection{Global description of $\widehat{\mathrm{\Sigma}}$}\label{section:Sigma hat} 
Let $\mathbb{I}$ be as in \eqref{definition I}, and $\mathbb{J}$ as in \eqref{definition J}. Let $\widehat{\mathbb{I}}$ be the blow-up of $\mathbb{I}$ along $\mathbb{I}^{\bm{\mu}_2} = \PP(\Lca^{\perp_\omega}/\Lca)$, and $\widehat{\mathbb{E}}$ be its exceptional divisor. Let $\widehat{\mathbb{J}}$ be the strict transform of $\mathbb{J}$: it is the blow-up of $\mathbb{J}$ along $\mathbb{I}^{\bm{\mu}_2}$, and its exceptional divisor is $\widehat{\mathbb{J}} \cap \widehat{\mathbb{E}}$. The $\bm{\mu}_2$-action on $\mathbb{I}$ and $\mathbb{J}$ leaves $\PP(\Lca^{\perp_\omega} / \Lca)$ invariant, and hence lifts to $\widehat{\mathbb{I}}$ and $\widehat{\mathbb{J}}$, respectively.
\begin{theorem}\label{thm:Sigma hat} We have isomorphisms of closed subschemes of $\widehat{M}_n$
\begin{enumerate}
\item[(1)] $\widehat{\mathrm{\Sigma}} \simeq \widehat{\mathbb{I}}/ \bm{\mu}_2$,
\item[(2)] $\widehat{\mathrm{\Sigma}} \cap \widehat{\mathrm{\Delta}} \simeq \widehat{\mathbb{E}}/\bm{\mu}_2$,
\item[(3)] $\widehat{\mathrm{\Sigma}} \cap \widehat{\mathrm{\Omega}} \simeq \widehat{\mathbb{J}}/\bm{\mu}_2$,
\item[(4)] $\widehat{\mathrm{\Sigma}} \cap \widehat{\mathrm{\Delta}} \cap \widehat{\mathrm{\Omega}} \simeq (\widehat{\mathbb{E}}\cap \widehat{\mathbb{J}})/\bm{\mu}_2$.
\end{enumerate}
\end{theorem}
\begin{proof}
The theorem follows from \cref{rem:principal bundles for Sigma S} by \cref{prop:blow-up and principal G bundles} and \cref{prop:commuting actions}. Let us explain in some detail how to prove the first two statements. Recall that $\widehat{\mathrm{\Sigma}}$ and $\widehat{\mathrm{\Sigma}} \cap \widehat{\mathrm{\Delta}}$ are good quotients of $\mathrm{\Sigma}_T^{s}$ and of $(\mathrm{\Delta}_T \cap \mathrm{\Sigma}_T)^{s}$ by $\mathrm{PGL}(N)$, respectively. 

Look at the diagram \eqref{SigmaS-OmegaS-DeltaS Quot and moduli}, and consider the blow-up of the schemes in the third row along the respective subschemes in the first row. Combining \cref{prop:blow-up and principal G bundles} and \cref{prop:commuting actions}, we deduce that $\widehat{\mathrm{\Sigma}}$ and $\widehat{\mathrm{\Sigma}} \cap \widehat{\mathrm{\Delta}}$ are quotients by $\mathrm{PO}(2)$ of the blow-up of $\PP(\overline{b^* \mathrm{\Upsilon}}^{-1} (0))^{s}$ along $(E^\Tca_{2,[w]})^{s}$ and of its exceptional divisor, respectively. As usual, we compute such quotients in two steps: first by $\mathrm{PSO}(2)$, then by the residual action of $\bm{\mu}_2$.
Again by \cref{prop:blow-up and principal G bundles}, the quotients by $\mathrm{PSO}(2)$ are $\widehat{\mathbb{I}}$ and $\widehat{\mathbb{E}}$, because $\PP(\overline{b^* \mathrm{\Upsilon}}^{-1} (0))^{s}$ is a principal $\mathrm{PSO}(2)$-bundle on $\mathbb{I}$. Taking the quotient by $\bm{\mu}_2$, we get the desired result.

The last two statements are obtained similarly: use that $\widehat{\mathrm{\Sigma}} \cap \widehat{\mathrm{\Omega}}$ is a good $\mathrm{PGL}(N)$-quotient of $(\mathrm{\Omega}_T \cap \mathrm{\Sigma}_T)^{s}$ -- and hence, by \cref{lem:OmegaT-DeltaT}, of $\mathrm{Bl}_{(\mathrm{\Delta}_S \cap \mathrm{\Sigma}_S)^{s}}[(\mathrm{\Omega}_S \cap \mathrm{\Sigma}_S)^{s}]$ -- and consider the blow-up of the schemes in the second row of \eqref{SigmaS-OmegaS-DeltaS Quot and moduli} along the respective subschemes in the first row.
\end{proof}

\subsection{Global description of $\widehat{\mathrm{\Delta}}$}\label{sec:Delta hat}
Firstly, let us compare the description of $(\mathrm{\Delta}_S \cap \mathrm{\Sigma}_S)^s \sslash \mathrm{PGL}(N)$ obtained in \cref{cor:DeltaS-SigmaS modulo PGLN} with the one given in \cref{prop:DeltaS}.

Let $\alpha \colon \Grass^\omega(2, \Tca) \to X^{[n]}$ be the relative symplectic Grassmannian, $\Aca \subset \alpha^*\Tca$ its tautological rank-two bundle, $\Aca^{\perp_\omega}$ the orthogonal complement of $\Aca$ with respect to $\alpha^*\omega$ and $\Oca_\alpha(1)\coloneqq \wedge^2 \Aca^\vee$. Consider the projective bundle $\vartheta \colon \PP (\Aca) \to \Grass^\omega(2, \Tca)$,  with tautological quotient line bundle $\Oca_\vartheta(1)$ and relative tangent bundle $\Tca_\vartheta$. The composition $\Oca_\vartheta(-1) \hookrightarrow \vartheta^* \Aca \hookrightarrow \vartheta^*\alpha^*\Tca$ is fibrewise injective; by the universal property of the projective bundle $\PP (\Tca)$, we get a morphism 
\[
\lambda \colon \PP(\Aca) \to \PP(\Tca)
\]
over $X^{[n]}$ such that $\lambda^*\Lca^\vee \simeq \Oca_\vartheta(1)$. Note that, because of the inclusions $\Aca \subset \Aca^{\perp_\omega}$, $\lambda^*\Lca \subset \vartheta^*\Aca$ and hence $\vartheta^*(\Aca^{\perp_\omega}) \subset \lambda^*(\Lca^{\perp_\omega})$, we have also $\vartheta^* \Aca \hookrightarrow \lambda^*(\Lca^{\perp_\omega})$; thanks to the relative Euler sequence, we easily deduce an injective homomorphism $\Tca_\vartheta \hookrightarrow \lambda^*[(\Lca^{\perp_\omega} / \Lca) \otimes \Lca^\vee]$.

If we take the product of $\lambda$ with $\PP (W)$, identified as usual with $\PP (W^\vee)$ via the Killing form, we get a morphism $\lambda_{\PP (W)} \colon \PP \sHom_1(W,\Aca) \to \PP\sHom_1(W,\Tca)$. It is induced by the natural forgetful $X^{[n]}$-morphism $\PP \sHom(W,\Aca) \to \PP \sHom_2^\omega(W,\Tca)$ that was lifted to an isomorphism $\mathrm{Bl}^\Aca \xrightarrow{\sim} \mathrm{Bl}_2^\Tca$ in \ref{sec:Grass(2,T)}. In particular, $\lambda_{\PP (W)}$ lifts to an $\mathrm{SO}(W)$-equivariant isomorphism
\begin{equation}\label{morphism Delta_S-Sigma_S}
E^\Aca = \PP(\Tca_{\PP (W)} \boxtimes \Tca_\vartheta) \xrightarrow{\sim} E_2^\Tca \left(\hookrightarrow \PP(\Tca_{\PP (W)} \boxtimes [(\Lca^{\perp_\omega} / \Lca) \otimes \Lca^\vee])\right)
\end{equation}
between the respective exceptional divisors.
This isomorphism of $\PP(W) \times \PP(\Tca)$-schemes comes from the box product of the tangent bundle $\Tca_{\PP (W)}$ of $\PP (W)$ and the aforementioned inclusion $\Tca_\vartheta \hookrightarrow \lambda^*[(\Lca^{\perp_\omega} / \Lca) \otimes \Lca^\vee]$.

\medskip

After \cref{prop:DeltaS;DeltaS-SigmaS}(4) and \cref{prop:DeltaS}(4), in order to get to $(\mathrm{\Delta}_S \cap \mathrm{\Sigma}_S)^s \sslash \mathrm{PGL}(N)$, we need to compute the GIT quotient of $E^\Aca \xrightarrow{\sim} E_2^\Tca$ by $\mathrm{SO}(W)$. This is done as in \cref{sec:SigmaS-OmegaS;SigmaS-DeltaS}: for $w$ as in \eqref{distinguished point}, we first base change \eqref{morphism Delta_S-Sigma_S} along $E=\PP(\Tca) \simeq [w] \times \PP (\Tca) \hookrightarrow \PP(W) \times \PP(\Tca)$, obtaining
\[
E^\Aca_{[w]} = \PP((W/\mathbb{C}w) \otimes \Tca_\vartheta) \xrightarrow{\sim} E_{2,[w]}^\Tca \left(\hookrightarrow \PP((W/\mathbb{C}w) \otimes [(\Lca^{\perp_\omega} / \Lca) \otimes \Lca^\vee])\right)
\]
(\cf \cref{rem:tangent space of P(W) at distinguished point}); then we take the GIT quotient by $\mathrm{St}([w]) = \mathrm{PO}(2)$ in two steps as usual. The quotient by $\mathrm{PSO}(2)$ yields a $\bm{\mu}_2$-equivariant closed imbedding
\begin{equation*}
\PP(\Tca_\vartheta) \times_{\PP(\Aca)}\PP(\Tca_\vartheta) \hookrightarrow \PP((\Lca^{\perp_\omega} / \Lca) \otimes \Lca^\vee) \times_E \PP((\Lca^{\perp_\omega} / \Lca) \otimes \Lca^\vee),
\end{equation*}
which is an isomorphism onto the diagonal of $\PP((\Lca^{\perp_\omega} / \Lca) \otimes \Lca^\vee) \times_E \PP((\Lca^{\perp_\omega} / \Lca) \otimes \Lca^\vee)$, i.e.\ onto the fixed locus for the $\bm{\mu}_2$-action.
Identifying $\PP (\Aca)$ with $\PP(\Tca_\vartheta) \times_{\PP(\Aca)}\PP(\Tca_\vartheta)$ and $\PP((\Lca^{\perp_\omega} / \Lca) \otimes \Lca^\vee)$ with $\PP(\Lca^{\perp_\omega} / \Lca)$, we end up with the following

\begin{proposition}\label{proposition:DeltaS-SigmaS}
We have an isomorphism
\[
\mathrm{\Lambda} \colon \PP (\Aca) \xrightarrow{\sim}\PP(\Lca^{\perp_\omega} / \Lca)
\]
comparing the descriptions of $(\mathrm{\Delta}_S \cap \mathrm{\Sigma}_S)^s \sslash \mathrm{PGL}(N)$ given in \cref{prop:DeltaS} and \cref{cor:DeltaS-SigmaS modulo PGLN}. 

The pullback $\mathrm{\Lambda}^*$ sends $\varrho^*\Lca^\vee$ to $\Oca_\vartheta(1)$, and the tautological quotient line bundle $\Oca_\varrho(1)$ of $\varrho \colon  \PP(\Lca^{\perp_\omega} / \Lca) \to \PP(\Tca)$ to $(\vartheta^*\Aca/\lambda^*\Lca)^\vee \simeq \vartheta^*\Oca_\alpha(1) \otimes \Oca_\vartheta(-1)$.
\end{proposition}

To provide a global description of $\widehat{\mathrm{\Delta}}$, we need now a detailed analysis of $\widehat{\mathrm{\Delta}} \cap \widehat{\mathrm{\Omega}}$, $\widehat{\mathrm{\Delta}} \cap \widehat{\mathrm{\Sigma}}$ and $\widehat{\mathrm{\Delta}} \cap \widehat{\mathrm{\Sigma}} \cap \widehat{\mathrm{\Omega}}$. 
The first intersection was described in \cref{prop:Delta hat}. The other two, by \cref{thm:Sigma hat}, are the $\bm{\mu}_2$-quotient of the exceptional divisors of the blow-ups of $\mathbb{I}$ and $\mathbb{J}$ along $\PP(\Lca^{\perp_\omega}/\Lca)$. Thus, it is important to compute the normal bundles $\Nca_{\PP(\Lca^{\perp_\omega} / \Lca )/\mathbb{I}}$ and $\Nca_{\PP(\Lca^{\perp_\omega}/\Lca)/\mathbb{J}}$.
\begin{lemma}\label{lem:normal bundles in I and J}
We have isomorphisms
\begin{gather*}
\mathrm{\Lambda}^* \Nca_{\PP(\Lca^{\perp_\omega}/\Lca)/\mathbb{J}} \simeq \vartheta^*[(\Aca^{\perp_\omega}/\Aca) \otimes \Oca_\alpha(1)] \otimes \Oca_\vartheta(-1), \\
\mathrm{\Lambda}^*\Nca_{\PP(\Lca^{\perp_\omega} / \Lca )/\mathbb{I}} \simeq \vartheta^*\Gca \otimes \Oca_\vartheta(-1),
\end{gather*}
where $\Gca$ is a locally free sheaf on $\Grass^\omega(2,\Tca)$ sitting into a short exact sequence
\[
\begin{tikzcd}
0 \ar[r] & (\Aca^{\perp_\omega}/\Aca) \otimes \Oca_\alpha(1) \ar[r] & \Gca \ar[r] & \Oca \ar[r] & 0.
\end{tikzcd}
\]
\end{lemma}

\begin{proof}
The inclusions of smooth varieties $\PP(\Lca^{\perp_\omega} / \Lca) \hookrightarrow \mathbb{J} \hookrightarrow \mathbb{I}$ yield a short exact sequence
\begin{equation}\label{ses of normal bundles}
\begin{tikzcd}
0 \ar[r] & \Nca_{\PP(\Lca^{\perp_\omega} / \Lca)/\mathbb{J}} \ar[r] & \Nca_{\PP(\Lca^{\perp_\omega} / \Lca)/\mathbb{I}} \ar[r] & \Nca_{\mathbb{J}/\mathbb{I}}|_{\PP(\Lca^{\perp_\omega} / \Lca)} \ar[r] & 0.
\end{tikzcd}
\end{equation}
Since $r \colon \mathbb{I} \to \Xca^{[n]}$ is flat, and since its base change along $E \hookrightarrow \Xca^{[n]}$ is $\rho \colon \mathbb{J} \to E$, we have $\Nca_{\mathbb{J}/\mathbb{I}} \simeq \rho^* \Nca_{E/\Xca^{[n]}}= \rho^*\Lca$; hence, $\Nca_{\mathbb{J}/\mathbb{I}}|_{\PP(\Lca^{\perp_\omega} / \Lca)} = \varrho^*\Lca$, which corresponds to $\Oca_\vartheta(-1)$ under $\mathrm{\Lambda}$ by \cref{proposition:DeltaS-SigmaS}.

Let us move on to $\Nca_{\PP(\Lca^{\perp_\omega}/\Lca)/\mathbb{J}}$. Consider the commutative diagram with exact second row
\begin{equation}\label{J as projective bundle}
\begin{tikzcd}
0 \ar[r] & \Oca_\varrho(-1) \ar[r] \ar[d, hook] &\varrho^*(\Lca^{\perp_\omega}/\Lca) \isoarrow{d} & & \\
0 \ar[r] & \Oca_\varrho(-1)^\perp \ar[r] & \varrho^*(\Lca^{\perp_\omega}/\Lca)^\vee \ar[r] & \Oca_\varrho(1) \ar[r] & 0.
\end{tikzcd}
\end{equation}
The incidence divisor $\mathbb{J}$ in $\PP (\Lca^{\perp_\omega}/\Lca) \times_E \PP ((\Lca^{\perp_\omega}/\Lca)^\vee)$, regarded as a projective bundle over $\PP(\Lca^{\perp_\omega}/\Lca)$, is isomorphic to $\PP(\Oca_\varrho(-1)^\perp)$. Furthermore, $\PP(\Lca^{\perp_\omega}/\Lca)$ is imbedded in $\mathbb{J}$ as the section corresponding to the inclusion $\Oca_\varrho(-1) \hookrightarrow \Oca_\varrho(-1)^\perp$. If we denote by $\Qca$ its cokernel, we can express the normal bundle to $\PP(\Lca^{\perp_\omega}/\Lca)$ in $\mathbb{J}$ as
\[
\Nca_{\PP(\Lca^{\perp_\omega}/\Lca)/\mathbb{J}} \simeq \Qca \otimes \Oca_\varrho(1).
\]
The isomorphism $\mathrm{\Lambda} \colon \PP(\Aca) \xrightarrow{\sim} \PP(\Lca^{\perp_\omega}/\Lca)$ yields an alternative description of $\Nca_{\PP(\Lca^{\perp_\omega}/\Lca)/\mathbb{J}}$. First, we claim that 
\begin{equation}\label{Pullback of Q by Lambda}
\mathrm{\Lambda}^*\Qca \simeq \vartheta^*( \Aca^{\perp_\omega}/ \Aca).
\end{equation}
The pullback of \eqref{J as projective bundle} by $\mathrm{\Lambda}$ can be completed to a morphism of short exact sequences
\[
\begin{tikzcd}
0 \ar[r] & \vartheta^*\Aca/ \lambda^*\Lca \ar[r] \ar[d, hook,"\mu'"] & \lambda^*(\Lca^{\perp_\omega}/\Lca) \isoarrow{d} \ar[r] & \lambda^*(\Lca^{\perp_\omega})/\vartheta^*\Aca \ar[r] \ar[d, "\mu''"] & 0 \\
0 \ar[r] & \mathrm{\Lambda}^*(\Oca_\varrho(-1)^\perp) \ar[r] & \lambda^*(\Lca^{\perp_\omega}/\Lca)^\vee \ar[r] & (\vartheta^*\Aca/ \lambda^*\Lca)^\vee \ar[r] & 0.
\end{tikzcd}
\]
By the snake lemma, $\coker(\mu')$ is isomorphic to $\ker(\mu'')$. We compute it applying the snake lemma to the morphism of short exact sequences
\[
\begin{tikzcd}
0 \ar[r] & \vartheta^*(\Aca^{\perp_\omega}/\Aca) \ar[r] \ar[d, "\nu"] & \vartheta^*(\alpha^*\Tca/\Aca) \ar[d,equal] \ar[r] & \vartheta^*\Aca^\vee \ar[r] \ar[d] & 0 \\
0 \ar[r] & \lambda^* (\Lca^{\perp_\omega})/\vartheta^*\Aca \ar[r] & (\beta \circ \lambda)^*\Tca/\vartheta^*\Aca \ar[r] &  \lambda^*\Lca^\vee \ar[r] & 0,
\end{tikzcd}
\]
which yields the short exact sequence
\[
\begin{tikzcd}
0 \ar[r] & \vartheta^*(\Aca^{\perp_\omega}/\Aca) \ar[r,"\nu"] & \lambda^*(\Lca^{\perp_\omega})/\vartheta^*\Aca \ar[r,"\mu''"] & (\vartheta^*\Aca/\lambda^*\Lca)^\vee \ar[r] & 0.
\end{tikzcd}
\]
We deduce that $\coker(\mu') \simeq \ker(\mu'') \simeq \vartheta^*(\Aca^{\perp_\omega}/\Aca)$. This proves the isomorphism \eqref{Pullback of Q by Lambda}, which in turn implies, via \cref{proposition:DeltaS-SigmaS},
\[
\mathrm{\Lambda}^* \Nca_{\PP(\Lca^{\perp_\omega}/\Lca)/\mathbb{J}} \simeq \vartheta^*[(\Aca^{\perp_\omega}/\Aca) \otimes \Oca_\alpha(1)] \otimes \Oca_\vartheta(-1).
\]
As a consequence, by \eqref{ses of normal bundles},
$\mathrm{\Lambda}^* \Nca_{\PP(\Lca^{\perp_\omega}/\Lca)/\mathbb{I}}$ can be regarded as an element in 
\[
\mathrm{Ext}^1(\Oca_\vartheta(-1), \vartheta^*[(\Aca^{\perp_\omega}/\Aca) \otimes \Oca_\alpha(1)] \otimes \Oca_\vartheta(-1));
\]
under the fully faithful functor $\vartheta^*(-) \otimes \Oca_\vartheta(-1)$, it corresponds to a unique extension $\Gca \in \mathrm{Ext}^1(\Oca,(\Aca^{\perp_\omega}/\Aca) \otimes \Oca_\alpha(1))$, which has to be locally free.
\end{proof}
We denote as $s \colon \PP(\Gca) \to \Grass^\omega(2,\Tca)$ the projective bundle associated to $\Gca$, and as $\Oca_s(1)$ the corresponding tautological quotient line bundle.
\begin{proposition}\label{cor:Delta-Sigma; Delta-Sigma-Omega} The following hold true.
\begin{enumerate}
\item[(1)] $\widehat{\mathbb{E}} \simeq \PP(\Aca) \times_{\Grass^\omega(2,\Tca)} \PP(\Gca)$, and $\Nca_{\widehat{\mathbb{E}}/\widehat{\mathbb{I}}}$ corresponds to $\Oca_\vartheta(-1) \boxtimes \Oca_s(-1)$. 
\item[(2)] $\widehat{\mathrm{\Sigma}} \cap \widehat{\mathrm{\Delta}} \simeq \PP(\Aca) \times_{\Grass^\omega(2,\Tca)} \PP(\Gca)$.

\smallskip

\item[(3)] $\widehat{\mathbb{E}} \cap \widehat{\mathbb{J}} \simeq \PP(\Aca) \times_{\Grass^\omega(2,\Tca)} \PP((\Aca^{\perp_\omega}/\Aca) \otimes \Oca_\alpha(1))$.
\item[(4)] $\widehat{\mathrm{\Sigma}} \cap \widehat{\mathrm{\Delta}} \cap \widehat{\mathrm{\Omega}} \simeq \PP(\Aca) \times_{\Grass^\omega(2,\Tca)} \PP((\Aca^{\perp_\omega}/\Aca) \otimes \Oca_\alpha(1))$.
\end{enumerate}
\end{proposition}
\begin{proof}
It suffices to go through the proof of \cref{lem:normal bundles in I and J}: statements $(1)$ and $(3)$ follow immediately; to prove $(2)$ and $(4)$, note that $\bm{\mu}_2$ acts on the fibres of $\Nca_{\PP(\Lca^{\perp_\omega} / \Lca)/\mathbb{J}}$, $\Nca_{\mathbb{J}/\mathbb{I}}|_{\PP(\Lca^{\perp_\omega} / \Lca)}$, and hence also on those of $\Nca_{\PP(\Lca^{\perp_\omega} / \Lca)/\mathbb{I}}$, as the multiplication by $-1$; as a consequence, $\bm{\mu}_2$ acts trivially on the associated projective bundles.
\end{proof}

\begin{proposition}\label{prop:Delta global}
The divisor $\widehat{\mathrm{\Delta}}$ of $\widehat{M}_n$ is isomorphic to $\PP (\mathrm{S}^2\Aca) \times_{\Grass^\omega(2,\Tca)} \PP (\Gca)$.
\end{proposition}
\begin{proof}
Consider the projective bundle $t \colon \PP(\mathrm{S}^2 \Aca) \to \Grass^\omega(2,\Tca)$ parametrising quadrics in the fibres of $\PP(\Aca^\vee)$ over $\Grass^\omega(2,\Tca)$. Let $i \colon \PP(\Aca) \to \PP(\mathrm{S}^2 \Aca)$ be the inclusion of the locus of rank-1 quadrics; in particular, $t \circ i = \vartheta$. 

By \cref{prop:Delta hat}, $\widehat{\mathrm{\Delta}} \simeq \PP (\Sca)$ for some vector bundle $\Sca$ on $\PP(\mathrm{S}^2\Aca)$ sitting into a short exact sequence
\begin{equation}\label{short exact sequence projective bundle Delta}
\begin{tikzcd}
0 \ar[r] & t^*(\Aca^{\perp_\omega}/\Aca) \otimes \Mca_1 \ar[r] & \Sca \ar[r] & \Oca \ar[r] & 0,
\end{tikzcd}
\end{equation}
where $\Mca_1$ is a suitable line bundle on $\PP(\mathrm{S}^2\Aca)$. 

The pullback of \eqref{short exact sequence projective bundle Delta} along $i$ coincides, a priori up to a twist by a line bundle $\Mca_2 \in \mathrm{Pic}(\PP (\Aca))$, with the short exact sequence
\[
\begin{tikzcd}
0 \ar[r] & \vartheta^*[(\Aca^{\perp_\omega}/\Aca) \otimes \Oca_\alpha(1)] \ar[r] & \vartheta^*\Gca \ar[r] & \Oca \ar[r] & 0.
\end{tikzcd}
\]
This implies that $\Mca_2$ is trivial, so that $i^*[t^*(\Aca^{\perp_\omega}/\Aca) \otimes \Mca_1] \simeq \vartheta^*[(\Aca^{\perp_\omega}/\Aca) \otimes \Oca_\alpha(1)]$. Now, the Picard groups of $\Grass^\omega(2,\Tca)$ and hence of $\PP(\Aca)$ are torsion-free: using \cite[Ex.\ III.12.4]{Rag} and then the injectivity of $\alpha^* \colon \mathrm{Pic}(X^{[n]}) \to \mathrm{Pic}(\Grass^\omega(2,\Tca))$, this comes from the torsion-freeness of the Picard groups of any fibre of $\alpha$ and of $X^{[n]}$. We deduce that $i^*\Mca_1\simeq \vartheta^*\Oca_\alpha(1)$, and also $\Mca_1 \simeq t^*\Oca_\alpha(1)$ because $i^* \colon \mathrm{Pic}(\PP(\mathrm{S}^2\Aca)) \to \mathrm{Pic}(\PP(\Aca))$ is injective.

Therefore, $\Sca$ can be regarded as an element in $ \mathrm{Ext}^1(\Oca,t^*[(\Aca^{\perp_\omega}/\Aca) \otimes \Oca_\alpha(1)])$. Since $i^*\Sca \simeq \vartheta^* \Gca$, we conclude that $\Sca \simeq t^*\Gca$ thanks to the commutative diagram
\[
\begin{tikzcd}
 \mathrm{Ext}^1(\Oca,(\Aca^{\perp_\omega}/\Aca) \otimes \Oca_\alpha(1)) \ar[d, "t^*"', "\sim"]  \ar[dr,"\vartheta^*", "\sim"']  \\
\mathrm{Ext}^1(\Oca,t^*[(\Aca^{\perp_\omega}/\Aca) \otimes \Oca_\alpha(1)]) \ar[r,"\Lder i^*"]  & \mathrm{Ext}^1(\Oca,\vartheta^*[(\Aca^{\perp_\omega}/\Aca) \otimes \Oca_\alpha(1)]).
\end{tikzcd} \qedhere
\]
\end{proof}
The relations between $\widehat{\mathrm{\Sigma}}$ and $\widehat{\mathrm{\Delta}}$ are summarised in the following diagram:
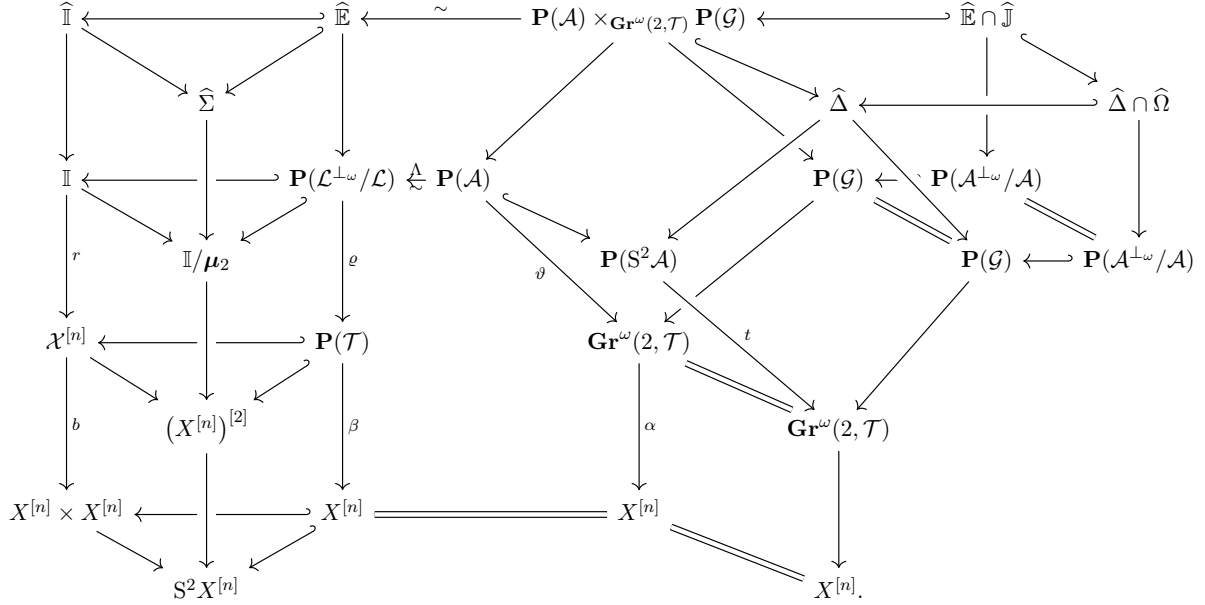
\[
\begin{tikzcd}[column sep = 0.7 em, row sep=1.3em, scale cd=0.8]
\widehat{\mathbb{I}} \ar[rd] \ar[dd] & & \widehat{\mathbb{E}} \ar[ll, hook'] \ar[ld, hook'] \ar[dd] & & \PP(\Aca) \times_{\Grass^\omega(2,\Tca)} \PP(\Gca) \ar[rd, hook] \ar[ll,"\sim"'] \ar[ldd] \ar[rdd] & & \widehat{\mathbb{E}} \cap \widehat{\mathbb{J}} \ar[ll, hook'] \ar[rd, hook] \ar[dd] \\
& \widehat{\mathrm{\Sigma}}  & & & & \widehat{\mathrm{\Delta}} \ar[ldd, crossing over] & & \widehat{\mathrm{\Delta}} \cap \widehat{\mathrm{\Omega}} \ar[ll, hook', crossing over] \ar[dd]\\
\mathbb{I} \ar[rd] \ar[dd, "r"] & & \PP(\Lca^{\perp_\omega}/\Lca) \ar[ll, hook'] \ar[ld, hook'] \ar[dd, "\varrho"] & \PP(\Aca) \ar[l,"\sim","\mathrm{\Lambda}"'] \ar[ddr,"\vartheta"'] \ar[rd, hook]  & & \PP(\Gca) \ar[rd, equal] \ar[ddl] & \PP(\Aca^{\perp_\omega}/\Aca) \ar[l, hook'] \ar[rd, equal] \\
& \mathbb{I}/\bm{\mu}_2 \ar[from=uu, crossing over] & & & \PP(\mathrm{S}^2\Aca) \ar[rdd,"t", crossing over] & &  \PP(\Gca) \ar[ldd] \ar[from=uul, crossing over] & \PP(\Aca^{\perp_\omega}/\Aca) \ar[l, hook'] \\
\Xca^{[n]} \ar[rd] \ar[dd, "b"] & &  \PP(\Tca) \ar[ll, hook'] \ar[ld, hook'] \ar[dd,"\beta"] & &
 \Grass^\omega(2,\Tca) \ar[dd,"\alpha"] \ar[rd,equal] &  \\
& \left(X^{[n]}\right)^{[2]} \ar[from=uu, crossing over] & & & &  \Grass^\omega(2,\Tca) \ar[dd] \\
X^{[n]} \times X^{[n]} \ar[rd] &  & X^{[n]} \ar[ll, hook'] \ar[ld, hook'] & & X^{[n]} \ar[rd, equal] \ar[ll, equal] & & & \\
& \mathrm{S}^2X^{[n]} \ar[from=uu, crossing over] & & & & X^{[n]}. 
\end{tikzcd}
\]
\subsection{Further descriptions of $\widehat{\mathbb{J}}$} 
To conclude this section, we compare the different descriptions of $\widehat{\mathrm{\Sigma}} \cap \widehat{\mathrm{\Omega}}$ given in \cref{prop:OmegaT;OmegaT-SigmaT;DeltaT-SigmaT;OmegaT-SigmaT-DeltaT}(6), \cref{prop:dual conics}(3) and \cref{thm:Sigma hat}(3). 

\smallskip

We introduce some notation. Let $\gamma \colon \Grass^\omega(3, \Tca) \to X^{[n]}$ be the relative symplectic Grassmannian, $\Bca \subset \gamma^*\Tca$ its tautological rank-three bundle and $\Oca_\gamma(1) \coloneqq \wedge^3 \Bca^\vee$. Consider the projective bundle $\chi \colon \PP (\Bca) \to \Grass^\omega(3, \Tca)$, with tautological quotient line bundle $\Oca_\chi(1)$ and relative tangent bundle  $\Tca_\chi$.
Let $\varkappa \colon \PP (\Tca_\chi) \to \PP (\Bca)$ be the associated projective bundle with tautological quotient $\Oca_\varkappa(1)$, and denote by $\kappa \colon \PP (\Tca_\chi) \times_{\PP (\Bca)} \PP (\Tca_\chi) \to \PP (\Bca)$ the relative Segre variety; in particular, the restriction of $\kappa$ to the diagonal is $\varkappa$. For any integers $i$ and $j$, set $\Oca_\kappa(i,j) \coloneqq \Oca_\varkappa(i) \boxtimes \Oca_\varkappa(j)$.

\smallskip

To achieve our goal, we consider the morphism $E^\Bca \to E^\Tca$ and the isomorphism $E^\Bca \xrightarrow{\sim} \mathrm{Bl}_{E^\Tca_2}E^\Tca$ over $\PP(W) \times \PP(\Tca)$ that were obtained in \cref{sec:complete conics}, and compute their GIT quotient by $\mathrm{SO}(W)$. We argue exactly as we have done at the beginning of \cref{sec:Delta hat}: for $w$ in \eqref{distinguished point}, we take the base change along $[w] \times \PP (\Tca) \hookrightarrow \PP(W) \times \PP(\Tca)$, then the GIT quotient by $\mathrm{St}([w]) = \mathrm{PO}(2)$.

\smallskip

First, we get an $X^{[n]}$-morphism
\[
\psi \colon \PP(\Bca) \to \PP(\Tca)
\]
such that $\psi^*\Lca^\vee = \Oca_\chi(1)$, and also a $\bm{\mu}_2$-equivariant composite morphism
\begin{equation*}
\PP(\Tca_\chi) \times_{\PP(\Bca)}\PP(\Tca_\chi) \to \PP((\Lca^{\perp_\omega} / \Lca) \otimes \Lca^\vee) \times_E \PP((\Lca^{\perp_\omega} / \Lca) \otimes \Lca^\vee) \xrightarrow{\sim} \PP(\Lca^{\perp_\omega} / \Lca) \times_E \PP(\Lca^{\perp_\omega} / \Lca),
\end{equation*}
which factors through
\[
\mathrm{\Psi} \colon \PP(\Tca_\chi) \times_{\PP(\Bca)}\PP(\Tca_\chi) \to \mathbb{J}
\] 
(note that $\mathrm{\Psi}^*$ sends $\Oca_\rho(i,j)\coloneqq [\Oca_\varrho(i) \boxtimes \Oca_\varrho(j)]|_{\mathbb{J}}$ to $\Oca_\kappa(i,j) \otimes \kappa^*\Oca_\chi(i+j)$ for any integers $i,j$). 

Second, we obtain the following

\begin{proposition}\label{proposition: Omega hat-Sigma hat} The morphism $\mathrm{\Psi}$ lifts to a $\bm{\mu}_2$-equivariant isomorphism
\begin{equation}\label{morphism Phi}
\widehat{\mathrm{\Psi}} \colon \PP(\Tca_\chi) \times_{\PP(\Bca)}\PP(\Tca_\chi) \xrightarrow{\sim} \widehat{\mathbb{J}},
\end{equation}
which identifies the diagonal of $\PP(\Tca_\chi) \times_{\PP(\Bca)}\PP(\Tca_\chi)$ with $\widehat{\mathbb{E}} \cap \widehat{\mathbb{J}}$. Taking the quotient by $\bm{\mu}_2$, we get an isomorphism
\[
\PP(\mathrm{S}^2 \Tca_\chi) \xrightarrow{\sim} \widehat{\mathbb{J}}/\bm{\mu}_2,
\]
which compares the descriptions of $\widehat{\mathrm{\Sigma}} \cap \widehat{\mathrm{\Omega}}$ given in \cref{prop:OmegaT;OmegaT-SigmaT;DeltaT-SigmaT;OmegaT-SigmaT-DeltaT} and \cref{thm:Sigma hat}.
\end{proposition}

The description of $\widehat{\mathrm{\Sigma}} \cap \widehat{\mathrm{\Omega}}$ given in \cref{prop:dual conics}(3) can be lifted similarly. Before stating the result, we need further notation. Consider the projective bundle $\varphi \colon \PP(\Bca^\vee) \to \Grass^\omega(3,\Tca)$ with tautological quotient line bundle $\Oca_\varphi(1)$, and the relative Segre variety 
\[
\mathbb{S} \coloneqq \PP(\Bca^\vee) \times_{\Grass^\omega(3,\Tca)}\PP(\Bca^\vee) \xrightarrow{\phi} \Grass^\omega(3,\Tca);
\]
for any integers $i,j$, set $\Oca_\phi(i,j) \coloneqq \Oca_\varphi(i) \boxtimes \Oca_\varphi(j)$. Denote by $\widehat{\mathbb{S}}$ the blow-up of $\mathbb{S}$ along the diagonal, and by $\widehat{\mathbb{D}}$ its exceptional divisor.
\begin{proposition}\label{prop:Omega hat-Omega-hat} There is a $\bm{\mu}_2$-equivariant isomorphism
\begin{equation}\label{morphism Xi}
\mathrm{\Xi} \colon \PP (\Tca_\chi) \times_{\PP (\Bca)}\PP (\Tca_\chi) \xrightarrow{\sim} \widehat{\mathbb{S}},
\end{equation}
which identifies the diagonal of $\PP(\Tca_\chi) \times_{\PP(\Bca)}\PP(\Tca_\chi)$ with $\widehat{\mathbb{D}}$.
\end{proposition}
\begin{proof}
Combining the surjections $\chi^*\Bca \otimes \Oca_\chi(1) \to \Tca_\chi$ on $\PP(\Bca)$ and $\varkappa^*\Tca_\chi \to \varkappa^*\Tca_\chi/\Oca_\varkappa(-1)$ on $\PP(\Tca_\chi)$, we get a morphism $\xi \colon \PP (\Tca_\chi) \to \PP(\Bca^\vee)$ over $\Grass^\omega(3,\Tca)$ such that 
\begin{equation}\label{morphism xi on line bundles}
\xi^*\Oca_\varphi(1) = [\varkappa^*\Tca_\chi/\Oca_\varkappa(-1)] \otimes \varkappa^*\Oca_\chi(-1).
\end{equation}
Consider now the product morphism 
\[
\xi^{\times 2} \colon \PP (\Tca_\chi) \times_{\PP (\Bca)}\PP (\Tca_\chi) \to \PP (\Bca ^\vee) \times_{\Grass^\omega(3,\Tca)} \PP (\Bca^\vee)=\mathbb{S}.
\]
The schematic preimage of the diagonal of $\PP (\Bca ^\vee) \times_{\Grass^\omega(3,\Tca)} \PP (\Bca^\vee)$ is a Cartier divisor, namely the diagonal of $\PP (\Tca_\chi) \times_{\PP (\Bca)}\PP (\Tca_\chi)$. Therefore, $\xi^{\times 2}$ lifts to the morphism $\mathrm{\Xi}$ we are looking for: it is $\bm{\mu}_2$-equivariant by construction, and a bijection between smooth varieties, hence an isomorphism by Zariski's main theorem.\qedhere
\end{proof}
Let us study the isomorphisms $\widehat{\mathrm{\Psi}}^*$ and $\mathrm{\Xi}^*$ induced by \eqref{morphism Phi} and \eqref{morphism Xi} between the Picard groups of $\widehat{\mathbb{J}}$, $\PP(\Tca_\chi) \times_{\PP(\Bca)}\PP(\Tca_\chi)$ and $\widehat{\mathbb{S}}$. Modulo line bundles coming from $X^{[n]}$, these groups are free of rank 4, with bases
\begin{align*}
\mathscr{B}_1&\coloneqq (\Lca^\vee, \Oca_\rho(1,0), \Oca_\rho(0,1), \Oca(- \widehat{\mathbb{E}} \cap \widehat{\mathbb{J}})), \\
\mathscr{B}_2&\coloneqq(\Oca_\gamma(1), \Oca_\chi(1), \Oca_\kappa(1,0), \Oca_\kappa(0,1)), \\
\mathscr{B}_3&\coloneqq(\Oca_\gamma(1), \Oca_\phi(1,0), \Oca_\phi(0,1), \Oca(-\widehat{\mathbb{D}})),
\end{align*}
respectively (we have suppressed pullbacks to keep a lighter notation).

\begin{lemma}\label{lemma:Sigma-Omega1 Picard} The matrix of $\widehat {\mathrm{\Psi}}^*$ in the bases $\mathscr{B}_1$ and $\mathscr{B}_2$ is
\[
\mathrm{Mat}_{\mathscr{B}_1,\mathscr{B}_2}(\widehat{\mathrm{\Psi}}^*)=
\begin{pmatrix}
0 & 0 & 0 & 1 \\
1 & 1 & 1 & -3 \\
0 & 1 & 0 & -1 \\
0 & 0 & 1 & -1
\end{pmatrix}.
\]
\end{lemma}
\begin{proof}
The first three columns are a by-product of the previous geometric discussion. It remains to compute the last one. By \cref{proposition: Omega hat-Sigma hat}, $\widehat{\mathrm{\Psi}}^*\Oca(-\widehat{\mathbb{E}} \cap \widehat{\mathbb{J}})$ is the ideal sheaf $\Jca$ of the diagonal of $\PP (\Tca_\chi) \times_{\PP (\Bca)}\PP (\Tca_\chi)$. 
To determine the $\mathbb{Z}$-coordinates of $\Jca$ in the basis $\mathscr{B}_2$, we just note that the coefficients of $\Oca_\kappa(1,0)$ and $\Oca_\kappa(0,1)$ coincide because $\Jca$ is $\bm{\mu}_2$-invariant, and that
\begin{equation*}
\Jca|_{\PP (\Tca_\chi)} = \Kahler^1_\varkappa
= (\chi \circ \varkappa)^*\Oca_\gamma(1) \otimes \varkappa^*\Oca_\chi(-3) \otimes \Oca_\varkappa(-2)
\end{equation*}
by the relative Euler sequences for $\varkappa$ and $\chi$.
\end{proof}

\begin{lemma}\label{lamma:Omega1-Omega2 Picard}
The matrix of $\mathrm{\Xi}^*$ and its inverse with respect to the bases $\mathscr{B}_3$ and $\mathscr{B}_2$ are 
\[
\mathrm{Mat}_{\mathscr{B}_3,\mathscr{B}_2}(\mathrm{\Xi}^*)=
\begin{pmatrix}
1 & -1 & -1 & 1 \\
0 & 2 & 2 & -3 \\
0 & 1 & 0 & -1 \\
0 & 0 & 1 & -1
\end{pmatrix}
\ \ \text{and} \ \ \mathrm{Mat}_{\mathscr{B}_2,\mathscr{B}_3}((\mathrm{\Xi}^{-1})^*)=
\begin{pmatrix}
1 & 1 & -1 & -1 \\
0 & 1 & -1 & -2 \\
0 & 1 & -2 & -1 \\
0 & 1 & -2 & -2
\end{pmatrix}.
\]
\end{lemma}
\begin{proof}
We compute the left-hand side matrix. The first column is obvious.
The equality
\[
\xi^*\Oca_\varphi(1) = [\varkappa^*\Tca_\chi/\Oca_\varkappa(-1)] \otimes \varkappa^* \Oca_\chi(-1) 
= (\chi \circ \varkappa)^* \Oca_\gamma(-1) \otimes \varkappa^* \Oca_\chi(2) \otimes \Oca_\varkappa(1)
\]
coming from \eqref{morphism xi on line bundles} and from the Euler sequence for $\chi$ gives the second and third columns. For the last one, note that, by \cref{prop:Omega hat-Omega-hat}, $\mathrm{\Xi}^*\Oca(-\widehat{\mathbb{D}})$ is the line bundle $\Jca$ that was computed in the proof of \cref{lemma:Sigma-Omega1 Picard}. 
\end{proof}

\begin{remark}\label{rem:normal bundle of S in I bar}
The geometric discussion of this section shows that $\mathbb{J}$ and $\mathbb{S}$ are birational, and actually isomorphic out of subsets of codimension at least 2. In particular, their Picard groups are isomorphic. By Lemmata \ref{lemma:Sigma-Omega1 Picard} and  \ref{lamma:Omega1-Omega2 Picard}, the normal bundle $\Nca_{\mathbb{J}/\mathbb{I}}= \rho^*\Lca$ of $\mathbb{J}$ in $\mathbb{I}$ corresponds to $\phi^*\Oca_\gamma(-1) \otimes \Oca_\phi(-1,-1)$ on $\mathbb{S}$.
\end{remark}
The relations between $\widehat{\mathrm{\Sigma}}$ and $\widehat{\mathrm{\Omega}}$ are summarised in the following diagram:
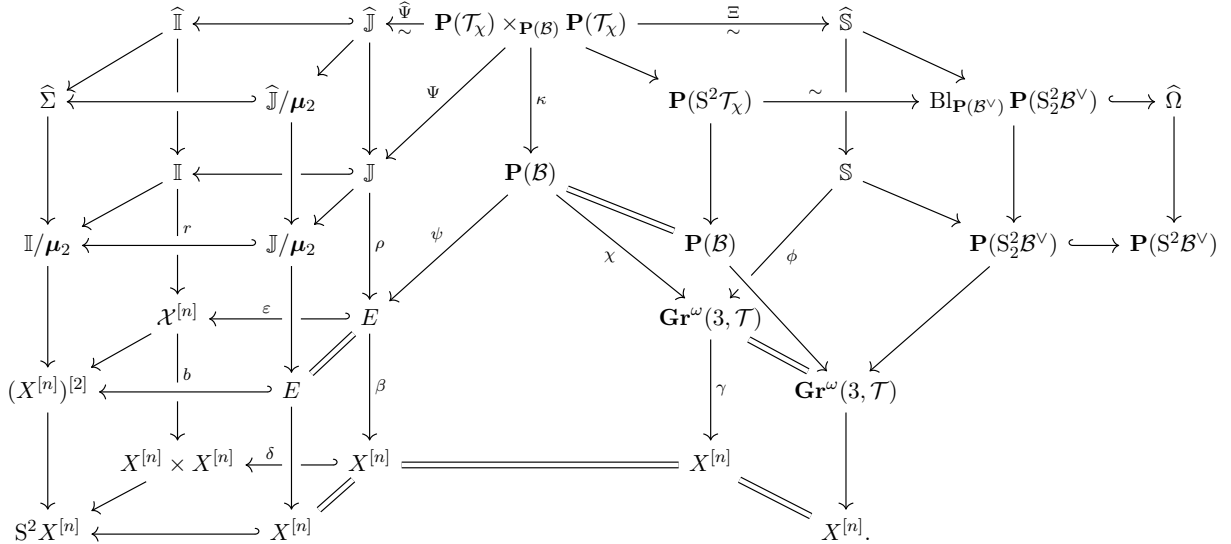
\[
\begin{tikzcd}[row sep = 1.0 em, column sep = 0.4 em, scale cd=0.8]
& \widehat{\mathbb{I}} \ar[ld] \ar[dd] & & \widehat{\mathbb{J}} \ar[ld] \ar[dd] \ar[ll,hook'] & & \PP(\Tca_\chi) \times _{\PP (\Bca)} \PP(\Tca_\chi) \ar[dd,"\kappa"] \ar[rd] \ar[ll,"\widehat{\mathrm{\Psi}}"',"\sim"] \ar[lldd,"\mathrm{\Psi}"'] \ar[rr,"\mathrm{\Xi}","\sim"'] & & \widehat{\mathbb{S}} \ar[dd] \ar[rd] \\
\widehat{\mathrm{\Sigma}} \ar[dd] & & \widehat{\mathbb{J}}/\bm{\mu}_2 \ar[ll, hook', crossing over] & & & & \PP (\mathrm{S}^2\Tca_\chi) \ar[dd] \ar[rr, "\sim"{xshift=-12pt}, crossing over] & & \Bl_{\PP(\Bca^\vee)}\PP(\mathrm{S}^2_2\Bca^\vee) \ar[dd] \ar[r,hook] & \widehat{\mathrm{\Omega}} \ar[dd] \\
&\mathbb{I} \ar[ld] \ar[dd,"r"{yshift=6pt}] & & \mathbb{J} \ar[ll, hook'] \ar[ld] \ar[dd, "\rho"] & & \PP(\Bca) \ar[lldd,"\psi"'] \ar[rdd,"\chi"'] \ar[rd, equal]& & \mathbb{S} \ar[rd] \ar[ldd,"\phi"] \\
\mathbb{I}/\bm{\mu}_2 \ar[dd] & & \mathbb{J}/\bm{\mu}_2 \ar[ll, hook', crossing over] \ar[from=uu, crossing over] & & & & \PP(\Bca) \ar[rdd, crossing over] & &\PP( \mathrm{S}^2_2\Bca^\vee) \ar[ldd] \ar[r,hook] & \PP(\mathrm{S}^2\Bca^\vee) \\
&\Xca^{[n]} \ar[ld] \ar[dd,"b"{yshift=6pt}] & & E \ar[ll, hook', "\varepsilon"' {xshift=-6pt}] \ar[dd, "\beta"] \ar[ld, equal] & & & \Grass^\omega(3,\Tca) \ar[dd,"\gamma"] \ar[rd, equal] & \\
(X^{[n]})^{[2]} \ar[dd] & & E \ar[ll,hook', crossing over] \ar[from=uu, crossing over] & & & & & \Grass^\omega(3,\Tca) \ar[dd] \\
& X^{[n]} \times X^{[n]}\ar[ld] & & X^{[n]} \ar[ll, hook', "\delta"'{xshift=-10pt}] \ar[ld, equal] \ar[rrr,equal] & & & X^{[n]} \ar[rd, equal] & \\
\mathrm{S}^2 X^{[n]} & & X^{[n]} \ar[ll, hook'] \ar[from=uu, crossing over] & & & & &  X^{[n]}.
\end{tikzcd}
\]

\section{Contractions of $\widehat{M}_n$}\label{sec:divisorial contractions of M hat}
In search for a resolution $\widetilde{M}_n \to M_n$ closer to being crepant than $\widehat{\pi} \colon \widehat{M}_n \to M_n$, O'Grady performed two blow-downs. In this section, we briefly look at their effect on $\widehat{\mathrm{\Sigma}}$, $\widehat{\mathrm{\Omega}}$ and $\widehat{\mathrm{\Delta}}$.

\medskip

By \cref{prop:Delta global}, the divisor $\widehat{\mathrm{\Delta}}$ has a structure of $\PP^2$-bundle over $\overline{\mathrm{\Delta}} \coloneqq \PP(\Gca)$. Combining the Fujiki-Nakano criterion with Mori theory, O'Grady showed the existence of a smooth projective variety $\overline{M}_n$ containing $\overline{\mathrm{\Delta}}$, such that $\widehat{M}_n$ is the blow-up of $\overline{M}_n$ along $\overline{\mathrm{\Delta}}$, and $\widehat{\mathrm{\Delta}}$ is its exceptional divisor. Moreover, the resolution $\widehat{\pi} \colon \widehat{M}_n \to M_n$ descends to a regular morphism $\overline{\pi} \colon \overline{M}_n \to M_n$, \cf \cite[Proposition 3.0.3]{O'GdmsK3I}.

Let $\overline{\mathrm{\Omega}}$ and $\overline{\mathrm{\Sigma}}$ be the (schematic) images under $ \widehat{M}_n \to \overline{M}_n$ of $\widehat{\mathrm{\Omega}}$ and $\widehat{\mathrm{\Sigma}}$, respectively. Note that $\overline{\mathrm{\Delta}}$ is a closed subscheme of $\overline{\mathrm{\Sigma}}$ by \cref{cor:Delta-Sigma; Delta-Sigma-Omega}(2), so that the irreducible components of the exceptional locus of $\overline{\pi}$ are $\overline{\mathrm{\Omega}}$ and $\overline{\mathrm{\Sigma}}$.
\begin{proposition}\label{prop:loci in M bar} Keep the notation as in \ref{sec:notation dual conics}. We have isomorphisms
\begin{enumerate}
\item[(1)] $\overline{\mathrm{\Omega}} \simeq \PP(\mathrm{S}^2 \Bca^\vee)$,
\item[(2)] $\overline{\mathrm{\Omega}} \cap \overline{\mathrm{\Sigma}} \simeq  \PP (\mathrm{S}^2_2\Bca^\vee)$,
\item[(3)] $\overline{\mathrm{\Omega}} \cap \overline{\mathrm{\Delta}}\simeq \PP (\mathrm{S}^2_1\Bca^\vee) \simeq \PP (\Bca^\vee)$.
\end{enumerate}
\end{proposition}
\begin{proof}
Propositions \ref{prop:dual conics} and \ref{prop:Delta hat}, together with \cref{rem:Delta hat and Omega hat}, yield bijective morphisms 
\[
\nu_1 \colon \PP(\mathrm{S}^2 \Bca^\vee) \to \overline{\mathrm{\Omega}},\quad \nu_2 \colon \PP (\mathrm{S}^2_2\Bca^\vee) \to \overline{\mathrm{\Omega}} \cap \overline{\mathrm{\Sigma}}  \quad \text{and} \quad  \nu_3 \colon \PP (\mathrm{S}^2_1\Bca^\vee) \to \overline{\mathrm{\Omega}} \cap \overline{\mathrm{\Delta}}.
\]
To prove that $\nu_1$ is an isomorphism (and not just the normalisation of $\overline{\mathrm{\Omega}}$), we use Zariski's main theorem. We check that $\overline{\mathrm{\Omega}}$ is normal using \cite[Chapter II, Theorem 8.22A]{Rag}: it is smooth off the codimension 3 subset $\overline{\mathrm{\Omega}} \cap \overline{\mathrm{\Delta}}$, and, as a divisor on the smooth $\overline{M}_n$, it is Cohen-Macaulay. Hence, $\overline{\mathrm{\Omega}} \simeq \PP(\mathrm{S}^2 \Bca^\vee)$, so that $\overline{\mathrm{\Omega}}$ is even smooth.

To prove that $\nu_2$ and $\nu_3$ are isomorphisms, it suffices now to check that the closed subschemes $\overline{\mathrm{\Omega}} \cap \overline{\mathrm{\Sigma}}$ and $\overline{\mathrm{\Omega}} \cap \overline{\mathrm{\Delta}}$ of $\overline{\mathrm{\Omega}}$ are reduced. 

As a divisor on the smooth $\overline{\mathrm{\Omega}}$, the scheme $\overline{\mathrm{\Omega}} \cap \overline{\mathrm{\Sigma}}$ is Cohen-Macaulay; thus, it is reduced, because it is so away from $\overline{\mathrm{\Delta}}$. 

To show that $\overline{\mathrm{\Omega}} \cap \overline{\mathrm{\Delta}}$ is reduced, we prove that the smooth schemes $\overline{\mathrm{\Omega}}$ and $\overline{\mathrm{\Delta}}$ intersect transversely in the smooth scheme $\overline{M}_n$. We have to show that, for any $x \in \overline{\mathrm{\Omega}} \cap \overline{\mathrm{\Delta}}$,
\[
\mathrm{T}_{\overline{\mathrm{\Omega}},x} + \mathrm{T}_{\overline{\mathrm{\Delta}},x}  = \mathrm{T}_{\overline{M}_n,x}.
\]
Denote by $\widehat{\theta} \colon \widehat{M}_n \to \overline{M}_n$ the contraction map. Pick $v \in \mathrm{T}_{\overline{M}_n,x}$ and let $\bar{v}$ be its image in $N_{\overline{\mathrm{\Delta}}/\overline{M}_n,x}$. There exist unique $y \in \widehat{\theta}^{-1}(x) = \PP(N_{\overline{\mathrm{\Delta}}/\overline{M}_n,x})$ and $\bar{v}' \in N_{\widehat{\mathrm{\Delta}}/\widehat{M}_n,y}$ such that $\widehat{\theta}_*(\bar{v}' ) = \bar{v}$. Now, $\widehat{\theta}^{-1}(x) = \widehat{\mathrm{\Omega}} \cap \widehat{\theta}^{-1}(x)$, and the natural inclusion $N_{\widehat{\mathrm{\Delta}}\cap \widehat{\mathrm{\Omega}}/\widehat{\mathrm{\Omega}},y} \subset N_{\widehat{\mathrm{\Delta}}/\widehat{M}_n,y}$ is an equality. Thus, we can lift $\bar{v}'$ to $v' \in \mathrm{T}_{\widehat{\mathrm{\Omega}},y}$, which satisfies
\[
\widehat{\theta}_*(v') \in \mathrm{T}_{\overline{\mathrm{\Omega}},x} \quad \text{and} \quad \widehat{\theta}_*(v') - v \in \mathrm{T}_{\overline{\mathrm{\Delta}},x}. \qedhere
\]
\end{proof}

By \cref{prop:loci in M bar}(1), $\overline{\mathrm{\Omega}}$ has a structure of $\PP^5$-bundle over $\widetilde{\mathrm{\Omega}} \coloneqq \Grass^\omega(3,\Tca)$. Reasoning as for the first blow-down, O'Grady showed the existence of a smooth projective variety $\widetilde{M}_n$ containing $\widetilde{\mathrm{\Omega}}$ such that $\overline{M}_n$ is the blow-up of $\widetilde{M}_n$ along $\widetilde{\mathrm{\Omega}}$, and $\overline{\mathrm{\Omega}}$ is its exceptional divisor. Moreover, the resolution $\overline{\pi}$ descends to a regular morphism $\widetilde{\pi} \colon \widetilde{M}_n \to M_n$, \cf \cite[Proposition 3.0.4]{O'GdmsK3I}.

Let $\widetilde{\mathrm{\Sigma}}$ and $\widetilde{\mathrm{\Delta}}$ be the (schematic) images under $ \overline{M}_n \to \widetilde{M}_n$ of $\overline{\mathrm{\Sigma}}$ and $\overline{\mathrm{\Delta}}$, respectively. By \cref{prop:loci in M bar}, we have inclusions of closed subschemes $\widetilde{\mathrm{\Omega}} \subset \widetilde{\mathrm{\Delta}} \subset \widetilde{\mathrm{\Sigma}}$ ; in particular, the exceptional locus of $\widetilde{\pi}$ is the irreducible divisor $\widetilde{\mathrm{\Sigma}}$.

\section{Contractions of $\widehat{\mathbb{I}}$}\label{sec: contractions of I hat}
Keep the notation as in \cref{section:Sigma hat}. In \cref{thm:Sigma hat}, we have seen that $\widehat{\mathrm{\Sigma}}$ can be described as a $\bm{\mu}_2$-quotient of $\widehat{\mathbb{I}}$. In this section, we prove that the contractions $\overline{\mathrm{\Sigma}}$ and $\widetilde{\mathrm{\Sigma}}$ of $\widehat{\mathrm{\Sigma}}$ obtained in \cref{sec:divisorial contractions of M hat} are  $\bm{\mu}_2$-quotients of suitable varieties as well. 

\smallskip

Our proof mimics the computations, done in \cite[Section 3]{O'GdmsK3I}, that led O'Grady from $\widehat{M}_n$ to $\overline{M}_n$ and $\widetilde{M}_n$: basically, we study $\widehat{\mathbb{I}}$ using Mori theory. We adopt the following notation. Given a projective variety $Y$, we denote by $\mathrm{N}^1(Y)$ (resp.\ $\mathrm{N}_1(Y)$) the $\mathbb{R}$-vector space of $\mathbb{R}$-Cartier divisors (resp.\ of 1-cycles) on $Y$ modulo numerical equivalence; thus, $\mathrm{N}^1(Y)$ and $\mathrm{N}_1(Y)$ are canonically dual to each other. Furthermore, we denote by $\overline{\Mori}_1(Y)$ the Mori cone of $Y$, i.e.\ the closure in $\mathrm{N}_1(Y)$ of the convex cone $\Mori_1(Y)$ of effective 1-cycles.

\medskip

Pick $[Z] \in X^{[n]}$. As usual, let $\omega$ denote the symplectic form on $\mathrm{E}_Z \coloneqq \mathrm{Ext}^1(I_Z,I_Z)$. Given $[L] \in \PP(\mathrm{E}_Z)$, a point $[A] \in \PP(L^{\perp_\omega}/L)$ corresponds to a two-dimensional $\omega$-isotropic subspace $A$ of $\mathrm{E}_Z$ containing $L$. Similarly, given $[A] \in \Grass^\omega(2,\mathrm{E}_Z)$, a point $[B] \in \PP(A^{\perp_\omega}/A)$ corresponds to a three-dimensional $\omega$-isotropic subspace $B$ of $\mathrm{E}_Z$ containing $A$. 

Let $\widehat{\mathbb{J}}_Z$ (resp.\ $\mathbb{J}_Z$) be the fibre of $\widehat{\mathbb{J}}$ (resp.\ $\mathbb{J}$) over $[Z]$. Concretely, we have by \eqref{definition J}
\begin{equation}\label{definition J_Z}
\mathbb{J}_Z \simeq \left\{ ([L],[A],[A']) \middle| \begin{aligned} [L] \in \PP(\mathrm{E}_Z), \ & ([A], [A']) \in \PP(L^{\perp_\omega}/L) \times \PP(L^{\perp_\omega}/L) \\ &A+A' \subset \mathrm{E}_Z \text{ is $\omega$-isotropic} \end{aligned} \right\}.
\end{equation}
Accordingly, $\widehat{\mathbb{J}}_Z$ is the blow-up of $\mathbb{J}_Z$ along the locus of points of the form $([L],[A],[A])$. The fibre of the exceptional divisor over such a point is isomorphic to $\PP(A^{\perp_\omega}/A)$ by \cref{lem:normal bundles in I and J}.

\smallskip

Let us define classes
\[
\widehat{\varepsilon}_Z^1, \widehat{\varepsilon}_Z^2, \widehat{\sigma}_Z, \widehat{\gamma}_Z \in \Mori_1(\widehat{\mathbb{J}}_Z)
\]
as follows. In the vector space $\mathrm{E}_Z$, fix $L \subset A \subset B$ $\omega$-isotropic linear subspaces of dimension 1, 2 and 3, respectively.

\smallskip

For any $t \in \PP(B/L) \subset \PP(L^{\perp_\omega}/L)$, denote by $A_t$ the corresponding two-dimensional subspace of $B$ containing $L$. In $\mathbb{J}_Z$ consider the curves
\[
\{ [L],[A_t],[A] \} \quad \text{and} \quad  \{ [L],[A],[A_t] \}.
\]
Define $\widehat{\varepsilon}_Z^1$ and $\widehat{\varepsilon}_Z^2$ as the numerical equivalence classes of their respective strict transforms in $\widehat{\mathbb{J}}_Z$.

Next, for any $t \in \PP(A)$, denote by $L_t$ the corresponding one-dimensional subspace of $A$. Define
\[
\widehat{\sigma}_Z \coloneqq \text{class of} \ \{ [L_t],[A], [A], [B]\}.
\]

Finally, fix $\mathrm{\Lambda}$ a line in $\PP(A^{\perp_\omega}/A)$; for any $t \in \mathrm{\Lambda}$, let $B_t$ be the corresponding $\omega$-isotropic three-dimensional subspace of $\mathrm{E}_Z$ containing $A$. Define
\[
\widehat{\gamma}_Z \coloneqq \text{class of} \ \{ [L],[A], [A], [B_t]\}.
\]
Denote by $i_Z \colon \widehat{\mathbb{J}}_Z \hookrightarrow \widehat{\mathbb{I}}$ the inclusion morphism, and define in $\Mori_1(\widehat{\mathbb{I}})$
\begin{gather*}
\widehat{\varepsilon}^j \coloneqq i_{Z, *} \widehat{\varepsilon}_Z^j \quad \text{for} \ j=1,2, \\
\widehat{\sigma} \coloneqq i_{Z, *} \widehat{\sigma}_Z, \\
\widehat{\gamma} \coloneqq i_{Z, *} \widehat{\gamma}_Z.
\end{gather*}
Let us collect the main results of this section before giving their proof.
\begin{lemma}\label{lem:K-negative extremal face} The face $\mathbb{R}^+\widehat{\varepsilon}^1 + \mathbb{R}^+\widehat{\varepsilon}^2 + \mathbb{R}^+\widehat{\sigma} + \mathbb{R}^+\widehat{\gamma} \subset \overline{\Mori}_1(\widehat{\mathbb{I}})$ is a $K_{\widehat{\mathbb{I}}}$-negative extremal face. 
\end{lemma}
Let us denote the contraction of the $K_{\widehat{\mathbb{I}}}$-negative extremal ray $\mathbb{R}^+\widehat{\sigma}$ as
\[
g \colon \widehat{\mathbb{I}} \to \overline{\mathbb{I}}.
\]
\begin{proposition}\label{prop:contraction of I hat}
\begin{enumerate}
\item[(1)] The morphism $g$ identifies with the contraction of $\widehat{\mathbb{I}}$ along the $\PP^1$-bundle $\widehat{\mathbb{E}} \to \PP(\Gca)$ given by \cref{cor:Delta-Sigma; Delta-Sigma-Omega}(1). As a consequence, the morphism $\widehat{\mathbb{I}} \to X^{[n]} \times X^{[n]}$ factors through a morphism $\overline{\mathbb{I}} \to X^{[n]} \times X^{[n]}$.
\item[(2)] The projective variety $\overline{\mathbb{I}}$ is smooth, and $\widehat{\mathbb{I}}$ is the blow-up of $\overline{\mathbb{I}}$ along $\PP(\Gca)$.
\item[(3)] The $\bm{\mu}_2$-action on $\widehat{\mathbb{I}}$ descends to $\overline{\mathbb{I}}$, and $\overline{\mathbb{I}}/\bm{\mu}_2 \simeq \overline{\mathrm{\Sigma}}$.
\item[(4)] Let $\overline{\mathbb{J}}$ be the image of $\widehat{\mathbb{J}}$ under $g$:
\[
\overline{\mathbb{J}} \coloneqq g(\widehat{\mathbb{J}}).
\]
Then $\overline{\mathbb{J}} \simeq \PP(\Bca^\vee) \times _{\Grass^\omega(3,\Tca)} \PP(\Bca^\vee)$, and $\overline{\mathbb{J}}/\bm{\mu}_2 \simeq \overline{\mathrm{\Omega}} \cap \overline{\mathrm{\Sigma}}$.
\end{enumerate}
\end{proposition}
Let $\overline{\varepsilon}^1$ and $\overline{\varepsilon}^2$ be the images in $\mathrm{N}_1(\overline{\mathbb{I}})$ of $\widehat{\varepsilon}^1$ and $\widehat{\varepsilon}^2$, respectively. Then $\mathbb{R}^+\overline{\varepsilon}^1+\mathbb{R}^+\overline{\varepsilon}^2$ is a $K_{\overline{\mathbb{I}}}$-negative extremal face, and we can consider the corresponding contraction 
\[
f \colon \overline{\mathbb{I}} \to \widetilde{\mathbb{I}}.
\] 
\begin{proposition}\label{prop:contraction of I bar} 
\begin{enumerate}
\item[(1)] The morphism $f$ identifies with the contraction of $\overline{\mathbb{I}}$ along the $\PP^2 \times \PP^2$-bundle $\overline{\mathbb{J}} \to \Grass^\omega(3,\Tca)$ given by \cref{prop:contraction of I hat}(4). As a consequence, the morphism $\overline{\mathbb{I}} \to X^{[n]} \times X^{[n]}$ factors through a morphism $\widetilde{\mathbb{I}} \to X^{[n]} \times X^{[n]}$.
\item[(2)] The projective variety $\widetilde{\mathbb{I}}$ is singular and normally flat along $\Grass^\omega(3,\Tca)$, and $\overline{\mathbb{I}}$ is the blow-up of $\widetilde{\mathbb{I}}$ along $\Grass^\omega(3,\Tca)$.
\item[(3)] The $\bm{\mu}_2$-action on $\overline{\mathbb{I}}$ descends to $\widetilde{\mathbb{I}}$, and $\widetilde{\mathbb{I}}/\bm{\mu}_2 \simeq \widetilde{\mathrm{\Sigma}}$. 
\end{enumerate}
\end{proposition}
\noindent These results are summed up in the following diagram:
\[
\begin{tikzcd}[row sep = 1.0em, column sep = 0.3 em, scale cd=0.8]
& & & \widehat{\mathbb{E}} \cap \widehat{\mathbb{J}} \ar[rd, hook] \ar[rrrr,equal] \ar[ddd] & & & & \widehat{\mathbb{E}} \cap \widehat{\mathbb{J}} \ar[rd,hook] \ar[ddd] \\
\widehat{\mathbb{I}} \ar[rd] \ar[ddd] & & \widehat{\mathbb{E}} \ar[ll, hook'] \ar[ld, hook'] \ar[rr, equal, crossing over]  \ar[ddd] & & \widehat{\mathbb{E}} \ar[r,hook] &  \widehat{\mathbb{I}} \ar[rd] & & & \widehat{\mathbb{J}} \ar[ddd] \ar[lll,hook', crossing over] \ar[rd] \\
& \widehat{\mathrm{\Sigma}} & & & & & \widehat{\mathrm{\Sigma}} & & & \widehat{\mathrm{\Sigma}} \cap \widehat{\mathrm{\Omega}} \ar[ddd] \ar[lll,hook', crossing over] \\
& & & \PP(\Aca^{\perp_\omega}/\Aca) \ar[rd, hook] \ar[rrrr,"\sim" {xshift=24pt}] & & & & \PP(\Bca^\vee) \ar[rd, hook] \\
\mathbb{I} \ar[rd] \ar[dd, "r"] & & \PP(\Lca^{\perp_\omega}/\Lca) \simeq \PP(\Aca) \ar[ll,hook'] \ar[ld,hook'] \ar[dd, "\varrho"] \ar[rdd,"\vartheta"] & & \PP(\Gca)  \ar[from=uuu, crossing over] \ar[ldd] \ar[r,hook] & \overline{\mathbb{I}} \ar[dd,"f"] \ar[rd] \ar[from=uuu, crossing over, "g"] & & & \overline{\mathbb{J}} \ar[lll,hook'] \ar[rd] \ar[dd, "\phi"{yshift=12pt}] \\
& \mathbb{I}/\bm{\mu}_2 \ar[from=uuu, crossing over] & & & & & \overline{\mathrm{\Sigma}} \ar[from=uuu, crossing over] & & & \overline{\mathrm{\Sigma}} \cap \overline{\mathrm{\Omega}} \ar[dd] \ar[lll,hook', crossing over] \\
\Xca^{[n]} \ar[rd] \ar[dd, "b"] & &  \PP(\Tca) \ar[ll, hook', "\varepsilon"'{xshift=12pt}] \ar[ld, hook'] \ar[dd,"\beta"] &
 \Grass^\omega(2,\Tca) \ar[dd,"\alpha"] & & \widetilde{\mathbb{I}} \ar[rd] \ar[dd] & & & \Grass^\omega(3,\Tca) \ar[lll,hook'] \ar[rd,"\sim"] \ar[dd,"\gamma"{yshift=12pt}]  \\
& \left(X^{[n]}\right)^{[2]} \ar[from=uu, crossing over] & & & & & \widetilde{\mathrm{\Sigma}} \ar[from=uu, crossing over] & & & \widetilde{\mathrm{\Omega}} \ar[lll,hook', crossing over] \ar[dd] \\
X^{[n]} \times X^{[n]} \ar[rd] &  & X^{[n]} \ar[ll, hook', "\delta"'] \ar[ld, hook'] & X^{[n]} \ar[l, equal] \ar[rr,hook,"\delta" ]& &  X^{[n]} \times X^{[n]} \ar[rd] & & & X^{[n]}\ar[lll, hook',"\delta"'] \ar[rd, equal]\\
& \mathrm{S}^2X^{[n]} \ar[from=uu, crossing over] & & & & &  \mathrm{S}^2X^{[n]} \ar[from=uu, crossing over] & & & X^{[n]}. \ar[lll, hook'] 
\end{tikzcd}
\]

\subsection{Proof of \cref{lem:K-negative extremal face}}
We shall need to compute the Mori cone of various closed subschemes of $\widehat{\mathbb{I}}$. To this aim, here is a useful lemma.
\begin{lemma}\label{lem:Mori cones}
Let $Y$ be a projective variety. Assume we are given $\mathrm{\Gamma}_1, \mathrm{\Gamma}_2, \dots, \mathrm{\Gamma}_m$ integral curves on $Y$ and two morphisms $c' \colon Y \to Y'$, and $c'' \colon Y \to Y''$ with the following properties:
\begin{enumerate}
\item[(1)] the classes $[\mathrm{\Gamma}_1], [\mathrm{\Gamma}_2], \dots,[\mathrm{\Gamma}_m]$ form an $\mathbb{R}$-basis of $\mathrm{N}_1(Y)$;
\item[(2)] $Y'$ and $Y''$ are projective varieties;
\item[(3)] $c'$ contracts $\mathrm{\Gamma}_1$, but none of the other curves $\mathrm{\Gamma}_2, \dots, \mathrm{\Gamma}_m$;
\item[(4)]  $c''$ contracts $\mathrm{\Gamma}_2, \dots, \mathrm{\Gamma}_m$, but not $\mathrm{\Gamma}_1$.
\end{enumerate}
If $[c'(\mathrm{\Gamma}_2)], \dots,[c'(\mathrm{\Gamma}_m)]$ form an $\mathbb{R}$-basis of $\mathrm{N}_1(Y')$ and $\overline{\Mori}_1(Y')= \mathbb{R}^+[c'(\mathrm{\Gamma}_2)]+ \dots + \mathbb{R}^+[c'(\mathrm{\Gamma}_m)]$, then $\overline{\Mori}_1(Y) = \mathbb{R}^+[\mathrm{\Gamma}_1]+ \dots + \mathbb{R}^+[\mathrm{\Gamma}_m]$. 
\end{lemma}
\begin{proof}
Let $\mathrm{\Gamma}$ be an integral curve in $Y$. Write $[\mathrm{\Gamma}] = y_1 [\mathrm{\Gamma}_1] + y_2 [\mathrm{\Gamma}_2] + \dots + y_m [\mathrm{\Gamma}_m]$, for $y_1,y_2, \dots, y_m \in \mathbb{R}$. We show that $y_i \ge 0$ for any $i$. Since $Y''$ is projective, intersecting $c_*''[\mathrm{\Gamma}] = y_1 c_*''[\mathrm{\Gamma}_1]$ with an ample divisor we get $y_1 \ge 0$. On the other hand, considering $c_*'[\mathrm{\Gamma}]$ and the assumption on the Mori cone of $Y'$, we deduce that $y_i \ge 0$ for any $i \ge 2$.
\end{proof}

Let $[Z] \in X^{[n]}$ be a closed point. If $Y$ is a scheme over $X^{[n]}$ and $\Fca$ is a vector bundle on $Y$, we shall denote as $Y_Z$ the fibre of $Y$ over $[Z]$, and as $\Fca_Z$ the restriction of $\Fca$ to $Y_Z$.

\subsubsection{Digression on $\widehat{\mathbb{J}}_Z$}
Recall that $\Lca^\vee$ (resp.\ $\Oca_\varrho(1)$) denotes the tautological quotient line bundle of $\PP(\Tca)$ (resp.\ $\PP(\Lca^{\perp_\omega}/\Lca)$). For any integers $i,j \in \mathbb{Z}$, we get a line bundle $\Oca_\varrho(i) \boxtimes \Oca_\varrho(j)$ on $\PP(\Lca^{\perp_\omega}/\Lca) \times_{\PP(\Tca)}\PP(\Lca^{\perp_\omega}/\Lca)$. We set $\Oca_\rho(i,j) \coloneqq [\Oca_\varrho(i) \boxtimes \Oca_\varrho(j)]|_{\mathbb{J}}$.

Consider the classes
\begin{alignat*}{2}
x_1 &\coloneqq c_1(\Lca_Z^\vee) \quad && \in \mathrm{N}^1(\PP(\Tca)_Z ) \\
x_2 &\coloneqq c_1(\Oca_\rho(1,0)_Z) \quad && \in \mathrm{N}^1(\mathbb{J}_Z)\\
x_3 &\coloneqq c_1(\Oca_\rho(0,1)_Z) \quad && \in \mathrm{N}^1(\mathbb{J}_Z) \\
x_4 &\coloneqq c_1(\Oca(\widehat{\mathbb{E}} \cap \widehat{\mathbb{J}}_Z)) \quad && \in \mathrm{N}^1(\widehat{\mathbb{J}}_Z).
\end{alignat*}
Their pullbacks to $\widehat{\mathbb{J}}_Z$ give classes in $\mathrm{N}^1(\widehat{\mathbb{J}}_Z)$, that we still denote with the same letters. 

\begin{lemma}\label{intersection matrix} The intersection matrix of $\{x_1, x_2, x_3, x_4 \}$ with $\{ \widehat{\varepsilon}_Z^1, \widehat{\varepsilon}_Z^2, \widehat{\sigma}_Z, \widehat{\gamma}_Z \}$ is
\begin{center}
\begin{tabular}{ |c||c|c|c|c| } 
 \hline
  & $x_1$ & $x_2$ & $x_3$ & $x_4$ \\
  \hline \hline
 $\widehat{\varepsilon}_Z^1$ & $0$ & $1$ & $0$ & $1$ \\
 \hline
 $\widehat{\varepsilon}_Z^2$ & $0$ & $0$ & $1$ & $1$ \\
 \hline
 $\widehat{\sigma}_Z$ & $1$ & $-1$ & $-1$ & $-1$ \\ 
 \hline
 $\widehat{\gamma}_Z$ & $0$ & $0$ & $0$ & $-1$ \\
 \hline
\end{tabular}
\end{center}
\end{lemma}
\begin{proof}
We discuss only the non-trivial entries. The intersection of $\widehat{\sigma}_Z$ with $x_4$, thanks to Propositions \ref{proposition: Omega hat-Sigma hat} and \ref{prop:Omega hat-Omega-hat}, can be computed in $\widehat{\mathbb{S}}_Z$. The intersection of $\widehat{\sigma}_Z$ with $x_2$ and $x_3$ reduces first to a computation in $\PP(\Lca^{\perp_\omega}/\Lca)_Z$, then, via \cref{proposition:DeltaS-SigmaS}, in $\PP(\Aca)_Z$.
\end{proof}

\begin{lemma}\label{lemma:basis for N1(JZ)}
The classes $x_1, x_2, x_3, x_4$ form a basis of $\mathrm{N}^1(\widehat{\mathbb{J}}_Z)$.
\end{lemma}
\begin{proof}
The Picard group of $\widehat{\mathbb{J}}_Z$ is generated by $\Oca(\widehat{\mathbb{E}} \cap \widehat{\mathbb{J}}_Z)$ and by the pullback of the line bundles $\Lca_Z^\vee, \Oca_\rho(1,0)_Z$ and $\Oca_\rho(0,1)_Z$. In particular, $x_1, x_2, x_3, x_4$ generate $\mathrm{N}^1(\widehat{\mathbb{J}}_Z)$. They are also linearly independent, because the intersection matrix of \cref{intersection matrix}  is non-degenerate. 
\end{proof}
\begin{proposition}\label{prop:Mori cone JZ}
The classes $\widehat{\varepsilon}_Z^1, \widehat{\varepsilon}_Z^2, \widehat{\sigma}_Z, \widehat{\gamma}_Z$ form an $\mathbb{R}$-basis of $\mathrm{N}_1(\widehat{\mathbb{J}}_Z)$. Moreover
\[
\overline{\Mori}_1(\widehat{\mathbb{J}}_Z) = \mathbb{R}^+\widehat{\varepsilon}_Z^1 + \mathbb{R}^+\widehat{\varepsilon}_Z^2 + \mathbb{R}^+\widehat{\sigma}_Z + \mathbb{R}^+\widehat{\gamma}_Z. 
\]
\end{proposition}
\begin{proof}
The first statement results from \cref{lemma:basis for N1(JZ)} by duality. For the second one, consider the following pairs of morphisms:
\[
\begin{tikzcd}[column sep=1.0em]
\mathbb{J}_Z  & \widehat{\mathbb{J}}_Z \ar[l] \ar[r] & \PP(\Bca)_Z \ar[r] & \Grass^\omega(3, \Tca)_Z
\end{tikzcd}
\]
(given by blow-up and by \cref{proposition: Omega hat-Sigma hat}),
\[
\begin{tikzcd}[column sep=1.0em]
\PP(\Lca^{\perp_\omega}/\Lca)_Z & \mathbb{J}_Z \ar[l] \ar[r] & \PP(\Lca^{\perp_\omega}/\Lca)_Z \simeq \PP(\Aca)_Z \ar[r] & \Grass^\omega(2, \Tca)_Z
\end{tikzcd}
\]
(given by the natural projections and by \cref{proposition:DeltaS-SigmaS}),
\[
\begin{tikzcd}[column sep=1.0em]
\PP(\Tca)_Z & \PP(\Lca^{\perp_\omega}/\Lca)_Z \simeq \PP(\Aca)_Z \ar[l] \ar[r] & \Grass^\omega(2, \Tca)_Z.
\end{tikzcd}
\]
We apply \cref{lem:Mori cones} to these pairs of morphisms: starting from $\overline{\Mori}_1(\PP(\Tca)_Z)$, we deduce one after the other $\overline{\Mori}_1(\PP(\Lca^{\perp_\omega}/\Lca)_Z)$, $\overline{\Mori}_1(\mathbb{J}_Z)$ and $\overline{\Mori}_1(\widehat{\mathbb{J}}_Z)$.
\end{proof}
\begin{lemma}\label{lemma:NJZ-NI}
The linear map $i_{Z,*} \colon \mathrm{N}_1(\widehat{\mathbb{J}}_Z) \to \mathrm{N}_1(\widehat{\mathbb{I}})$ induced by $i_Z \colon \widehat{\mathbb{J}}_Z \hookrightarrow \widehat{\mathbb{I}}$ is injective.
\end{lemma}
\begin{proof}
By duality, it is enough to show that $i_Z^* \colon \mathrm{N}^1(\widehat{\mathbb{I}}) \to \mathrm{N}^1(\widehat{\mathbb{J}}_Z)$ is surjective. 
Consider on $\Xca^{[n]}$ the line bundle $\Oca(-E)$ given by the divisor $E$. Consider on $\mathbb{I}$ the line bundle $\Oca_r(1,0)$ (resp.\ $\Oca_r(0,1)$) given by the pullback of the tautological quotient line bundle of $\PP(V)$ (resp.\ $\PP(V^\vee)$) along the natural projection, see \eqref{definition I}. We have
\begin{equation}\label{restriction to J_Z of line bundles on I hat}
x_1= i_Z^*[c_1(\Oca(-E))], \ x_2 = i_Z^*[c_1(\Oca_r(1,0))], \ x_3 = i_Z^*[c_1(\Oca_r(0,1))], \ x_4 = i_Z^*[c_1(\Oca(\widehat{\mathbb{E}}))]. \qedhere
\end{equation}
\end{proof}
\begin{proposition}\label{prop:K negative}
The restriction to $\widehat{\mathbb{J}}_Z$ of the canonical class of $\widehat{\mathbb{I}}$ is 
\[
K_{\widehat{\mathbb{I}}}|_{\widehat{\mathbb{J}}_Z} = (1-2n)x_1 + (3-2n)x_2 + (3-2n)x_3 + (2n-4)x_4.
\]
In particular, $K_{\widehat{\mathbb{I}}} \cdot \widehat{\varepsilon}^1 = -1, \quad K_{\widehat{\mathbb{I}}} \cdot \widehat{\varepsilon}^2 = -1, \quad K_{\widehat{\mathbb{I}}} \cdot \widehat{\sigma} = -1, \quad K_{\widehat{\mathbb{I}}} \cdot \widehat{\gamma} = 4-2n$.
\end{proposition}
\begin{proof}
Let us compute first the restriction to $\mathbb{J}_Z$ of the canonical class of $\mathbb{I}$. Applying adjunction formula to the inclusions $\mathbb{J}_Z \hookrightarrow \mathbb{J}$ and $\mathbb{J} \hookrightarrow \mathbb{I}$, we have
\[
K_{\mathbb{I}}|_{\mathbb{J}_Z} = K_{\mathbb{J}_Z}-c_1(\Nca_{\mathbb{J}_Z/\mathbb{J}})-[c_1(\Nca_{\mathbb{J}/\mathbb{I}})|_{\mathbb{J}_Z}].
\]
Since $\mathbb{J} \to X^{[n]}$ is flat, the second term vanishes. As for the third term, we have seen in the proof of \cref{lem:normal bundles in I and J} that $c_1(\Nca_{\mathbb{J}/\mathbb{I}})|_{\mathbb{J}_Z} =c_1((\rho^*\Lca)|_{\mathbb{J}_Z}) = -x_1$. Using the geometric description of $\mathbb{J}_Z$, the first term is easily computed to be $-2n x_1 + (3-2n)(x_2+x_3)$. Thus, we obtain
\[
K_{\mathbb{I}}|_{\mathbb{J}_Z} = (1-2n)x_1 + (3-2n)x_2
+(3-2n)x_3.
\]
Now, recall that $\widehat{\mathbb{I}}$ is the blow-up of the smooth $\mathbb{I}$ along $\PP(\Lca^{\perp_\omega}/\Lca)$, which is smooth of codimension $2n-3$; in particular, there is a well known expression of $K_{\widehat{\mathbb{I}}}$ in terms of $K_{\mathbb{I}}$ and $c_1(\Oca(\widehat{\mathbb{E}}))$, whose restriction to $\widehat{\mathbb{J}}_Z$ yields the desired result.

The intersection numbers are obtained from \cref{intersection matrix}.
\end{proof}

\subsubsection{Digression on $\widehat{\mathbb{E}}_Z$}\label{sec:digression on E_Z hat} By \cref{cor:Delta-Sigma; Delta-Sigma-Omega}(1), $\widehat{\mathbb{E}}_Z$ is isomorphic to $[\PP(\Aca)  \times_{\Grass^\omega(2,\Tca)} \PP(\Gca)]_Z$.
Let $j_Z \colon \widehat{\mathbb{E}}_Z \hookrightarrow \widehat{\mathbb{I}}$ be the inclusion morphism. 

\begin{lemma}\label{1-cycles in E_Z}
The map $j_{Z,*} \colon \mathrm{N}_1(\widehat{\mathbb{E}}_Z) \to \mathrm{N}_1(\widehat{\mathbb{I}})$ is injective, with image $\mathbb{R}(\widehat{\varepsilon}^1+\widehat{\varepsilon}^2) \oplus \mathbb{R}\widehat{\sigma} \oplus \mathbb{R}\widehat{\gamma}$.
\end{lemma}
\begin{proof}
Since $\widehat{\mathbb{E}}_Z \cap \widehat{\mathbb{J}} \simeq [\PP(\Aca)  \times_{\Grass^\omega(2,\Tca)} \PP(\Aca^{\perp_\omega}/\Aca)]_Z$ by \cref{cor:Delta-Sigma; Delta-Sigma-Omega}(3), we have
\[
\mathrm{N}_1(\widehat{\mathbb{E}}_Z) \simeq \mathrm{N}_1(\widehat{\mathbb{E}}_Z \cap \widehat{\mathbb{J}}).
\]
 Let us study the map 
\[
\mathrm{N}_1(\widehat{\mathbb{E}}_Z \cap \widehat{\mathbb{J}}) \to \mathrm{N}_1(\widehat{\mathbb{J}}_Z).
\]
The dual map $\mathrm{N}^1(\widehat{\mathbb{J}}_Z) \to \mathrm{N}^1(\widehat{\mathbb{E}}_Z \cap \widehat{\mathbb{J}})$ is surjective, because $\mathrm{N}^1(\widehat{\mathbb{E}}_Z \cap \widehat{\mathbb{J}})$ is generated by (the pullback of) $c_1(\Lca_Z)$, $c_1(\Oca_\varrho(1)_Z)=c_1(\Oca_\rho(1,0)_Z)|_{\PP(\Lca^\perp/\Lca)_Z}$ and $c_1(\Oca(\widehat{\mathbb{E}} \cap \widehat{\mathbb{J}}_Z))|_{\widehat{\mathbb{E}}_Z \cap \widehat{\mathbb{J}}}$.
Moreover, its 1-dimensional kernel is generated by $x_2-x_3$.

We deduce that $\mathrm{N}_1(\widehat{\mathbb{E}}_Z \cap \widehat{\mathbb{J}}) \to \mathrm{N}_1(\widehat{\mathbb{J}}_Z)$ is injective, and that
\[
\im[\mathrm{N}_1(\widehat{\mathbb{E}}_Z \cap \widehat{\mathbb{J}}) \to \mathrm{N}_1(\widehat{\mathbb{J}}_Z)] \simeq \ker [\mathrm{N}^1(\widehat{\mathbb{J}}_Z) \to \mathrm{N}^1(\widehat{\mathbb{E}}_Z \cap \widehat{\mathbb{J}})]^\perp = (x_2-x_3)^\perp = \mathbb{R}(\widehat{\varepsilon}^1_Z+\widehat{\varepsilon}^2_Z) \oplus \mathbb{R}\widehat{\sigma}_Z \oplus \mathbb{R}\widehat{\gamma}_Z
\]
by \cref{intersection matrix}. The statement now follows by \cref{lemma:NJZ-NI}.
\end{proof}

By \cref{1-cycles in E_Z}, there are uniquely determined classes
\[
\widehat{\varepsilon}_Z', \widehat{\sigma}'_Z, \widehat{\gamma}'_Z \in \mathrm{N}_1(\widehat{\mathbb{E}}_Z)
\]
such that 
\[
j_{Z,*}\widehat{\varepsilon}_Z' = \widehat{\varepsilon}^1 + \widehat{\varepsilon}^2, \quad j_{Z,*}\widehat{\sigma}'_Z = \widehat{\sigma}, \quad j_{Z,*}\widehat{\gamma}'_Z = \widehat{\gamma}.
\]
Let us  give explicit representatives of these classes. Let us fix $L \subset A \subset B$ a flag of $\omega$-isotropic linear subspaces of $\mathrm{Ext}^1(I_Z,I_Z)$ of dimension 1, 2 and 3, respectively. Then
\begin{align*}
\widehat{\varepsilon}_Z' &\coloneqq \text{class of } \{ ([L],[A_t],[B]) \mid [A_t] \in \PP(B/L) \}, \\
\widehat{\sigma}_Z' &\coloneqq \text{class of } \{ ([L_t],[A],[B]) \mid [L_t] \in \PP(A) \}, \\
\widehat{\gamma}_Z' &\coloneqq \text{class of } \{ ([L],[A],[B_t]) \mid [B_t] \in \PP(A^{\perp_\omega}/A) \ \text{varies in a line} \}.
\end{align*}

\begin{proposition}\label{prop:Mori cone EZ}
The classes $\widehat{\varepsilon}'_Z,\widehat{\sigma}'_Z,\widehat{\gamma}'_Z$ form an $\mathbb{R}$-basis of $\mathrm{N}_1(\widehat{\mathbb{E}}_Z)$. Moreover
\[
\overline{\Mori}(\widehat{\mathbb{E}}_Z) = \mathbb{R}^+\widehat{\varepsilon}_Z' + \mathbb{R}^+\widehat{\sigma}'_Z + \mathbb{R}^+\widehat{\gamma}'_Z.
\]  
\end{proposition}
\begin{proof} Only the second statement needs to be proved. We apply \cref{lem:Mori cones} to the pair of morphisms
\[
\begin{tikzcd}[column sep=1.0em, row sep=0.1em]
 \PP(\Lca^{\perp_\omega}/\Lca)_Z \simeq \PP(\Aca)_Z  &\PP(\Aca)_Z \times_{\Grass^\omega(2,\Tca)_Z} \PP(\Gca)_Z \ar[l] \ar[r] &  \PP(\Gca)_Z=\overline{\mathrm{\Delta}}_Z \ar[r] &\widetilde{\mathrm{\Delta}}_Z
\end{tikzcd}
\]
given by the natural projections,  \cref{proposition:DeltaS-SigmaS} and the contraction $\overline{\mathrm{\Delta}} \to \widetilde{\mathrm{\Delta}}$ obtained in \cref{sec:divisorial contractions of M hat}: we can deduce $\overline{\Mori}_1(\widehat{\mathbb{E}}_Z)$ from $\overline{\Mori}_1(\PP(\Lca^{\perp_\omega}/\Lca)_Z)$.
\end{proof}

\subsubsection{Digression on $\widehat{\mathbb{I}}_{Z_1, Z_2}$} Let $[Z_1],[Z_2]$ be two \emph{distinct} points of $X^{[n]}$. Let $\widehat{\mathbb{I}}_{Z_1, Z_2}$ be the fibre of $\widehat{\mathbb{I}}$ over $([Z_1],[Z_2])\in X^{[n]} \times X^{[n]}$; it is isomorphic to the incidence divisor in $ \PP^{2n-3} \times (\PP^{2n-3})^\vee$. Denote as $i_{Z_1,Z_2} \colon \widehat{\mathbb{I}}_{Z_1, Z_2} \hookrightarrow \widehat{\mathbb{I}}$ the inclusion morphism. 
\begin{proposition}\label{prop:Mori cone IZ1Z2}
The linear map $i_{Z_1,Z_2,*} \colon \mathrm{N}_1(\widehat{\mathbb{I}}_{Z_1, Z_2}) \to \mathrm{N}_1(\widehat{\mathbb{I}})$ is injective, and we have
\[
i_{Z_1,Z_2,*} \overline{\Mori}_1(\widehat{\mathbb{I}}_{Z_1, Z_2}) = \mathbb{R}^+ (\widehat{\varepsilon}^1 + \widehat{\gamma}) + \mathbb{R}^+ (\widehat{\varepsilon}^2 + \widehat{\gamma}).
\]
\end{proposition}
\begin{proof}
Using the geometric construction of $\widehat{\mathbb{I}}$, keeping the notation as in the proof of \cref{lemma:NJZ-NI}, and suppressing pullbacks, we have
\[
\mathrm{Pic}(\widehat{\mathbb{I}}) \simeq \mathrm{Pic}(X^{[n]}\times X^{[n]}) \oplus \mathbb{Z}\Oca(-E) \oplus \mathbb{Z}\Oca_r(1,0) \oplus \mathbb{Z}\Oca_r(0,1) \oplus \mathbb{Z}\Oca(\widehat{\mathbb{E}}).
\]
The injectivity of $i_{Z_1,Z_2,*}$ is equivalent to the surjectivity of $i_{Z_1,Z_2}^* \colon \mathrm{N}^1(\widehat{\mathbb{I}}) \to \mathrm{N}^1(\widehat{\mathbb{I}}_{Z_1, Z_2})$, which clearly holds.

Let us prove the second statement. First, we claim that
\begin{equation}\label{curves in I hat and J hat}
i_{Z_1,Z_2,*} \mathrm{N}_1(\widehat{\mathbb{I}}_{Z_1, Z_2}) \subset i_{Z,*} \mathrm{N}_1(\widehat{\mathbb{J}}_Z) = \ker[i_Z^* \colon \mathrm{N}^1(\widehat{\mathbb{I}}) \to \mathrm{N}^1(\widehat{\mathbb{J}}_Z)]^\perp
\end{equation}
for some (actually, for any) $[Z] \in X^{[n]}$. This holds true, because $\ker(i_Z^*)$ is generated by numerical equivalence classes of pullbacks to $\widehat{\mathbb{I}}$ of line bundles on $X^{[n]} \times X^{[n]}$, as one sees from \eqref{restriction to J_Z of line bundles on I hat}; the intersection of such classes with integral curves contained in $\widehat{\mathbb{I}}_{Z_1, Z_2}$ is $0$.

Now, the Mori cone $\overline{\Mori}_1(\widehat{\mathbb{I}}_{Z_1, Z_2})$ is generated by
\begin{gather*}
\tau_1 \coloneqq \text{class of } \{ ([e_t], [f]) \in \PP\mathrm{Ext}^1(I_{Z_1},I_{Z_2}) \times \PP\mathrm{Ext}^1(I_{Z_2},I_{Z_1}) \mid e_t \cup f = 0, [e_t] \ \text{varies in a line} \}, \\
\tau_2 \coloneqq \text{class of } \{ ([e], [f_t]) \in \PP\mathrm{Ext}^1(I_{Z_1},I_{Z_2}) \times \PP\mathrm{Ext}^1(I_{Z_2},I_{Z_1}) \mid e \cup f_t = 0, [f_t] \ \text{varies in a line} \}.
\end{gather*}
Let us compute $i_{Z_1,Z_2,*} (\tau_1)$; by \eqref{curves in I hat and J hat}, it can be written as 
\[
i_{Z_1,Z_2,*} (\tau_1) = y_1 \widehat{\varepsilon}^1 + y_2 \widehat{\varepsilon}^2 + y_3 \widehat{\sigma} + y_4 \widehat{\gamma}, \quad y_i \in \mathbb{R}.
\]
On the one hand, we have
\begin{gather*}
i_{Z_1,Z_2,*} (\tau_1) \cdot c_1(\Oca(-E)) = i_{Z_1,Z_2,*} (\tau_1) \cdot c_1(\Oca(\widehat{\mathbb{E}})) = i_{Z_1,Z_2,*} (\tau_1) \cdot c_1(\Oca_r(0,1))= 0, \\
i_{Z_1,Z_2,*} (\tau_1) \cdot c_1(\Oca_r(1,0))=1.
\end{gather*}
On the other hand, using \cref{intersection matrix}, we deduce that $i_{Z_1,Z_2,*} (\tau_1) = \widehat{\varepsilon}^1 + \widehat{\gamma}$. Analogously, one proves that $i_{Z_1,Z_2,*} (\tau_2) = \widehat{\varepsilon}^2 + \widehat{\gamma}$. 
\end{proof}

\begin{proof}[Proof of \cref{lem:K-negative extremal face}] 
After \cref{prop:K negative}, it remains to show that the face in question is extremal: the argument given in \cite[Section 3.6]{O'GdmsK3I} carries over verbatim, using Propositions \ref{prop:Mori cone JZ}, \ref{prop:Mori cone EZ} and \ref{prop:Mori cone IZ1Z2}. 
\end{proof}

\subsection{Proof of \cref{prop:contraction of I hat}}\label{sec:contraction of I hat}
We start with the proof of items (1) and (2), which goes along the same lines as \cite[Section 3.7]{O'GdmsK3I}. By \cref{cor:Delta-Sigma; Delta-Sigma-Omega}(1), the divisor $\widehat{\mathbb{E}}$ of $\widehat{\mathbb{I}}$ is a $\PP^1$-bundle over $\PP (\Gca)$ and
\[
\Oca(\widehat{\mathbb{E}})|_{\PP^1} \simeq \Oca_{\PP^1}(-1).
\] 
By the Fujiki-Nakano criterion, $\widehat{\mathbb{I}}$ can be contracted along the fibration $\widehat{\mathbb{E}} \to \PP(\Gca)$ to a \emph{complex manifold} $\overline{\mathbb{I}}'$, and $\widehat{\mathbb{I}}$ is the blow-up of $\overline{\mathbb{I}}'$ along $\PP(\Gca)$. The following claim shows that $g$ coincides with this contraction; in particular, $\overline{\mathbb{I}} = \overline{\mathbb{I}}'$ is smooth, and $\widehat{\mathbb{I}}$ is the blow-up of $\overline{\mathbb{I}}$ along $\PP(\Gca)$.

\begin{claim}\label{claim:equality of contractions hat}
Let $\mathrm{\Gamma} \subset \widehat{\mathbb{I}}$ be an integral curve. Then $[\mathrm{\Gamma}] \in \mathbb{R}^+\widehat{\sigma}$ if and only if $\mathrm{\Gamma}$ lies in a fibre of $\widehat{\mathbb{E}} \to \PP(\Gca)$.
\end{claim}
\begin{proof}
Assume that $\mathrm{\Gamma}$ is contained in a fibre of $\widehat{\mathbb{E}} \to \PP(\Gca)$. In particular, $\mathrm{\Gamma} \subset \widehat{\mathbb{E}}_Z$ for some $[Z] \in X^{[n]}$. The digression \ref{sec:digression on E_Z hat} yields $[\mathrm{\Gamma}] \in \mathbb{R}^+\widehat{\sigma}'_Z$ in $\mathrm{N}_1(\widehat{\mathbb{E}}_Z)$, and also $[\mathrm{\Gamma}] \in \mathbb{R}^+\widehat{\sigma}$ in $\mathrm{N}_1(\widehat{\mathbb{I}})$. 

Conversely, suppose that $[\mathrm{\Gamma}] \in \mathbb{R}^+\widehat{\sigma}$. By \cref{intersection matrix}, $\mathrm{\Gamma}  \cdot c_1(\Oca(\widehat{\mathbb{E}})) < 0$, so that $\mathrm{\Gamma} \subset \widehat{\mathbb{E}}$; moreover, $\mathrm{\Gamma} \subset \widehat{\mathbb{E}}_Z$ for some $[Z] \in X^{[n]}$, because $\widehat{\sigma}$ is sent to 0 by the pushforward along $\widehat{\mathbb{I}} \to X^{[n]} \times X^{[n]}$; in particular, $[\mathrm{\Gamma}] \in \mathbb{R}^+\widehat{\sigma}_Z'$, so that $\mathrm{\Gamma}$ is contained in a fibre of $\widehat{\mathbb{E}}_Z \to \PP (\Gca)_Z$. 
\end{proof}

Now, as $g$ is obtained by Mori theory, it is proper, and $g_*\Oca_{\widehat{\mathbb{I}}} \simeq \Oca_{\overline{\mathbb{I}}}$. Since the morphism $\widehat{\mathbb{I}} \to X^{[n]} \times X^{[n]}$ is constant on the fibres of $g$, it factors through a morphism $\overline{\mathbb{I}} \to X^{[n]} \times X^{[n]}$ by \cite[II, Lemme 8.11.1]{EGA}.

\smallskip

Let us prove (3). If $\widehat{\mathbb{I}}$ were acted on by $\mathrm{G}$ a \emph{connected} algebraic group, so would be $\overline{\mathbb{I}}$ by Blanchard's lemma \cite[Proposition 4.2.1]{BSUlsagga}. Now, $\bm{\mu}_2$ is not connected, but the proof of the lemma carries over: indeed, $\mathrm{G}$ is required to be connected just to show that it permutes the fibres of $g$; we check that $\bm{\mu}_2$ does the same, by looking at its explicit action on $\widehat{\mathbb{I}}$.

Let us show that the quotient of  $\overline{\mathbb{I}}$ by the induced $\bm{\mu}_2$-action is isomorphic to  $\overline{\mathrm{\Sigma}}$. Since the composition $\widehat{\mathbb{I}} \to \widehat{\mathrm{\Sigma}} \to \overline{\mathrm{\Sigma}}$ is constant on the fibres of $g$, we get a morphism $\overline{\mathbb{I}} \to \overline{\mathrm{\Sigma}}$, which is easily proved to be $\bm{\mu}_2$-invariant. To show that the induced bijective morphism $\overline{\mathbb{I}}/ \bm{\mu}_2 \to \overline{\mathrm{\Sigma}}$ is an isomorphism, we use Zariski's main theorem. Let us verify its assumptions. As $\overline{\mathbb{I}}$ is connected, so is $\overline{\mathbb{I}}/ \bm{\mu}_2$; moreover, $\overline{\mathrm{\Sigma}}$ is normal: indeed, it is smooth off the codimension 2 subset $\overline{\mathrm{\Delta}} = \PP(\Gca)$, and it is Cohen-Macaulay (it is a divisor on the smooth $\overline{M}_n$). 

\smallskip

Let us prove (4). Propositions \ref{proposition: Omega hat-Sigma hat} and \ref{prop:Omega hat-Omega-hat} describe $\widehat{\mathbb{J}}$ as the blow-up of $\PP(\Bca^\vee) \times _{\Grass^\omega(3,\Tca)} \PP(\Bca^\vee)$ along its diagonal. On the fibres of the blow-up morphism, $g$ is constant. We deduce a morphism
\[
\PP(\Bca^\vee) \times _{\Grass^\omega(3,\Tca)} \PP(\Bca^\vee) \to \overline{\mathbb{J}}=g(\widehat{\mathbb{J}}),
\]
which is bijective. In fact, it is an isomorphism, because $\overline{\mathbb{J}}$ is normal: indeed, $\overline{\mathbb{J}}$ is Cohen-Macaulay (it is a divisor on the smooth $\overline{\mathbb{I}}$), and it is smooth in codimension 1.

The isomorphism $\overline{\mathbb{J}} / \bm{\mu}_2 \simeq \overline{\mathrm{\Omega}} \cap \overline{\mathrm{\Sigma}}$ results then from \cref{prop:loci in M bar}(2).

\subsection{Proof of \cref{prop:contraction of I bar}} 
Basically, we argue as in \cref{sec:contraction of I hat}. However, as the morphism 
\[
\phi \colon \overline{\mathbb{J}} \to \Grass^\omega(3,\Tca)
\]
resulting from \cref{prop:contraction of I hat}(4) is a $\PP^2 \times \PP^2$-bundle instead of a projective bundle, the Fujiki-Nakano criterion does not apply, and the whole proof gets more involved. Let us begin with a preliminary statement. 
\begin{claim}\label{claim:equality of contractions bar}
An integral curve in $\overline{\mathbb{I}}$ is contracted by $f$ if and only if it lies in a fibre of $\phi$. In particular, we get a bijective morphism $\Grass^\omega(3,\Tca) \to f(\overline{\mathbb{J}})$.
\end{claim}
This is proved as \cref{claim:equality of contractions hat}, using two facts. First, for $[Z] \in X^{[n]}$, the injectivity of the map $\mathrm{N}_1(\overline{\mathbb{J}}_Z) \to \mathrm{N}_1(\overline{\mathbb{I}})$ and its behaviour on $\overline{\Mori}_1(\overline{\mathbb{J}}_Z)$, which are left to the reader. Second, the isomorphism
\begin{equation}\label{normal bundle J bar in I bar}
\Oca(\overline{\mathbb{J}})|_{\overline{\mathbb{J}}} \simeq \Oca_\phi(-1,-1) \otimes \phi^*\Oca_\gamma(-1),
\end{equation}
which results from \cref{rem:normal bundle of S in I bar}: since $\mathbb{I}$ (resp.\ $\mathbb{J}$) is isomorphic to $\overline{\mathbb{I}}$ (resp.\ $\overline{\mathbb{J}}$) in codimension one, $\Nca_{\overline{\mathbb{J}}/\overline{\mathbb{I}}}$ corresponds to $\Nca_{\mathbb{J}/\mathbb{I}}$ under the isomorphism $\mathrm{Pic}(\overline{\mathbb{J}}) \simeq \mathrm{Pic}(\mathbb{J})$.

\smallskip
We can now show the three statements of \cref{prop:contraction of I bar}.
We start with (3). Thanks to \cref{claim:equality of contractions bar}, the proof of \cref{prop:contraction of I hat}(3) carries over verbatim.

\smallskip

Next, we prove (1). We only need to prove that the bijective morphism $\Grass^\omega(3,\Tca) \to f(\overline{\mathbb{J}})$ of \cref{claim:equality of contractions bar} is an isomorphism. We exhibit an inverse. Note that, by Propositions \ref{prop:contraction of I hat}(4) and \ref{prop:loci in M bar}(2), $\Grass^\omega(3,\Tca)$ is also the scheme-theoretic image of $\overline{\mathbb{J}}$ under the composition $\overline{\mathbb{I}} \to \overline{\mathrm{\Sigma}} \to \widetilde{\mathrm{\Sigma}}$, and so under $\overline{\mathbb{I}} \to \widetilde{\mathbb{I}} \to \widetilde{\mathrm{\Sigma}}$; we deduce a bijective morphism $f(\overline{\mathbb{J}}) \to \Grass^\omega(3,\Tca)$, which is an isomorphism by Zariski's main theorem and provides the desired inverse.

\smallskip

It remains to prove (2). The argument is pretty technical, and closely akin to the proof of \cite[Theorem 3]{Ispct}. We begin with a vanishing result. 
\begin{lemma}\label{lemma:vanishing}
 For any $q>0$ and any $k \ge 0$, $\Rder^q f_* (\Oca(-k \overline{\mathbb{J}}))=0$.
\end{lemma}
\begin{proof}
Let $y \in  \Grass^\omega(3,\Tca)$, and denote as $\Ica_y$ (resp.\ $\Jca_y$) the ideal sheaf of $f^{-1}(y) \simeq \PP^2 \times \PP^2$ in $\overline{\mathbb{I}}$ (resp.\ $\overline{\mathbb{J}}$). By the formal function theorem, it suffices to prove that
\[
\Ho^q(\PP^2 \times \PP^2, (\Oca_{\overline{\mathbb{I}}}/\Ica_y^m)(-k\overline{\mathbb{J}})) = 0 \ \ \text{for any} \ q>0, k \ge 0, m\ge 0.
\]
Since $\PP^2 \times \PP^2$, $\overline{\mathbb{J}}$ and $\overline{\mathbb{I}}$ are smooth, we have a short exact sequence of conormal sheaves
\[
\begin{tikzcd}
0 \ar[r] & \Oca(- \overline{\mathbb{J}})|_{\PP^2 \times \PP^2} \ar[r] & \Ica_y|_{\PP^2 \times \PP^2} \ar[r] & \Jca_y|_{\PP^2 \times \PP^2} \ar[r] & 0.
\end{tikzcd}
\]
The first term is $\Oca(1,1)$ by \eqref{normal bundle J bar in I bar}. The last is a free sheaf of rank $d \coloneqq \dim  \Grass^\omega(3,\Tca)$ on $\PP^2 \times \PP^2$.
The Künneth formula shows that $\mathrm{Ext}^1(\Oca^{\oplus d},\Oca(1,1)) \simeq \Ho^1(\PP^2 \times \PP^2,\Oca(1,1))^{\oplus d}$ vanishes, so that
\begin{equation}\label{splitting of conormal}
\Ica_y/\Ica_y^2 \simeq \Oca(1,1) \oplus \Oca^{\oplus d}.
\end{equation}
Moreover, as both $\PP^2 \times \PP^2$ and $\overline{\mathbb{I}}$ are smooth, for any $m \ge 0$ we have an isomorphism
\begin{equation}\label{equality of normal bundle and cone}
\Ica_y^m/\Ica_y^{m+1} \simeq \mathrm{S}^m(\Ica_y/\Ica_y^2).
\end{equation}
Now, for any $k,m \ge 0$, consider the short exact sequence
\[
\begin{tikzcd}
0 \ar[r] &  (\Ica_y^m/\Ica_y^{m+1})(-k \overline{\mathbb{J}}) \ar[r] & (\Oca_{\overline{\mathbb{I}}}/\Ica_y^{m+1})(-k \overline{\mathbb{J}}) \ar[r] & (\Oca_{\overline{\mathbb{I}}}/\Ica_y^m)(-k \overline{\mathbb{J}}) \ar[r] & 0.
\end{tikzcd}
\]
On the one hand, we deduce from \eqref{splitting of conormal} and \eqref{equality of normal bundle and cone} that $\Ho^q(\PP^2 \times \PP^2,(\Ica_y^m/\Ica_y^{m+1})(-k \overline{\mathbb{J}}))$ vanishes for any $q>0, k \ge 0, m\ge 0$.
On the other hand, since $(\Oca_{\overline{\mathbb{I}}}/\Ica_y)(-k \overline{\mathbb{J}}) \simeq \Oca(k,k)$  by \eqref{normal bundle J bar in I bar}, we have that $\Ho^q(\PP^2 \times \PP^2, (\Oca_{\overline{\mathbb{I}}}/\Ica_y)(-k \overline{\mathbb{J}})) \simeq 0 $  for any $q>0, k \ge 0$. By induction on $m$, we conclude that 
\[
\Ho^q(\PP^2 \times \PP^2, (\Oca_{\overline{\mathbb{I}}}/\Ica_y^m)(-k \overline{\mathbb{J}})) \simeq 0 \ \text{for any} \ q>0, k \ge 0, m\ge 0,
\]
as desired. Together with the previously known isomorphism $f_* \Oca_{\overline{\mathbb{I}}} \simeq \Oca_{\widetilde{\mathbb{I}}}$, this computation proves in particular that $\widetilde{\mathbb{I}}$ has rational singularities.
\end{proof}
We prove now that $\widetilde{\mathbb{I}}$ is normally flat along $Y \coloneqq \Grass^\omega(3,\Tca)$. Let $\Jca$ be the ideal sheaf of $Y$ in $\widetilde{\mathbb{I}}$. By \cref{lemma:vanishing}, for any $k \ge 0$ the short exact sequence
\[
\begin{tikzcd}
0 \ar[r] & \Oca(-(k+1)\overline{\mathbb{J}}) \ar[r] & \Oca(-k \overline{\mathbb{J}}) \ar[r] & \Oca(-k \overline{\mathbb{J}})|_{\overline{\mathbb{J}}} \ar[r] & 0
\end{tikzcd}
\]
remains exact after applying the functor $f_*$. For $k=0$, we deduce in particular that $f_*\Oca(-\overline{\mathbb{J}}) \simeq \Jca$, because $f_* \Oca_{\overline{\mathbb{I}}} \simeq \Oca_{\widetilde{\mathbb{I}}}$ and $f_*\Oca_{\overline{\mathbb{J}}} \simeq \Oca_Y$. For $k > 0$, we have the inclusions $\Jca^k \subset f_*(f^{-1}(\Jca^k)) = f_*((f^{-1}\Jca)^k) \subset f_*\Oca(-k \overline{\mathbb{J}})$. Thus, we obtain a morphism of graded $\Oca_Y$-algebras
\begin{equation*}
\mathrm{\Phi} \colon \bigoplus_{k \ge 0} \Jca^k/\Jca^{k+1} \to \bigoplus_{k \ge 0} f_*\Oca(-k \overline{\mathbb{J}})/f_*\Oca(-(k+1)\overline{\mathbb{J}}).
\end{equation*}
By \eqref{normal bundle J bar in I bar}, we have the isomorphism of $\Oca_Y$-modules
\[
f_*\Oca(-k \overline{\mathbb{J}})/f_*\Oca(-(k+1)\overline{\mathbb{J}}) \simeq f_*(\Oca(-k\overline{\mathbb{J}})|_{\overline{\mathbb{J}}}) \simeq (\mathrm{S}^k \Bca)^{\otimes 2} \otimes \Oca_\gamma(k).
\] 
We draw a couple of consequences. 
\begin{enumerate}
\item[(i)] We have an equality $f^{-1}\Jca = \Oca(- \overline{\mathbb{J}})$ (so that $f^{-1}(\Jca^k) = \Oca(-k \overline{\mathbb{J}})$, too). We only need to check the surjectivity of the adjunction morphism $f^*f_* (\Oca(-\overline{\mathbb{J}})) \to \Oca(-\overline{\mathbb{J}})$, and in fact of its restriction to $\overline{\mathbb{J}}$, which factors through the morphisms
\[
[f^*f_*(\Oca(-\overline{\mathbb{J}}))]|_{\overline{\mathbb{J}}} \to [f^*f_*(\Oca(-\overline{\mathbb{J}})|_{\overline{\mathbb{J}}})]|_{\overline{\mathbb{J}}} \ \ \ \text{and} \ \ \ [f^*f_*(\Oca(-\overline{\mathbb{J}})|_{\overline{\mathbb{J}}})]|_{\overline{\mathbb{J}}} \to \Oca(-\overline{\mathbb{J}})|_{\overline{\mathbb{J}}}.
\] 
Both are surjective: the former by the previous discussion and the right exactness of $f^*(-)|_{\overline{\mathbb{J}}}$, the latter because it is identified with $\phi^*(\Bca^{\otimes 2} \otimes \Oca_\gamma(1)) \to \Oca_\phi(1,1) \otimes \phi^*\Oca_\gamma(1)$. 
\item[(ii)] The homomorphism $\mathrm{\Phi}$ is surjective: indeed, in degree 1 it surjects onto 
\[
f_*\Oca(-\overline{\mathbb{J}})/f_*\Oca(-2\overline{\mathbb{J}}) \simeq \Bca^{\otimes 2} \otimes \Oca_\gamma(1),
\]
which locally generates the sheaf of algebras $\oplus_k [(\mathrm{S}^k \Bca)^{\otimes 2} \otimes \Oca_\gamma(k)]$.
\end{enumerate}
We show now that $\mathrm{\Phi}$ is also injective. By \cite[III, Proposition 3.3.1]{EGA}, we have
\[
f_*\Oca(-(i+k)\overline{\mathbb{J}}) \simeq \Jca^k f_*\Oca(-i \overline{\mathbb{J}}) \ \ \ \text{for any} \ k \ge 0 \  \text{and all sufficiently large} \ i>0.
\]
As $f_*\Oca(-i \overline{\mathbb{J}}) \subset f_*\Oca(-\overline{\mathbb{J}}) = \Jca$, we have in particular $f_*\Oca(-(i+k)\overline{\mathbb{J}}) \subset \Jca^{k+1}$, and then
\[
\ker \left(\Jca^k \to  f_*\Oca(-k\overline{\mathbb{J}})/f_*\Oca(-(i+k)\overline{\mathbb{J}})\right) \subset \Jca^{k+1} \quad  \text{for any} \ k \ge 0 \  \text{and all sufficiently large} \ i>0.
\]
By using descending induction on $i$ and the surjectivity of $\mathrm{\Phi}$, we deduce that this last containment holds for $i=1$. This proves that $\mathrm{\Phi}$ is injective, and thus an isomorphism. In particular, for any $k \ge 0$, $\Jca^k/\Jca^{k+1}$ is isomorphic to the locally free $\Oca_Y$-module $(\mathrm{S}^k \Bca)^{\otimes 2} \otimes \Oca_\gamma(k)$, i.e.\ $\widetilde{\mathbb{I}}$ is normally flat along $Y$.

\smallskip

We are left to prove that $\overline{\mathbb{I}} \simeq \Bl_Y \widetilde{\mathbb{I}}$.
Since $f^{-1}\Jca= \Oca(-\overline{\mathbb{J}})$, by the universal property of the blow-up we get a morphism $\overline{\mathbb{I}} \to \Bl_Y\widetilde{\mathbb{I}}$ over $\widetilde{\mathbb{I}}$. To prove that it is an isomorphism, we use Zariski's main theorem: this morphism is bijective, because it is an isomorphism off $\overline{\mathbb{J}}$, and restricts on $\overline{\mathbb{J}}$ to an isomorphism
\[
\overline{\mathbb{J}} \simeq \PP (\Bca^\vee) \times _Y \PP (\Bca^\vee) \xrightarrow{\sim} \mathbf{Proj}(\bigoplus_{k \ge 0}\Jca^k/\Jca^{k+1});
\]
moreover, $\Bl_Y \widetilde{\mathbb{I}}$ is smooth, since so are its exceptional Cartier divisor, and $\widetilde{\mathbb{I}}$ off $Y$.

As a consequence, $\widetilde{\mathbb{I}}$ is singular at any point $y \in Y$: otherwise, the fibre of $f$ over $y$ would be a projective space, and not $\PP^2 \times \PP^2$.

\section{Appendix}
The following technical result was needed in the proof of \cref{lem:ses of differentials for blow up}.
\begin{lemma}\label{lemma:relative differentials blow up}
Let $Y$ be the closed subscheme of the affine space $\AA^n$ defined by a homogeneous ideal $J\subsetneq \mathbb{C}[x_1, \dots, x_n]$. Let $\sigma \colon \widetilde{Y} \to Y$ be the blow-up of $Y$ at the origin, and $i \colon \widetilde{Z} \to \widetilde{Y}$ the imbedding of the exceptional divisor.
Then the restriction map $\Kahler^1_{\widetilde{Y}/Y} \to i_* \Kahler^1_{\widetilde{Z}}$ is an isomorphism.
\end{lemma}
\begin{proof}
We need to show that the composition $\Kahler^1_{\widetilde{Y}} \to \Kahler^1_{\widetilde{Y}/Y} \to i_* \Kahler^1_{\widetilde{Z}}$ is the cokernel of the natural homomorphism $\sigma^* \Kahler^1_Y \to \Kahler^1_{\widetilde{Y}}$.

Let $s \colon \widetilde{\AA^n} \to \AA^n$ be the blow-up of $\AA^n$ at the origin. Using coordinates $x_1, \dots, x_n$ for $\AA^n$ and homogeneous coordinates $y_1, \dots, y_n$ for $\PP^{n-1}$, $\widetilde{\AA^n}$ can be seen as the closed subvariety of $\AA^n \times \PP^{n-1}$ cut by the equations $x_iy_j=x_jy_i$, for $1 \le i,j \le n$. The blow-up morphism $s$ is the restriction to $\widetilde{\AA^n}$ of the projection onto the first factor.

Let $U_1$ be the open subset of $\widetilde{\AA^n}$ where $y_1$ does not vanish. It is isomorphic to an affine space, with coordinates $x_1$ and $u_j = y_j/y_1$, for $2 \le j \le n$; moreover, the restriction of $s$ to $U_1$ is given by $(x_1, u_2, \dots u_n) \mapsto ( x_1, u_2x_1, \dots, u_nx_1)$.

Let us now focus on $Y$. Let $(f_k)_k$ be a finite family of homogeneous generators of $J$; we denote by $g_k$ the dehomogeneisation of $f_k$ with respect to $x_1$. Since the blow-up $\widetilde{Y}$ identifies with the strict transform of $Y$ under $s$, on $U_1$ we have
\[
V_1 \coloneqq \widetilde{Y} \cap U_1 = \{ (x_1, u_2, \dots, u_n) \in U_1 \mid g_k(u_2, \dots, u_n)=0 \ \text{for every} \  k \}.
\]
In the coordinates of $U_1$, the exceptional divisor $\widetilde{Z}$ is described as
\[
\widetilde{Z} \cap V_1 = \{ (x_1, u_2, \dots, u_n) \in U_1 \mid  x_1=0, \ g_k(u_2, \dots, u_n)=0 \ \text{for every} \  k  \}.
\]

We compute now the cokernel of $\sigma^* \Kahler^1_Y \to \Kahler^1_{\widetilde{Y}}$ on $V_1$. Since $\Kahler^1_{\AA^n}|_Y \to \Kahler^1_Y$ is surjective, it is enough to compute the cokernel of the $\mathbb{C}[V_1]$-linear map
\[
\xi \colon \bigoplus_{l=1}^n \mathbb{C}[V_1] \mathrm{d}x_l\to \left( \mathbb{C}[V_1] \mathrm{d}x_1 \oplus \bigoplus_{l=2}^n \mathbb{C}[V_1] \mathrm{d}u_l \right) / \sum_k \mathbb{C}[V_1]\mathrm{d}g_k
\]
that sends $\mathrm{d}x_1$ to $\mathrm{d}x_1$, and $\mathrm{d}x_l$ to $x_1 \mathrm{d}u_l + u_l \mathrm{d}x_1$  for $l \ge 2$. The image of $\xi$ is generated as a $\mathbb{C}[V_1]$-module by $\mathrm{d}x_1$ and $x_1 \mathrm{d}u_l$. Thus, $\coker(\xi)$ is given by
\[
\frac{\mathbb{C}[V_1] \mathrm{d}x_1 \oplus \bigoplus_{l=2}^n \mathbb{C}[V_1] \mathrm{d}u_l}{(\mathbb{C}[V_1] \mathrm{d}x_1 + \sum_l x_1 \mathbb{C}[V_1] \mathrm{d}u_l) + \sum_k \mathbb{C}[V_1]\mathrm{d}g_k} \simeq \frac{\bigoplus_{l=2}^n \mathbb{C}[u_2, \dots,u_n] \mathrm{d}u_l}{\sum_k \mathbb{C}[u_2, \dots,u_n]\mathrm{d}g_k}.
\]
The $\mathbb{C}[V_1]$-module on the right-hand side is precisely $\Ho^0(V_1,i_*\Kahler^1_{\widetilde{Z}})$. To conclude, it suffices to repeat the argument on the other affine charts of $\widetilde{Y}$.
\end{proof}

\bibliographystyle{alpha}
\bibliography{main}

\end{document}